\documentclass[11pt,a4paper,reqno]{amsart}
\usepackage[english]{babel}
\usepackage[T1]{fontenc}
\usepackage{verbatim}
\usepackage{palatino}
\usepackage{amsmath}
\usepackage{mathabx}
\usepackage{amssymb}
\usepackage{amsthm}
\usepackage{amsfonts}
\usepackage{graphicx}
\usepackage{esint}
\usepackage{color}
\usepackage{mathtools}

\usepackage[colorlinks = true, citecolor = black]{hyperref}

\theoremstyle{plain}

\newtheorem{theo}{Theorem}[section]
\newtheorem{lem}[theo]{Lemma}
\newtheorem{prop}[theo]{Proposition}
\newtheorem{coroll}[theo]{Corollary}

\theoremstyle{definition}
\newtheorem{defi}[theo]{Definition}

\theoremstyle{remark}
\newtheorem{rem}[equation]{Remark}
\newtheorem{exa}{Example}[section]

\newcommand{\R}{\mathbb{R}}
\newcommand{\QQ}{\mathbb{Q}}
\newcommand{\N}{\mathbb{N}}
\newcommand{\M}{\mathbb{M}}

\newcommand{\W}{\mathbb{W}}
\newcommand{\V}{\mathbb{V}}

\newcommand{\E}{\mathcal{E}}
\newcommand{\HH}{\mathbb{H}}
\newcommand{\He}{\mathbb{H}}
\newcommand{\B}{\mathcal{B}}
\newcommand{\G}{\mathbb{G}}

\newcommand{\mcal}{\mathcal}
\newcommand{\tl}{\tilde}

\newcommand{\graph}[1]{\mathrm{graph}\,(#1)}

\newcommand{\average}{{\mathchoice {\kern1ex\vcenter{\hrule height.4pt
width 6pt
depth0pt} \kern-9.7pt} {\kern1ex\vcenter{\hrule height.4pt width 4.3pt
depth0pt}
\kern-7pt} {} {} }}

\newcommand{\res}{\mathop{\hbox{\vrule height 7pt width .5pt depth 0pt
\vrule height .5pt width 6pt depth 0pt}}\nolimits}

\title{Intrinsic Lipschitz maps vs. Lagrangian type solutions in Carnot groups of step 2}
\author[Daniela Di Donato]{Daniela Di Donato}
\address{Daniela Di Donato: SISSA\\ Mathematics Area\\ Via Bonomea, 265\\ 34136, Trieste - Italy\\}  

\email{ddidonat@sissa.it}
    \thanks{D.D.D. is supported by SISSA, Italy.}


\date{\today}
\begin{document}

\maketitle

\begin{abstract}
We focus our attention on the notion of intrinsic Lipschitz graphs, inside a subclass of Carnot groups of step 2 which includes corank 1 Carnot groups (and so the Heisenberg groups), Free groups of step 2 and the complexified Heisenberg group. More precisely, we prove the equivalence between an intrinsic Lipschitz map and a suitable notion of weak solution of a Burgers' type PDE, which generalizes the Lagrangian solution in the context of Heisenberg groups.
\end{abstract}

\maketitle

\section{Introduction}

%CAMBIARESubRiemannian geometry has been a full research domain from the 80's, with motivations and ramifications in several parts of pure and applied mathematics. 
SubRiemannian geometry is a generalization of Riemannian one but they are significantly different from each other.  A SubRiemannian manifold is  defined as a manifold $M$ of dimension $n$ joint with a distribution $\Delta$ of $m$-planes (with $m < n$) which satisfies the known H\"ormander condition and a Riemannian metric on $\Delta$. One can define a distance between two points of $M$ as the infimum of the lengths of absolutely continuous paths that are tangent to $\Delta$ and link these two points. Here the length of a path is defined via usual way using the fact that the metric considered is Riemannian. 

In general, SubRiemannian distances are not Euclidean at any scale, and hence not Riemannian.  %Indeed, there are no (even local) biLipschitz maps from a general non commutative Carnot group to Euclidean spaces.  %, and the distance of two points of $M$ is in turn defined as %If no such path exists, one sets $d(P,Q)=+\infty$. 
Consequently, in the context of SubRiemannian the proofs  often require new techniques. 

%The paper \cite{NY} is an excellent reference about the connected to other areas of mathematics like computer science; \cite{BLU} like subelliptic Differential Equations; \cite{CDPT} like Differential Geometry;  \cite{MR1865002, MR3682744} like Calculus of Variations.
%SR geometry is connected to other areas of mathematics like subelliptic Differential Equations \cite{BLU}; like Differential Geometry \cite{CDPT}; like Analysis of Partial Differential Equations \cite{MR2223801}; like Calculus of Variations \cite{MR1865002, MR3682744}.
%The paper \cite{NY} is an excellent reference about the connected to other areas of mathematics  like computer science \cite{NY}

We focus our attention on particular SubRiemannian groups called Carnot groups. 
%Following \cite{BLU, LeDonne17, SC16}, 
Le Donne establishes  in \cite{biblioLeDonne} that Carnot groups $\G$ are the only metric spaces that are: 
\begin{description}
\item[ a)] locally compact;
\item[ b)]  geodesic;
\item[ c)]  isometrically homogeneous;
\item[ d)] self-similar.
\end{description}
Here %$a)$ means that every point of the space has a compact neighborhood; $b)$, i.e. for any couple of two points there is an isometric embedding from an interval of $\R$ to $\G$ which connects it; 
 $c)$ means for any couple of two points there exists a distance-preserving homeomorphism on $\G$; $d)$ i.e., there exists $\lambda >1$ and a homeomorphism $f$ on $\G$ such that the distance $d$ on $\G$ satisfies the following equality $d(f(p),f(q)) = \lambda d(p,q)$, for all $p,q \in \G$.

It is useful to know that the  Lie algebra $\mathfrak g$ associated to a Carnot group $\G$ is such that
\[
\mathfrak g=V_1\oplus\dots\oplus V_\kappa,\quad [V_j,V_1]=V_{j+1},\quad \text{for } j=1,\dots, \kappa-1,\quad [V_\kappa,V_1]=\{0\},
\] 
where $V_1,\dots, V_\kappa$ are complementary linear subspaces and $[V_j,V_1]$ denotes the subspace of ${\mathfrak{g}}$ generated by
the commutators $[X,Y]$ with $X\in V_j$ and $Y\in V_1$. The integer $\kappa$ is called step of the group $\G$, while $\dim (V_1)$ is called rank of $\G$. 

Euclidean spaces are commutative Carnot groups of step $1$ and are the only commutative ones. This paper is dedicated to a suitable subclass of Carnot groups of step 2, (i.e., $\kappa=2$) which includes  the Heisenberg groups $\He ^n$, corank 1 Carnot groups  \cite{biblio3}, Free groups of step 2 \cite{biblio3} and the complexified Heisenberg group \cite{biblioRR}.

The core of this paper is to study the notion of  intrinsic Lipschitz graphs. This concept is important to develop a satisfactory theory of intrinsic rectifiable sets which is an active line of research \cite{biblioAM3, biblioAKLD, biblioDLDMV, biblioOMM}. %on intrinsic rectifiability, but it also studies connections to different notions like Lipschitz image rectifiability and foliated corona decompositions. 

%  The long-term objective of our  research is to establish \textbf{a good notion of rectifiability in Carnot groups}.  Hence the notion of  intrinsic Lipschitz graphs, which is the core of this paper, is important to develop a satisfactory theory of intrinsic rectifiable sets. 
  
   Rectifiability, introduced by Besicovich in the plane, is a key notion in Geometric Measure Theory. The classical definition was given by Federer in \cite{biblio5}: in Euclidean spaces, rectifiable sets are defined as being essentially contained in the countable union of $C^1$ submanifolds or of Lipschitz graphs. The equivalence of these two notions follows from well-known theorems: Rademacher Theorem; Extension of Lipschitz maps; Whitney's Extension Theorem.   

Regarding Carnot groups, different notions of rectifiability have been proposed in the literature:
\begin{enumerate}
\item Rectifiability using images of Lipschitz maps defined on subsets of $\R^d$;
\item  Lipschitz image  rectifiability, using homogeneous subgroups;
\item  Intrinsic Lipschitz graphs  rectifiability;
\item  Rectifiability using intrinsic $C^1$ surfaces.
\end{enumerate}

The first approach (1) is a general metric space approach, given by Federer in \cite{biblio5}. He states that a $d$-dimensional rectifiable set in a Carnot group $\G$ is essentially covered by the images of Lipschitz maps from $\R^d$ to a Carnot group $\G$. Unfortunately, this definition is too restrictive because often there are only rectifiable sets of measure zero (see \cite{biblioAMBkirc, biblioMAGN1}).%one gets only "low dimensional" rectifiable sets (i.e., $d<n+1$, where $n$ is the topological dimension of $\G$). Indeed, $\R^d$ might be an homogeneous subgroup of $\G$ only for $d<n+1$.

Another metric space approach but more fruitful than $(1)$ in the setting of groups is given by Pauls \cite{biblioPAUl} (see (2)). It is called Lipschitz image (LI) rectifiability. The author  considers images in $\G$ of Lipschitz maps defined not on $\R^d$ but on subset of homogeneous subgroups of $\G.$

 Intrinsic Lipschitz graphs (iLG) rectifiability $(3)$ and the notion of intrinsic $C^1$ surfaces $(4)$ were both introduced by Franchi, Serapioni, Serra Cassano \cite{biblio6, biblio8}.  %Indeed, Serapioni was my PhD advisor and Serra Cassano was PhD coordinator. 
 The concept $(3)$ is studied with different degrees of generality in \cite{biblioAM2, biblio2, biblio21, biblio22, biblio26, bibliofsscINTRINSICLIP, biblioDDF, biblioVIT2}. The simple idea of intrinsic graph is the following one: let $\mathbb{V}$ and $\W$ be complementary homogeneous subgroups of $\G$, i.e., $\W \cap \mathbb{V}= \{ 0 \}$ and $\G=\W\cdot \mathbb{V}$, then the intrinsic left graph of $\phi :\W\to \mathbb{V}$ is the set $$ \mbox{graph} {(\phi)}:=\{ x\cdot \phi (x) \, |\, x\in \W \}.$$
%Hence the existence of intrinsic graphs depends on the possibility of splitting $\G$ as a product of complementary subgroups hence it depends on the structure of the Lie algebra associated to $\G$.
 A function $\phi $ is said to be intrinsic Lipschitz if it is possible to put, at each point $p\in \mbox{graph} {(\phi)}$, an intrinsic cone with vertex $p$, axis $\mathbb{V}$ and fixed opening, intersecting $\mbox{graph} {(\phi)}$ only at $p$. %This notion have been introduced in \cite{biblio24}.
%We call a set $S\subset \G$ an intrinsic Lipschitz graph if there exists an intrinsic Lipschitz function $\phi : \mathcal O \subset \W \to \M$ such that $S= \mbox{graph } {\phi}$ for suitable complementary subgroups $\W$ and $\mathbb{V}$ of $\G$.

% On the other hand, intrinsic $C^1$ regular surfaces $(4)$ were firstly introduced in  Heisenberg groups and then in \cite{Mag13} in arbitrary Carnot groups.  

Moreover, in \cite{biblio6, biblio7}, Franchi, Serapioni, Serra Cassano introduce the notion $(4)$ adapting to groups De Giorgi's classical technique valid in Euclidean spaces to show that the boundary of a finite perimeter set can be seen as a countable union of $C^1$ regular surfaces. A set $S$ is a $d$-codimensional intrinsic $C^1$ surface $(4)$  if there exists a continuous function $f:\G \to \R^d$ such that,  locally, $$S=\{ p\in \G : f(p)=0\},$$ and the horizontal jacobian of $f$ has maximum rank, locally.
  
The approaches $(2)$ and $(3)$ are natural counterparts of the notions of rectifiability in Euclidean spaces, where their equivalence is trivial. Hence it is surprising that the connection between iLG and LI rectifiability is poorly understood already in Carnot groups of step 2. %, i.e., the Heisenberg groups. 

In \cite{biblioALD}, Antonelli and Le Donne prove that these two definitions are different in general; their example is for a Carnot group of step $3$.  
 The paper \cite{biblioDDFO} makes progress towards the implication iLGs are LI rectifiable in $\He^n$. We proved that $C^{1,\alpha } $-surfaces are LI rectifiable, where $C^{1,\alpha } $-surfaces are intrinsic $C^1$ ones  whose horizontal normal is $\alpha$-H\"older continuous. %All the proofs in \cite{DFO2020} are based on a new general criterion for finding biLipschitz maps between "big pieces" of metric spaces. 

 Differently from Euclidean case, in Carnot groups the notions of rectifiability given in $(3)$ and $(4)$ are in general not equivalent any more.  %AGGIUNGERE QUALCOSA CON IL TEOREMA DELLA FUNZIONE IMPLICITA CHE CI Dà UN'IMPLICAZIONE  o forse non serve (?!) 
 The problem is that we don't have a suitable Whitney's Extension Theorem for maps from a closed subset of $\R^k$ to $\G$; vector valued extension theorems; vector valued Rademacher's type theorems that, in the intrinsic context, is equivalent to saying that intrinsic Lipschitz maps are a.e. differentiable in a suitable sense. Recently, in \cite{biblioVIT1}, Vittone gives a positive answer about Rademacher's type theorems in $\He^n$. Another positive answer is given by Franchi, Marchi, Serapioni \cite{biblio21} for a large class of Carnot groups which includes step 2 Carnot groups in codimension one (i.e., for a map $\phi:\V \to \W$ with $\V$ 1-dimensional). We also recall \cite{biblioJNGV}, where the authors give counterexamples to a Rademacher theorem in codimension 2 (when $\G \ne \HH^n$). % on the other hand, Whitney's Extension Theorem for maps $\R \to \G$ has been largely developed by  \cite{JS17, PSZ19, SS18}. %CITARE VITTONE \cite{Vittone20} E PINAMONTI E CO. K=1 \cite{PSZ19}
%A possible idea would be to study the problem in free Carnot groups $\mathbb{F}$ of step $2$, firstly, and then extend the results in any Carnot groups $\G$ of step $2$ using the fact that $\G$ can be written as a quotient of $\mathbb{F}$. This resulting strategy resembles the one we used successfully in \cite{ADDD20, ADDDLD20}.

%The results obtained in the Carnot groups of step 2 are more general than those of $\He^n$. The main difference between these two groups is that in $\He^n$ there is only one vertical (i.e. non-horizontal) coordinate, whereas for Carnot groups of step $2$ there can be many.  Examples of Carnot groups of step $2$ are the free step 2 Carnot groups; the H-type groups; the H-groups in the Sense of M\'etivier; the complexifed Heisenberg group. A good reference about them is \cite{BLU}. % which not include $\He^n$ are free Carnot groups of step 2, for $n\ne1.$ A wide range of Carnot groups of step 2 are the free step 2 Carnot groups; the H-type groups; the H-groups in the Sense of Metivier; the complexifed Heisenberg group

 %There are a lots of application of these results in Euclidean spaces,
%In  \cite{biblioDDFO},  following the approach of Whitney's Extension Theorem for real-valued $C^1$ maps on $\He ^n$ proved by Franchi, Serapioni and Serra Cassano \cite{FSSC}, we show that if we consider a $C^{1,\alpha } $-surface with $0<\alpha \leq 1$, then it is, locally, contained on iLG of a $C^{1,\frac \alpha 3} $ map, i.e., an intrinsic $C^1$ map with $\frac \alpha 3$-H\"older horizontal normal. 

In the context of $\He^n$, a characterization of intrinsic $C^1$ surfaces has been studied in \cite{biblio1}.  %AGGIUSTARE IL TUTTO and generalized in Carnot groups of step $2$ in \cite{ADDD20, DD19, DD20}. Recently, we generalize these results in any Carnot groups in \cite{ADDDLD20}. In these papers, the authors characterize intrinsic $C^1$ surfaces in terms of weak solution of a suitable PDE system. 
Namely, let a continuous map $\phi : \W \to \V$ be defined between two complementary subgroups of $\He^n$ where $\V$ is 1-dimensional, then the following conditions are equivalent: 
\begin{description}
\item[ a)] graph$(\phi)$ is, locally, an intrinsic $C^1$ surface;
\item[ b)] $\phi$ is a suitable weak solution of PDE system 
\begin{equation}\label{fottUTA}
D^\phi \phi =w, 
\end{equation}
where $w$ is a continuous map.
\end{description}

In $\He^1$, for $ \mathbb W=\{(0,x_2,x_3)\,:\, x_2,x_3 \in \R \}, \mathbb V=\{(x_1,0,0) \,:\, x_1\in \R\} \subset \He^ 1$ and  $ \phi(0,x_2,x_3):= \left(\phi_1(x_2,x_3),0,0 \right): \mathbb W \to \mathbb V$, we have that
 \begin{equation*}
D^\phi  :=\frac {\partial } {\partial x_2} + \phi_1\frac {\partial}{\partial x_3},
\end{equation*}
i.e., $D^\phi$ is Burgers' operator which is a non linear first order one studied in various areas of applied mathematics, such as fluid mechanics, nonlinear acoustics, gas dynamics and traffic flow. 

In 2015, in \cite{biblioKOZHEVNIKOV}, there is a general definition of this operator introduced in \cite{biblio1} which is the correct intrinsic replacement of Euclidean gradient for $C^1$ surfaces. Precisely for this reason, it is called intrinsic gradient of $\phi$. Specifically, in a Carnot group of any step, $D_j^\phi \phi$ is the projection  on $\W$ of a horizontal vector field of $\G$ on the points of the intrinsic graph of $\phi$, where $j$ is an integer smaller than the rank of $\G$.

Starting from \cite{biblio1}, the study of PDE in the context of intrinsic $C^1$ surfaces and then of intrinsic Lipschitz maps has been largely developed in $\He^n$ \cite{biblioCorni, biblio27} and in Carnot groups of step $2$ \cite{biblioADDD, biblioDDD, biblioDDD2}. Recently, in \cite{biblioADDDLD}, we generalize these results to any Carnot groups and in low codimensional case, i.e. when $\V$ has dimension not larger than the rank of $\G$.  Here the datum $w$ in  \eqref{fottUTA} is a continuous map.

In \cite{biblio17}, the authors study the case when $w$ is a \emph{bounded measurable function} in the context of Heisenberg groups $\He^n.$ A natural question is if it is possible to generalize the notions of weak solution of  \eqref{fottUTA} given in \cite{biblio17} and, finally, to prove their equivalence with iLG in more general cases, i.e., Carnot groups of step $2$ with $\G=\W \cdot \V$ and $\V$ is 1-dimensional.  In \cite{biblioDDD2}, we give a partially positive answer about the equivalence between iLG  and a $1/2$-H\"older map $\phi : \W \to \V$ satisfies \eqref{fottUTA} in the distributional sense. Here the datum $w$ is just a measurable map. %Here $\V, \W $ are complementary subgroups and, in particular, $\V$ is 1-dimensional. %A question could be if the  H\"older condition is redundant. 

In this paper, we go another step further towards the understanding of this question. 

The main result (see Theorem \ref{teoremaFinale1}) states that in a suitable subclass  of Carnot groups of step 2, we have that
\begin{equation*}
\begin{aligned}
\phi \mbox{ is intrinsic Lipschitz} & \quad  \Longleftrightarrow  \quad \phi \mbox{ is a Lagrangian type  solution of } D^\phi \phi = w, \\
%\phi \mbox{ is Distributional solution of } \eqref{solLagrangiana} 
%& \quad \Longleftrightarrow \quad   \phi \mbox{ is Lagrangian type  solution of } D^\phi \phi = w,\\
\end{aligned}
\end{equation*} where $w$ is a fixed measurable map and Lagrangian type  solution is defined in Definition \ref{defiLagrangiana}. 

Firstly, in the context of Carnot groups of step 2, we introduce a suitable weak solution of \eqref{fottUTA}. We will call it Lagrangian type  solution because it generalizes Lagrangian solution in the context of Heisenberg groups. The idea of this definition is that the reduction on characteristics is not required on any characteristic, as happened in the broad* solution (see Definition 3.24 in \cite{biblioADDDLD}), but on a suitable set of characteristics. Finally, we present the main result of this paper, i.e., Theorem \ref{teoremaFinale1} in a suitable subclass of step 2 Carnot groups (see Setting \ref{Setting}). Here we show the link between locally iLGs and Lagrangian solutions of $\eqref{fottUTA}$ establishing their  equivalence with the distributional solutions of $\eqref{fottUTA}.$ We refine the technique used in \cite{biblio17} in the context of $\HH^n$  noting that the main difference between 2 step Carnot groups and Heisenberg groups is that in $\HH^n$ there is only one vertical (i.e. non-horizontal) coordinate, whereas for 2 step Carnot groups there can be many.

We underline that in Section \ref{Existence Lagrangian type parameterization}, the results are true in any Carnot groups of step 2. Here the strategy is to solve the problem for corank 1 Carnot groups following \cite{biblio17} and then for the general case we reduce to this one observing that the vertical components of the integral curve of $D^\phi$ can be written in combination with each other (see \eqref{infunzgamma}).
%In this section we give the following weaker statement: one can reduce the PDE to ODEs along a selected family of characteristics constituting a Lagrangian parameterization. As well, the converse holds: if the ODEs on characteristics are satisfied, one has a continuous distributional solution to the PDE. The sources of the two formulations can be identified when $\phi$ is locally $1/2 $-H\"older continuous along the vertical components, but the explicit proof is not contained in this section. 
\\

The paper is organized as follows. In Section 2 we introduce the basic notions on Carnot groups of step 2, $C^1_\G$ functions, $\G$-regular surfaces and intrinsic Lipschitz graphs. The definition, some properties and examples of the intrinsic gradient $D^\phi$ of a continuous map $\phi$ is the object of Section 3. In Section 4 we introduce and give some properties of so-called Lagrangian type solution. Finally, Theorems \ref{teoremaFinale1} is proved in Section 5 together with some preliminary results.

$\mathbf{Acknowledgements.}$ We wish to express our gratitude to Raul Serapioni for many invaluable discussions about the notion of intrinsic Lipschitz maps. We also thank Gioacchino Antonelli for the useful discussions on the topic.

\section{Notations and preliminary results}

\subsection{Carnot groups of step 2}\label{Carnotinizio}

We here introduce Carnot groups of step 2 and we refer the reader to \cite[Chapter~3]{biblio3}.   
We denote with $m$ the rank of $\G$ and we identify $\mathbb G$ with $(\R^{m+n}, \cdot )$. If $q\in \G$, we write $q=(x,y)$ meaning that $x\in \R^m$ and $y\in \R^n$. The group operation $\cdot$ between two elements $q=(x,y)$ and $q'=(x',y')$ is given by
\begin{equation}\label{5.1.0}
q\cdot q'= \left(x+ x',y+ y'-\frac 1 2\langle \mathcal{B} x,  x' \rangle \right),
\end{equation} 
 where $\langle \mathcal{B}x,x' \rangle := (\langle \mathcal{B}^{(1)}x,x' \rangle, \dots , \langle \mathcal{B}^{(n)} x, x' \rangle)$ and $\mathcal B^{(i)}$ are linearly independent and skew-symmetric matrices in $\R^{m\times m}$, for $i=1,\dots, n$. Moreover the dilation $\delta _{\lambda } : \R^{m+n} \to \R^{m+n}$ defined as 
\begin{equation*}\label{dilatazioneintro2}
\delta_\lambda (x,y)  := (\lambda x , \lambda^2 y), \quad \mbox{ for all } (x,y)\in \R ^{m+n},
\end{equation*}
is an automorphism of $(\R^{m+n}, \cdot )$, for all $\lambda >0$.

The identity of $\G$ is the origin of $\R^{m+n}$ and $(x,y)^{-1}=(-x,-y)$. For any $p\in \G$ the intrinsic left translation $\tau _p:\G \to \G $ are  defined as
\begin{equation*}
q \mapsto \tau _p q := p\cdot q=pq.
\end{equation*} 

A \emph{homogeneous norm} on $\G$ is a nonnegative function $p\mapsto \|p\|$ such that for all $p,q\in \G$ and for all $\lambda \geq 0$
\begin{equation*}\label{defihomogeneous norm}
\begin{split}
\|p\|=0\quad &\text{if and only if }  p=0\\
\|\delta _\lambda p\|= \lambda \|p\|,& \qquad 
\|p \cdot q\|\leq \|p\|+ \|q\|.
\end{split}
\end{equation*}
We make the following choice of the homogeneous norm in $\G$:
\begin{equation}\label{normadinfty}
\Vert(x,y)\Vert:= \max \{\vert x\vert _{\R^{m}}, \varepsilon \vert y\vert^{1/2}_{\R^{n}} \},
\end{equation}
for a suitable $\varepsilon \in (0,1]$ (for the existence of such an $\varepsilon >0$ see Theorem 5.1 in \cite{biblio8}). From now on, with a bit abuse of notation, we will write the norm of $\R^s$ for every $s\in \N$ with the same symbol $|\cdot|.$

%We recall also that there is $c_1>1$ such that for all $(x,y)\in \G$
%\begin{equation}\label{deps}
%c_1^{-1}\left(\vert x\vert_{\R^m}+\vert y\vert_{\R^n}^{1/2} \right)\leq\Vert (x,y)\Vert\leq c_1\left(\vert x\vert_{\R^m}+\vert y\vert_{\R^n}^{1/2} \right)
%\end{equation}

However, given any homogeneous norm $\|\cdot \|$, it is possible to introduce a distance in $\G$ given by
\[
d(p,q)=d(p^{-1} q,0)= \|p^{-1} q\|, \qquad \text{for all $p,q\in G$}.   
\]
%We observe that any  distance $d $ obtained in this way is always equivalent with the $cc$-distance $d_{cc}$ of the group. 

%We shall denote $U(P,r)$ and $B(P,r)$ respectively the open and closed balls associated with $d$. 

The metric $d$ is well behaved with respect to left translations and dilations, i.e. for all $p,q,q' \in \G$ and $\lambda >0$,
\begin{equation*}
\begin{aligned}
d (p\cdot q,p\cdot q')=d(q,q'), \qquad d (\delta_\lambda q,\delta_\lambda q')=\lambda d(q,q'),
\end{aligned}
\end{equation*}
Moreover, for any bounded subset $\Omega \subset \G$ there exist positive constants $c_1=c_1(\Omega),c_2=c_2(\Omega)$ such that for all $p,q \in \Omega$
\[
c_1|p-q| \leq d(p,q) \leq c_2 |p-q|^{1/2 }
\] 
and, in particular, the topology induced on $\G$ by $d$ is the Euclidean topology. For $p\in \G$ and $r > 0$,  $\mathcal U (p,r)$ will be the open ball associated with the distance $d$. %Intrinsic $t$-dimensional spherical Hausdorff measure $\mathcal{S}^{t}$ on $\G$, $t \geq 0$, is obtained from $d$, following Carath\'eodory construction (see for instance \cite{biblio16}).

The Hausdorff dimension of $(\G, d )$ as a metric space is  denoted \textit{homogeneous dimension} of $\G$ and it can be proved to be   the integer $\sum_{l=1,2}l$ dim$V_l  = m+2n > m+n$ (see \cite{biblioMitchell}).

For any $i=1,\dots, n$ and any $j,\ell=1,\dots,m$, denote by $(\mathcal{B}^{(i)})_{j\ell} = (b_{j\ell}^{(i)})$, and define $m+n$ linearly independent left-invariant vector fields by setting
\begin{equation}\label{vectorfields}
\begin{aligned}
X_j (p) & := \partial _{x_j}  -\frac{1}{2 } \sum_ {i=1 }^{n} \sum_ {\ell=1 }^{m} b_ {j\ell}^{(i)} x_\ell  \,\partial _{y_i},  \quad \mbox{ for } j=1,\dots ,m,\\
Y_i(p)  & := \partial _{y_i }, \,\qquad \qquad \qquad \qquad \qquad \mbox{ for } i=1,\dots , n.
\end{aligned}
\end{equation}  
 The ordered set $(X_1,\dots,X_m,Y_1,\dots,Y_n)$ is an adapted basis of the Lie algebra $\mathfrak g$ of $\G$. Using the skew-symmetry of $\mathcal B$, it easy to see that
\begin{equation}\label{commutatoripasso2}
[X_j, X_\ell]= \sum _{i=1}^n b_{j\ell}^{(i)} Y_i,  \quad \mbox{and} \quad  [X_j , Y_{i}] = 0,\quad \forall j,\ell=1,\dots, m \;\text{ and }\; \forall i=1,\dots,n.
\end{equation}

\begin{rem}
Note that the above arguments show that there exist $2$ step Carnot groups of any di\-men\-sion $m\in \N$ of the first layer and any dimension $$n\leq \frac{m(m-1)}{2}, $$ of the second layer: it suffices to choose $n$ linearly independent matrices $\B^{(1)} , \dots , \B^{(n)}$  in the vector space of the skew-symmetric $m\times m$ matrices (which has dimension $m(m-1)/2$) and then define the composition law as in \eqref{5.1.0}. 
\end{rem}

\begin{rem}\label{remChange} 
  If we denote $\mathcal{M}_1$ a non singular $m\times m$ matrix and $\mathcal{M}_2 $ a non singular $n\times n$ matrix, the linear change of coordinates associated to $\mathcal M _1$ and $\mathcal M _2$ is 
%the linear change of basis on $\R^{m+n}$, i.e. the linear map $\mathcal{M}: \R^{m+n} \to \R^{m+n}$ given by
\begin{equation*}
(x,y)\mapsto (\mathcal M _1 x, \mathcal M _2 y).
\end{equation*}
The new composition law $\star$ in $\R^{m+n}$,   obtained by writing $\cdot $ in the new coordinates, is
\begin{equation*}
(\mathcal M _1 x, \mathcal M _2 y) \star (\mathcal M _1 x', \mathcal M _2 y') := (\mathcal M_1 x +\mathcal M_1 x', y + y' + \frac{1}{2} \langle  \mathcal{ \tilde B}  x , x' \rangle),
\end{equation*}
where $\mathcal{\tilde B} := (\mathcal{\tilde B}^{(1)}, \dots , \mathcal{\tilde B}^{(n)})$ 
and if we put $\mathcal{M}_2=(c_{sk})_{s,k=1}^n$ then $$ \mathcal{ \tilde B} ^{(s)}= (\mathcal{M}_1^{-1})^T \left( \sum_{k=1}^n c_{sk} \mathcal{B}^{(k)} \right) \mathcal{M}_1^{-1} ,$$ for $s=1,\dots ,n.$
It is easy to check that the matrices $\mathcal{ \tilde B}^{(1)} ,\dots ,\mathcal{ \tilde B}^{(n)}$ are skew-symmetric and that $(\R^{m+n}, \star, \delta_\lambda)$ is a Carnot groups of step 2 isomorphic to $\G = (\R^{m+n}, \cdot, \delta_\lambda)$ (see Section 3.4 in \cite{biblio3}). 
\end{rem}

 \begin{exa}\label{exaHeisenberg}
The simplest example of Carnot group of step 2 is provided by Heisenberg group $\HH ^k =\R^{2k+1}$. Exhaustive introductions to Heisenberg groups can be found in \cite{biblio3, biblioLIBROSC}. The group operation is of the form $\eqref{5.1.0}$ with  
 \[
\mathcal{B}^{(1)}=  \begin{pmatrix}
0 &  \mathcal{I}_k\\
-\mathcal{I}_k & 0
\end{pmatrix},
\]
where $\mathcal I_k$ is the $k\times k$ identity matrix and the family of (non isotropic) dilations is defined as
\begin{equation*}
\delta _\lambda (x,y) =(\lambda x, \lambda ^2 y), \quad \mbox{for all } (x,y)\in \R^{2k+1}, \lambda >0.
\end{equation*}
A basis of left invariant vector fields is given by
\begin{equation*}
\begin{aligned}
X_j& = \partial _{x_j} - \frac{1}{2} x_{k+j} \, \partial _{y} , \qquad \mbox{ for all } j=1,\dots, k\\
X_{k+j}& = \partial _{x_{k+j} } + \frac{1}{2} x_{j} \, \partial _{y} , \qquad \, \mbox{ for all } j=1,\dots, k\\
Y &=\partial _{y} .
\end{aligned}
\end{equation*}
The only non trivial commutator relations being $ [X_j, X_{k+j}] =Y,\, j=1,\dots , k.$ Moreover the stratification of the Lie algebra $\mathfrak h$ of the left invariant vector fields is given by $\mathfrak h=\mathfrak h _1\, \oplus \, \mathfrak h_2$,
\begin{equation*}
\mathfrak h _1=\mbox{span} \{X_1,\dots ,X_{2k}  \}, \quad \mathfrak h _2=\mbox{span}  \{Y \}.
\end{equation*}
 \end{exa}

 \begin{exa}\label{exaCorank1}
A corank 1 Carnot group is a Carnot group of step 2 where the dimension of vertical layer is 1 (i.e., $n=1$). The group operation is of the form $\eqref{5.1.0}$ 
where $(\mathcal{B}^{(1)})_{j\ell}  =(b_ {j\ell} )$ is a $m\times m$ skew symmetric matrix. Observe that
\begin{equation*}
\begin{aligned}
X_j  & = \partial _{x_j}  -\frac{1}{2 } \sum_ {\ell=1 }^{m} b_ {j\ell} x_\ell  \,\partial _{y},  \quad \mbox{ for } j=1,\dots ,m,\\
Y &=\partial _{y} ,
\end{aligned}
\end{equation*}  
and
\begin{equation*}
[X_j, X_\ell]= b_{j\ell} Y,   \quad \mbox{ for } j, \ell =1,\dots ,m.%quad \mbox{and} \quad  [X_j , Y_{i}] = 0,\quad \forall j,\ell=1,\dots, m \;\text{ and }\; \forall i=1,\dots,n.
\end{equation*}
Obviously, $\HH^k$ is a corank 1 Carnot group.
 \end{exa}
 
 \begin{exa}\label{exaFree Step 2 Groups}

Free Step 2 Groups are examples of Carnot groups of step 2  (see \cite[Section 3.3]{biblio3}).

%Let $m\geq 2$ be a fixed integer. We denote by $(\mathbb{F}_{m,2},\star )$ the Carnot group on $\R^{m}\times \R^{\frac{m(m-1)}{2}}$ with the composition law $\eqref{1}$ defined by the matrices  $\mathcal B^{(s)}\equiv \mathcal{B}^{(l,h)}$ where $1\leq h< l\leq m$ and $\mathcal{B}^{(l,h)}$ has entries $-1$ in  position $(l,h)$, $1$ in  position $(h,l)$ and $0$ everywhere else. 

Fix an integer $m\geq 2$ and denote by $h= m+\frac{m(m-1)}{2}$. In $\R^h$ denote the coordinates by $x_j$, for $1\leq j\leq m$, and by $y_{\ell s}$, for
	$1 \leq s< \ell \leq m$. Let $\partial _j$ and $\partial _{\ell s}$ denote the standard basis vectors in this coordinate system.
	We define $h$ linearly independent vector fields on $\R^h$ by setting:
	\begin{equation}\label{eqn:VectorFieldsFree}
	\begin{aligned}
	X_j &= \partial _j+ \frac 12 \sum_{ j<\ell\leq m } x_\ell \partial _{\ell j} -\frac 12 \sum_{ 1\leq \ell<j} x_\ell \partial _{j\ell}, \quad \mbox{ if } 1\leq j\leq m,\\
	Y_{\ell s}&= \partial _{\ell s}, \hphantom{\frac 12 \sum_{ j<\ell\leq m } x_\ell \partial _{\ell j} -\frac 12 \sum_{ 1\leq \ell<j} x_\ell \partial _{j\ell}} \qquad  \mbox{ if } 1\leq s<\ell\leq m.
	\end{aligned}
	\end{equation}
	Let  $\mathbb{F}= (\R^{m+\frac{m(m-1)}{2}}, \cdot)$ be the coordinate representation of the step 2 Carnot group with $m$ generators whose Lie algebra is generated by the vector fields in \eqref{eqn:VectorFieldsFree}. Then $\mathbb F$ is free and its Carnot structure is given by
	\begin{equation*}
	V_1:= \mbox{span} \{ X_j \,:\, 1\leq j\leq m\}\quad \mbox{and} \quad V_2:= \mbox{span} \{ Y_{\ell s} \,:\, 1\leq s<\ell \leq m\}.
	\end{equation*}
	Moreover, the composition law $\eqref{5.1.0}$ also tells us that $\mathcal{B}^{(\ell,s)}$ has entry $1$ in  position $(\ell,s)$, $-1$ in  position $(s,\ell)$ and $0$ elsewhere. That means 
		\begin{equation*}
	\begin{aligned}
	(p\cdot q)_j &=p_j+q_j,\quad \quad \qquad \qquad \qquad \,\,\,\,\, \text{ if } 1\leq j\leq m,\\
	(p\cdot q)_{\ell s} &=p_{\ell s}+q_{\ell s} +\frac 1 2 (p_\ell q_s-q_\ell p_s), \quad \text{ if } 1\leq s<\ell\leq m.\\
	\end{aligned}
	\end{equation*}

\noindent It is easily verified that for $1\leq s<\ell\leq m$ and $1\leq j\leq m$, one has
\begin{equation}\label{commutatorifree}
[X_\ell , X_s] = Y_{\ell s} \quad \mbox{and} \quad  [X_j , Y_{\ell s}] = 0.
\end{equation}

%\end{prop}

%\begin{rem}
%The Heisenberg group $\HH^k$ is a free Carnot group if and only if $k = 1$.
%Indeed, since the Lie algebra  $\mathfrak g$ of  $\HH^k$ has step 2 and $2k$ generators, a necessary condition for  $\HH^k$ to be free is that $$2k +1 = \mbox{dim } \mathfrak g= 2k(2k +1)/2,$$ i.e. that $k = 1$.
%\end{rem}

 \end{exa}

 \begin{exa}\label{exacomplexified Heisenberg group}
%The \emph{complexified Heisenberg group} $\HH^1_2$ is an example of H-type group and, consequently, of Carnot groups of step 2. More information on this group can be found in \cite{biblioReimannRicci}.
The  \emph{complexified Heisenberg group} $\HH^1_2$ \cite{biblioRR} (see also Section 12 in \cite{biblioMAGNANI}) is the Carnot group of topological dimension $6$  whose Lie algebra is decomposed into
	\[
	\mathfrak h ^1_2= \mathrm{span}\{X_1,X_2,X_3,X_4\}\oplus\mathrm{span}\{Y_1,Y_2\},
	\]
	where the only non-vanishing bracket relations are given by $[X_1,X_2]=[X_3,X_4]=Y_1, [X_1,X_4]=[X_2,X_3]=Y_2$.  
	%Identify $\mathbb E$ with $\mathbb R^4$ by means of exponential coordinates of the first kind and define $\mathbb W= \{x_1=0\}$, and $\mathbb L=\{x_2=x_3=x_4=0\}$. Then, by explicit computations\footnote{GA:Add e aggiungi anche $D_{X_3}$ e $D_{X_4}$. Citazione a Kozhevnikov e spendere due parole su come cambiano i sottogruppi in Engel?}, we get 
%	
%	
%We consider $\HH^1_2= (\R^{4+2} , \cdot )$ where $\cdot $ is given by \eqref{1} with
%$$  \B^{(1)} =
%\begin{pmatrix}
% 0 &   -1 & 0 & 0\\
%1&0  & 0 & 0\\
%0  & 0 & 0 &1\\
%0&0  & -1 & 0\\
%\end{pmatrix}
%, \qquad \quad  \B^{(2)}  = 
%\begin{pmatrix}
% 0 &   0 & 0 & -1\\
%0&0  & 1 & 0\\
%0  & -1 & 0 &0\\
%1&0  & 0 & 0\\
%\end{pmatrix}
%$$
The explicit group operation on $\HH^1_2$ is given by \eqref{5.1.0} with
$$  \B^{(1)} =
\begin{pmatrix}
 0 &   -1 & 0 & 0\\
1&0  & 0 & 0\\
0  & 0 & 0 &1\\
0&0  & -1 & 0\\
\end{pmatrix}
, \qquad \quad  \B^{(2)}  = 
\begin{pmatrix}
 0 &   0 & 0 & -1\\
0&0  & 1 & 0\\
0  & -1 & 0 &0\\
1&0  & 0 & 0\\
\end{pmatrix},
$$
and so for $(x,y),(x',y') \in \R^4\times \R^2$
\begin{equation*}
(x,y) \cdot  (x',y') = \begin{pmatrix}
x+x'\\
%x_2+x'_2\\
%x_3+x'_3\\
%x_4+x'_4\\
y_1+y'_1+\frac{1}{2}\big(-x_2x'_1+x_1x'_2+x_4x'_3 -x_3x'_4\big)\\
y_2+y'_2+\frac{1}{2}\big(-x_4x'_1+x_3x'_2-x_2x'_3+x_1x'_4\big)
\end{pmatrix}.
\end{equation*}
%and  the  dilation is $ \delta _\lambda (x,y)=(\lambda x_1, \lambda x_2, \lambda x_3,\lambda x_4, \lambda ^2 y_1,\lambda ^2 y).$
%Notice that $\HH^1_2$ is H-type group whose center has dimension 2 and the first layer has dimension $4$. 
Moreover, a  basis of the Lie algebra $\mathfrak h ^1_2$ of  $\HH^1_2$  is
\begin{equation*}
\begin{aligned}
X_1 &= \partial _{x_1}- \frac{1}{2} (x_2\partial _{y_1}+x_4\partial _{y_2}),\quad  X_2 = \partial _{x_2}+ \frac{1}{2} (x_1\partial _{y_1}+x_3\partial _{y_2}) , \\
 X_3 &= \partial _{x_3}+ \frac{1}{2} (x_4\partial _{y_1}-x_2\partial _{y_2}) ,\quad  X_4 = \partial _{x_4}- \frac{1}{2} (x_3\partial _{y_1}-x_1\partial _{y_2}) , \\
Y_{1} &=\partial _{y_1} , \qquad Y_{2} =\partial _{y_2} .
\end{aligned}
\end{equation*}

 \end{exa}

\subsection{$C^1_\G$ functions, $\G$-regular surfaces, Caccioppoli sets} (See \cite{biblioLIBROSC}). 
%For the Euclidean theory of BV functions and finite perimeter sets the reader can see \cite{biblio18}. 
%The following definitions relate to $\C^1_\G$ functions (see \cite{biblio11} or \cite{biblio9}) and the notion of BV function (see for instance \cite{biblio8}) in Carnot groups. 
%For the Euclidean theory of BV functions and finite perimeter sets the reader can see \cite{biblio18} and \cite{biblio4}.
In \cite {biblio11}, Pansu introduced an appropriate notion of differentiability for functions acting between Carnot groups. We recall this definition in the particular instance that is relevant  here. 

Let $\mathcal U$ be an open subset of a step 2 Carnot group $\G$.  A function $f:\mathcal U\to \R^k$ is Pansu differentiable or more simply P-differentiable in $a \in \mathcal U$ if there is a homogeneous homomorphism 
\[
d_\mathbf Pf(a): \G\to \R^k,
\]
the Pansu differential of $f$ in $a$, such that, for $b\in \mathcal U$, 
\[
\lim_{r\to 0^+}\sup_{0<\Vert a^{-1}b\Vert<r}\frac{|f(b)-f(a)- d_\mathbf Pf(a)(a^{-1}b)|_{\R^k}}{\Vert a^{-1}b\Vert}= 0.
\]
Saying that $d_\mathbf Pf(a)$ is a  homogeneous homomorphism we mean that $d_\mathbf Pf(a): \G\to \R^k$ is a group homomorphism and also that $d_\mathbf Pf(a)(\delta_\lambda b)=\lambda d_\mathbf Pf(a)(b)$ for all $b\in \G$  and $\lambda \geq 0$.

Observe that, later on in Definition \ref{d3.2.1}, we give a different notion of differentiability for functions acting between subgroups of a Carnot group of step 2 and we reserve the notation $df$ or $df(a)$ for that differential. 

We denote  $C^1_\G (\mathcal U ,\R^k )$  the set of functions $f:\mathcal U\to \R^k$ that are P-differentiable in each $a\in \mathcal U$ and such that $d_\mathbf Pf(a)$ depends continuously on $a$. 
%We denote also $\C^1_\G (\mcal U )$ for  $C^1_\G (\mcal U ,\R)$. 

It can be proved that $f=(f_1,\dots , f_k)\in C^1_\G (\mathcal U ,\R^k )$ if and only if the distributional horizontal derivatives  
$X_lf_j $,  for $l=1\dots, m$, $j=1,\dots, k$,
are continuous in $\mathcal U $.  Remember that $C^1(\mcal U ) \subset C^1_\G (\mcal U)$ with strict inclusion whenever $\G$ is not abelian (see Remark 6 in \cite{biblio6}).

%We recall that $X_1,\dots , X_{m_1}$ are the  horizontal vector fields  generating the horizontal layer of $\G$.  
The \emph{horizontal Jacobian} (or the \emph{horizontal gradient} if $k=1$) of $f:\mcal U \to \R^k$ in $a\in \mathcal U$ is the matrix
\[
 \nabla_\G f(a):=\left[X_lf_j(a)\right] _{l=1\dots m, j=1\dots k}
\]
when the partial derivatives $X_if_j$ exist. Hence $f=(f_1,\dots , f_k)\in C^1_\G (\mcal U ,\R^k )$ if and only if its horizontal Jacobian exists and is continuous in $\mathcal U$. 
 The \emph{horizontal divergence} of $\phi:=(\phi _1,\dots , \phi _{m}):\mcal U\to \R^{m}$  is defined as 
\begin{equation*}
 \mbox{div}_\G \phi := \sum_{j=1}^{m} X_j\phi _j,
\end{equation*}
if $X_j\phi _j$ exist for  $j=1,\dots ,m$.
%We denote by $\C^1_\G (\mcal U )$ the set of continuous real-valued functions in $\mcal U$ such that the distributional derivatives  $X_1f,\dots , X_{m_1} f$ are continuous in $\mcal U $ and by and $C^1_\G (\mcal U ,\R^k )$ the set of $k$-uples  $(f_1,\dots , f_k)$ such that $f_i\in \C^1_\G (\Omega )$ for $1\leq i \leq k$. 
%The horizontal gradient of $f=(f_1,\dots , f_k)$ is the matrix $\nabla_\G f:=\left(X_i f_j\right)$ for $1\leq j\leq k$ and $1\leq i\leq m_1$.
%\begin{equation*}
%\nabla_\G F:= \left(
%\nabla_\G F_1,
%\dots, 
%\nabla_\G F_k
%\right).
%\end{equation*}
% We denote by $\C^1_\G (\mcal U ,H\G )$ the set of all sections $\phi $ of $H\G$ with canonical coordinates $\phi _j\in \C^1_\G (\mcal U, \R )$ for $j=1,\dots ,m$. 
% \begin{equation*}
%\nabla_\G F:= \begin{pmatrix}
%\nabla_\G F_1\\
%\vdots \\
%\nabla_\G F_k
%\end{pmatrix} 
%\end{equation*}

Now we use the notion of P-differentiability do introduce introduce the $\G$-regular surfaces. Regarding the bibliography, in addition to the one already mentioned, the reader can read \cite{biblioMAGNANI}.
\begin{defi}\label{Gregularsurfaces}
$S\subset \G$ is a \emph{$k$-codimensional $\G$-regular surface} if for every $p\in S$ there are a neighbourhood $\mcal U$ of $p$ and a function $f=(f_1,\dots , f_k)\in C^1_\G(\mcal U,\R^k)$ such that
\[
S\cap \mcal U=\{ q\in \mcal U : f(q)=0 \}
\]
and $d_\mathbf Pf(q)$ is surjective, or equivalently if the $(k\times m)$ matrix $\nabla_\G f(q)$ has rank $k$, for all $q\in \mcal U$.
\end{defi}

The class of $\G$-regular surfaces is different from the class of Euclidean regular surfaces. In \cite{biblio19}, the authors give an example of $\HH^1$-regular surfaces, in $\mathbb{H}^1$ identified with $\R^3$, that are (Euclidean) fractal sets. Conversely, there are continuously differentiable 2-submanifolds in $\R^3$ that are not $\HH^1$-regular surfaces (see \cite{biblio6} Remark 6.2 and  \cite{biblio1} Corollary 5.11).
\\

In the setting of step 2 Carnot groups, there is a natural definition of bounded variation functions and of finite
perimeter sets (see \cite{biblioGARN} or \cite{biblioLIBROSC} and the bibliography therein).

We say that $f:\mcal U \to \R $ is of bounded $\G$-variation in an open set $\mcal U \subset \G$ and we write $f\in BV_\G(\mcal U )$, if $f\in \mathcal L^1(\mcal U )$ and
\[
\| \nabla _\G f \| (\mcal U ):= \sup \Bigl\{ \int _{\mcal U} f \, \mbox{div}_\G\phi \, d\mathcal{L}^{m+n} : \phi \in C^1_c (\mcal U , H\G), |\phi (p)| \leq 1 \Bigl\} <+\infty .
\] 
The space $BV_{\G , loc }(\mcal U )$ is defined in the usual way.

In the setting of step 2 Carnot groups, the structure theorem for $BV_\G$ functions reads as follows.
\begin{theo}\label{structure theorem BV}
If $f\in BV_{\G , loc }(\Omega )$ then $\| \nabla _\G f \|$ is a Radon measure on $\Omega$. Moreover, there is a $\| \nabla _\G f \|$ measurable horizontal section   $\sigma _f : \Omega \to H\G$ such that $|\sigma _f (P)|=1$ for $\| \nabla _\G f \|$-a.e. $P\in \Omega $ and
\begin{equation*}
\int_{\Omega } f \mbox{div}_\G \xi \, d \mathcal{L}^{m+n}  =   \int_{ \Omega } \langle \xi ,\sigma _f\rangle\, d\| \nabla _\G f \|,
\end{equation*}
for every $\xi \in C^1_c(\Omega , H\G)$. Finally the notion of gradient $\nabla _\G$ can be extended from regular functions to functions $f\in BV_\G$ defining $\nabla _\G f$ as the vector valued measure 
\begin{equation*}
\nabla _\G f:=-\sigma _f \res \|\nabla _\G f\| =(-(\sigma _f)_1 \res \|\nabla _\G f\|  , \dots , -(\sigma _f)_{m} \res \|\nabla _\G f\| ),
\end{equation*}
where $(\sigma _f)_i$ are the components of $\sigma _f$ with respect to the base $X_i$.
\end{theo}

A set $\mcal E\subset \G$ has locally finite $\G$-perimeter, or is a $\G$-Caccioppoli set, if $\chi_{\mcal E} \in BV_{\G , loc }(\G)$, where $\chi_{\mcal E}$ is the characteristic function of the set $\mcal E$. In this case the measure $\| \nabla _\G \chi_{\mcal E}\|$ is called the $\G$-perimeter measure of $\mcal E$ and is denoted by $|\partial \mcal E|_\G$. Moreover we call generalized intrinsic normal of $\partial \E$ in $\Omega$ the vector $$\nu _\E (p):= -\sigma _{\chi_\E } (p).$$

\subsection{Complementary subgroups and graphs} A homogeneous subgroup $\W$ of $\G$ is a Lie subgroup such that $\delta _\lambda p\in \W$ for every $p\in \W$ and for all $\lambda >0$. Homogeneous subgroups are linear subspaces of $\R^{m+n}$, when $\G$ is identified with $\R^{m+n} $.
\begin{defi} We say that $\W$ and $\M$ are \emph{complementary subgroups in $\G$} if 
$\W$ and $\M$ are homogeneous subgroups of $\G$ such that  $\W \cap \M= \{ 0 \}$ and  $$\G=\W\cdot \M.$$  By this we mean that for every $p\in \G$ there are $p_\W\in \W$ and $p_\M \in \M$ such that $p=p_\W  p_\M$.
\end{defi}

If $\W$ and $\M$ are complementary subgroups of $\G$ and one of them is a normal subgroup then $\G$ is said to be the semi-direct product of $\W$ and $\M$. If both $\W$ and $\M$ are normal subgroups then $\G$ is said to be the direct product of $\W$ and $\M$. 

The elements $p_\W \in \W$ and $p_\M \in \M$ such that $p=p_\W \cdot p_\M$ are unique because of $\W \cap \M= \{ 0 \}$ and are denoted  components of $p$ along $\W$ and $\M$ or  projections of $p$ on $\W$ and $\M$.
%Differently from Euclidean spaces, 
The projection maps $\mathbf{P}_\W :\G \to \W$ and $\mathbf{P}_\M:\G \to \M$ defined
\[
\mathbf{P}_\W (p)=p_\W, \qquad \mathbf{P}_\M (p)=p_\M, \qquad \text{for all $p\in \G$,}
\]
are polynomial functions (see Proposition 2.2.14 in \cite{biblio22}) if we identify $\G$ with $\R^{m+n}$, hence are $C^\infty$. Nevertheless in general they are not  Lipschitz maps, when $\W$ and $\mathbb{M}$ are endowed with the restriction of the left invariant distance $d$ of $\G$ (see Example 2.2.15 in \cite{biblio22}). 

\begin{rem}\label{rem2.2.1}
The stratification of $\G$ induces a stratifications on the complementary subgroups $\W$ and $\M$. If $\G=\G^1\oplus \G^2$ then also $\W= \W^1\oplus \W^2$, $\M=\M^1\oplus \M^2$ and $\G^i =\W^i \oplus \M^i$. A subgroup is \emph{horizontal} if it is contained in the first layer $\G^1$. If $\M$ is horizontal then the complementary subgroup $\W$  is normal.
\end{rem}

%We recall the following inequality proved in \cite{biblio2}, Proposition 3.2.
\begin{prop}[\cite{biblio2}, Proposition 3.2]
If $\W$ and $\M$ are complementary subgroups in $\G$ there is $c_0=c_0(\W , \M)\in (0,1)$ such that for each  $p_\W \in \W$ and $p_\M \in \M$
\begin{equation}\label{c_0}
c_0(\| p_\W \|+\|p_\M \|)\leq \| p_\W  p_\M \| \leq \| p_\W \|+\|p_\M \|
\end{equation}
\end{prop}
%Currently we assume that $\W$ and $\V$ are complementary homogeneous subgroups in $\G$ and that $\V$ is $k$ dimensional and consequently $\W$ is $N-k$ dimensional. Since $\V=\{ \exp (tV)\, :\, t\in \R \}$ for some $V\in V_1$, it can be identified with $\R$. Moreover $\W$ can be identified with $\R^{N-1}$. In this particular case $\G$ is semi-direct product of $\W$ and $\V$.

%Moreover we present another basic notion of this paper:
\begin{defi}
 We say that $S\subset \G$ is a \emph{left intrinsic graph} or more simply a \emph{intrinsic graph} if there are complementary subgroups $\W$ and $\M$  in $\G$ and  $\phi: \mathcal O \subset \W \to \M$ such that
\[
S=\graph {\phi} :=\{ a \phi (a):\, a\in \mathcal O \}.
\]
\end{defi}
Observe that, by uniqueness of the components along $\W$ and $\M$, if $S=\graph {\phi}$ then $\phi $ is uniquely determined among all functions from $\W$ to $\M$.  

We call graph map of $\phi $, the function $\Phi :\mathcal O \to \G$ defined as
\begin{equation}\label{Phi}
\Phi (a):= a \cdot \phi (a) \quad \mbox{for all } a\in \mathcal O. 
\end{equation}
Hence $S=\Phi (\mathcal O )$ is equivalent to $S=\graph{\phi}$. %We can also write $S=G^1_{\G, \phi }$ in this way.

%We say that $S$ is intrinsically parameterized by $\phi: \E \subset \W \to \M$ if $S=\Phi (\E)$, where the function $\Phi :\E \to \G$ is defined as 
%\begin{equation}\label{Phi}
%\Phi (A):= A \phi (A), \qquad \forall A\in \E . 
%\end{equation}
%We call $\Phi $ graph map of $\phi $. Hence $S=\Phi (\E )$ is equivalent to $S=\graph{ \phi }$. 
%Moreover we can also write $S=G^1_{\G, \phi }$ in this way.
The concept of intrinsic graph is preserved by translation and dilation, i.e.
\begin{prop}[Proposition 2.2.18, \cite{biblio22}]\label{P2.2.18} 
If $S$ is a intrinsic graph then, for all $\lambda >0$ and for all $q\in \G$, $q \cdot S$ and $\delta _\lambda S$ are intrinsic graphs. In particular, if $S=\graph {\phi}$ with $\phi :\mathcal O \subset \W \to \M$, then
\begin{enumerate}
\item
For all $\lambda >0$, \[\delta _\lambda \left(\graph {\phi}\right) =\graph {\phi _\lambda}\]  where
$\phi _\lambda :\delta _\lambda \mathcal O \subset \W \to \M $ and 
$ \phi _\lambda (a):= \delta _\lambda \phi (\delta _{1/\lambda }a)$,  for $a \in \delta _\lambda \mathcal O$.  
\item
For any $q\in \G$, \[q \cdot \graph {\phi} = \graph {\phi _q }\] where
$\phi _q : \mathcal O _q \subset \W \to \M$ is defined as 
$\phi _q (a):= (\mathbf P_\M (q^{-1}a))^{-1} \phi( \mathbf P_\W (q^{-1}a))$, for all $a \in \mathcal O_q:=\{ a\, :\, \mathbf P_\W (q^{-1}a)\in \mathcal O  \}$.  
\end{enumerate}
\end{prop}

\subsection{Intrinsic differentiability}
%The notion of $P$-differentiability makes sense and can be introduced also for functions acting between complementary subgroups of a Carnot group $\G$. Nevertheless $P$-differentiability  does not seem to be the right  notion in this context.  Indeed P-differentiability  is a property that can be lost after a function is shifted as in Proposition \ref{P2.2.18}.  Here we recall a different notion of differentiability, the so called \emph{intrinsic differentiability} that is, by its very definition, invariant under translations. 
%A function is intrinsic differentiable  if it is locally well approximated by   \emph{intrinsic linear} functions that are functions whose graph is a homogeneous subgroup in $\G$.

%Note that the subgroup $\HH$ is algebraic and topological isomorphic to the abelian group $\R^k$.  

%\subsection{Intrinsic linear functions} 
%First we present the notion of intrinsic linear functional.
\begin{defi}
Let $\W$ and $\M$ be complementary subgroups in $\G$. Then $\ell:\W\to \M$ is  \emph{intrinsic linear} if $\ell$ is defined on all of $\W$ and if $\graph{\ell} $ is a homogeneous subgroup of $\G$.
\end{defi}

We use intrinsic linear functions to define intrinsic differentiability as in the usual definition of differentiability.
\begin{defi}\label{d3.2.1}
Let $\W$ and $\M$  be complementary subgroups in $\G$ and let $\phi :\mathcal O \subset \W \to \M$ with $\mathcal O$ open in $\W$. For $a\in \mathcal O$, let $p:=a\cdot \phi (a)$ and $\phi _{p^{-1}}: \mathcal O _{p^{-1}} \subset \W \to \M$ be the shifted function defined in Proposition $\ref{P2.2.18}$.
\begin{enumerate}
\item We say that $\phi$ is \emph{intrinsic differentiable in $a$} if the shifted function $\phi_{p^{-1}}$ is intrinsic dif\-fe\-ren\-tia\-ble in $0$, i.e. if there is a intrinsic linear $d\phi_a:\W\to \M$ such that
 \begin{equation*}\label{3.0}
\lim_{r\to 0^+}\sup_{0<\|b\|<r}\frac{\| d\phi_{a} (b)^{-1} \phi _{p^{-1}} (b) \|}{\|b\|} =0.
\end{equation*}
The function $d\phi_a$ is the \emph{intrinsic differential of $\phi $ at $a$}.

\item We say that $\phi$ is \emph{uniformly intrinsic differentiable in $a_0\in \mathcal O$} or $\phi$ is \emph{u.i.d. in $a_0$} if  there exist a intrinsic linear function $d\phi_{a_0}: \W \to \M$ such that
\begin{equation}\label{3.0.1}
\lim_{r\to 0^+}\sup_{a, b }\frac{\| d\phi_{a_0} (b)^{-1} \phi _{p^{-1}} (b) \|}{\|b\|} =0,
\end{equation}
where the supremum is for $\|a_0^{-1}a\|<r, \, 0<\|b\|<r.$ Analogously, $\phi$ is u.i.d. in $\mathcal O$ if it is u.i.d. in every point of $\mathcal O$. 
\end{enumerate}
\end{defi}

%We denote $d_{D^\phi }\phi (A):=\ell$ and we call it

\begin{rem} Definition \ref{d3.2.1} is a natural one because of the following observations.

\emph{(i)} If $\phi$ is intrinsic differentiable in $a\in \mathcal O$, there is a unique  intrinsic linear function $d\phi_a$ satisfying $\eqref{3.0}$.  Moreover $\phi$ is continuous at $a$. (See Theorem 3.2.8 and Proposition 3.2.3 in \cite{biblio21}).

\emph{(ii)} The notion of intrinsic differentiability is invariant under group translations. Precisely, let $p:=a\phi (a), q:=b\phi (b)$, then $\phi $ is intrinsic differentiable in $a$ if and only if $\phi _{qp^{-1}} := (\phi _{p^{-1}})_{q}$ is intrinsic differentiable in $b$.

\emph{(iii)}  It is clear, taking $a=a_0$ in \eqref{3.0.1}, that if $\phi$ is uniformly intrinsic differentiable in $a_0$ then it is intrinsic differentiable in $a_0$ and $d\phi_{a_0}$ is the \emph{intrinsic differential of $\phi $ at $a_0$}.
\end{rem}

%In addition to pointwise intrinsic differentiability, we are interested in an appropriate  notion of continuously intrinsic differentiable functions. For functions acting between complementary subgroups, one possible way is to introduce a stronger, i.e. uniform, notion of intrinsic differentiability in the general setting of Definition \ref{d3.2.1}. 
\medskip

%We will not use here this more general notion and we limit ourselves to introduce the notion of uniform intrinsic differentiability for functions valued in  horizontal subgroups.

From now on we restrict our setting  studying the notions of intrinsic differentiability and of uniform intrinsic differentiability for  functions $\phi: \W\to \HH$ when $\HH$ is a horizontal subgroup.  When $\HH$ is horizontal, $\W$ is always a normal subgroup since, as observed in Remark $\ref{rem2.2.1}$, it contains the whole strata $\G^2$.
In this case,   the more explicit form of the shifted function $\phi_{p^{-1}} $ allows  a more explicit form of equations \eqref{3.0} and \eqref{3.0.1}.

%Keeping in mind this special form of intrinsic linear functions we obtain the following special form of intrinsic differentiability.
\begin{prop}[Theorem 3.5, \cite{biblioDDD}]\label{prop1.1.1}
Let $\W$ and $\HH$ be complementary subgroups of a step $2$ Carnot group $\G$, $\mathcal O$ open in $\W$ and $\HH$ horizontal. Then $\phi :\mathcal O \subset \W\to \HH$ is intrinsic differentiable in $a_0\in \mathcal O$ if and only if there is a intrinsic linear  $d\phi_{a_0}:\W\to \HH$ such that
\[
\lim_{r\to 0^+}\sup_{0<\|a_0^{-1}b\|<r} \frac{\|     \phi (b) -\phi (a_0)- d\phi_{a_0}( a_0^{-1}b )\|} {\|  \phi (a_0)^{-1} a_0^{-1}b \phi (a_0)\|} =0.
\]
Analogously,  $\phi$ is uniformly intrinsic differentiable in $a_0\in \mathcal O$, or $\phi$ is u.i.d. in $a_0\in \mathcal O$,  if  there is a intrinsic linear  $d\phi_{a_0}:\W\to \HH$ such that
%$\phi$ is intrinsic differentiable in $A_0$ with intrinsic differential $d\phi_{A_0}$ and if
\[
\lim_{r \to 0^+}\sup_{a,b }    \frac { \|  \phi ( b) - \phi (a)  - d\phi_{a_0}(a^{-1} b) \|}{\|\phi(a)^{-1}a^{-1}b\phi(a)  \|}  =0
\]
where $r$ is small enough so that $\mcal U(a_0,2r)\subset \mcal O$ and the supremum is for $\Vert{a_0^{-1} a}\Vert<r,\, 0<\Vert{a^{-1} b}\Vert<r.$
%Analogously, $\phi$ is u.i.d. in $\mathcal O$ if it is u.i.d. in every point of $\mathcal O$. 

Finally, if $k<m$ is the dimension of $\HH$, and if, w.l.o.g., we assume that
\[
\HH=\{p: p_{k+1}=\dots =p_{m+n}=0\}\qquad \W=\{p: p_{1}=\dots =p_k=0\}\
\]  then there is a $k\times (m-k)$ matrix, here denoted as   $\nabla^\phi\phi(a_0)$, such that 
\begin{equation*}\label{DISSUdifferential}
d\phi_{a_0} (b)= \left(\nabla^\phi\phi(a_0) (b_{k+1},\dots,b_{m})^T,0,\dots ,0\right),
\end{equation*}
for all $b=(b_1,\dots,b_{m+n})\in \W$. The matrix $\nabla^\phi\phi(a_0)$ is called the \emph{intrinsic horizontal Jacobian} of $\phi$ in $a_0$ or the \emph{intrinsic horizontal gradient} or even the \emph{intrinsic gradient} if $k=1$.
\end{prop}
 
 Observe that u.i.d. functions do exist. In particular, when $\HH$ is a horizontal subgroup,  $\HH$ valued euclidean $C^1$ functions are u.i.d.
%Finally we show that $\C^1$ functions are uniformly intrinsic differentiable too.

\begin{theo}[Theorem 4.7, \cite{biblioDDD}]\label{propC1implicauid}
If $\W$ and $\HH$ are complementary subgroups of $\G$ with $\HH$  hori\-zon\-tal and $k$ dimensional. If
 $\mathcal O$ is open in $\W$ and  $\phi :\mathcal O \subset \W \to \HH $ is such that $\phi  \in C^1( \mathcal O, \HH)$ then $\phi$ is u.i.d. in $ \mathcal O$. 
\end{theo}

In \cite{biblioDDD}, the author gets a comparison between $\G$-regular surfaces (see Definition \ref{Gregularsurfaces}) and the uniformly intrinsic differentiable maps. %Here we give the result in the context of step 2 but the author joint with  Antonelli, Don and Le Donne proved a more general result in \cite[Theorem 1.6]{biblioADDDLD}. 
\begin{theo}\label{teo4.1}
Let $\W$ and $\HH$ be complementary subgroups of a step $2$ Carnot group $\G$ with $\HH$  horizontal and $k$ dimensional. 
Let $\mathcal O$ be open in $\W$, $\phi :\mathcal O \subset \W \to \HH $ and $S:= \graph{\phi}$. Then the following are equivalent:
\begin{enumerate}
\item there are $ \mathcal U$ open in $\G$ and $f=(f_1,\dots, f_k)\in C_\G^1(  \mathcal U; \R^k)$ such that 
\begin{equation*}
\begin{split}
& S=\{p\in  \mathcal  U: f(p)=0\}\\
& d_{\bf P}f(q)_{\vert \HH}:\HH\to \R^k\quad \text{is bijective for all $q\in \mathcal U$}
%& \det\left(X_i f_j \right)(Q)\neq 0,\qquad \text{for all $Q\in  \mathcal  U$.}
\end{split}
\end{equation*}
and $q\mapsto \left(d_{\bf P}f(q)_{\vert \HH}\right)^{-1}$ is continuous.
\item $\phi $ is u.i.d. in $\mathcal O$. 
\end{enumerate}
Moreover, if {\rm(1)} or equivalently {\rm (2)}, hold then, for all $a\in \mathcal O$ the intrinsic differential $d\phi_a$ is
\[
d\phi_a= - \left(d_{\bf P}f(a\phi(a))_{\vert \HH} \right)^{-1}\circ d_{\bf P}f(a\phi(a))_{\vert \W}.
\]
Finally, if, without loss of generality, we choose a base $X_1,\dots, X_{m+n}$ of $\mathfrak g$ such that $X_1,\dots, X_k$ are  horizontal vector fields, $\HH=\exp(\text{\rm span} \{X_1,\dots, X_k\})$ and $\W=\exp(\text{\rm span} \{X_{k+1},\dots, X_{m+n}\})$ then
\[
\HH=\{p: p_{k+1}=\dots =p_{m+n}=0\}\qquad \W=\{p: p_{1}=\dots =p_k=0\},
\] 
$
 \nabla _\G f= \left(
\, \mathcal{M}_1 \, \, | \,\, \mathcal{M}_2 \,
\right)
$
where 
 \begin{equation*}
\mathcal{M}_1 := \begin{pmatrix}
X_1f_1 \dots  X_kf_1 \\
\vdots \qquad \ddots \qquad \vdots \\
X_1f_k \dots  X_kf_k
\end{pmatrix},\qquad\mathcal{M}_2 := \begin{pmatrix}
X_{k+1}f_1\dots  X_{m}f_1 \\
\vdots \qquad \ddots \qquad \vdots \\
X_{k+1}f_k \dots  X_{m}f_k
\end{pmatrix}.
\end{equation*}
Finally, for all $q\in \mathcal U$, for all $a\in \mathcal O$ and for all $p\in \G$ $$\left(d_{\bf P}f(q)\right)(p)= \left(\nabla _\G f(q)\right)p^1$$ and the intrinsic differential is 
\begin{equation}\label{teo4.1.1N}
\begin{split}
d\phi_a(b)&= \left( \left(\nabla^\phi \phi(a)\right)(b_{k+1},\dots,b_{m})^T,0 ,\dots , 0 \right)\\
&= \left( \left( - \mathcal M_1(a\phi(a))^{-1}\mathcal M_2(a\phi(a))\right)(b_{k+1},\dots,b_{m})^T,0 ,\dots , 0 \right),
\end{split}
\end{equation}
for all $b\in \W$.
%The $k\times m_1$ matrix $\nabla^\phi \phi(A)$ is the \emph{intrinsic Jacobian} of $\phi$ in $A$.

\end{theo}

\subsection{Intrinsic Lipschitz Function}
The following notion of intrinsic Lipschitz function appeared for the first time in \cite{biblio6} and was studied, more diffusely, in \cite{biblio17, biblio27, biblio21, biblio22,  biblio24}. Intrinsic Lipschitz functions play the same role as Lipschitz functions in Euclidean context. 
\begin{defi}
Let $\W ,\HH$ be complementary subgroups in $\G$, $\phi : \mathcal O \subset \W \to \HH$. We say that $\phi $ is \emph{intrinsic $C_L$-Lipschitz  in $\mathcal O$}, or simply intrinsic Lipschitz, if there is $C_L\geq 0$ such that
\[
\|\mathbf{P}_\HH (q^{-1} q') \|  \leq C_L  \|\mathbf{P}_\W (q^{-1} q')\|,  \qquad \text{for all $q,q' \in \graph {\phi} $.}
\]
$\phi : \mathcal O \to \HH$ is locally intrinsic Lipschitz in $\mathcal O$ if $\phi $ is intrinsic Lipschitz in $\mathcal O '$ for every $\mathcal O' \Subset \mathcal O$. 
\end{defi}

If $\phi : \mathcal O \subset \W \to \HH$ is intrinsic $C_L$-Lipschitz in $\mathcal O$ then it is continuous. Indeed if $\phi (0)=0$ then $\phi$ is continuous in $0$. To prove the continuity in $a\in \mathcal O$, observe that $\phi _{q^{-1}}$ is continuous in $0$, where $q=a\phi (a)$.

\begin{rem} \label{lip0}
In this paper we are interested mainly in the special case  when $\HH$ is a horizontal subgroup and consequently  $\W$ is a normal subgroup. 
Under these assumptions, for all $p= a\phi (a),q=b \phi (b) \in \graph {\phi} $ we have 
\[\mathbf{P}_\HH (p^{-1} q)= \phi (a)^{-1}\phi(b),\quad  \mathbf{P}_\W (p^{-1} q)=\phi (a)^{-1} a^{-1}b\phi (a).
\]
 Hence, if $\HH$ is a horizontal subgroup, $\phi : \mathcal O \subset \W \to \HH$ is intrinsic Lipschitz if 
\[
\|\phi (a)^{-1}\phi(b)\| \leq  C_L \|\phi (a)^{-1} a^{-1}b\phi (a) \| \qquad \text{for all  $a,b \in \mathcal O$.}
\]
Moreover,  if $\phi$ is intrinsic Lipschitz then  $\|\phi (a)^{-1} a^{-1}b\phi (a) \|$ is comparable with $\Vert{p^{-1} q}\Vert$. Indeed from \eqref{c_0}
\begin{equation*}
\begin{split}
c_0\|\phi (a)^{-1} a^{-1}b\phi (a) \| &\leq \Vert{p^{-1} q}\Vert\\ &\leq \|\phi (a)^{-1} a^{-1}b\phi (a) \|+\|\phi (a)^{-1}\phi(b)\|\\
&\leq (1+C_L) \|\phi (a)^{-1} a^{-1}b\phi (a) \|.
\end{split}
\end{equation*}
The quantity $\|\phi (a)^{-1} a^{-1}b\phi (a) \|$, or better a symmetrized version of it, can play the role of a $\phi$ dependent, quasi distance on $\mathcal O$.  See e.g. \cite{biblio1}. 
\end{rem}

\begin{rem}
A map $\phi$ is intrinsic $C_L$-Lipschitz if and only if the distance of two points $q, q'\in $ graph$(\phi )$ is bounded by the norm of the projection of $q^{-1} q'$ on the domain $\mathcal O$. Precisely $\phi :\mathcal O \subset \W \to \HH$ is intrinsic $C_L$-Lipschitz in $\mathcal O$ if and only if there exists a constant $C_1>0$ satisfying
\begin{equation*}\label{rellip} 
\| q^{-1} q' \|\leq C_1\| \mathbf{P}_\W (q^{-1} q') \|, 
\end{equation*}
for all $q,q' \in \graph{\phi}$. Moreover the relations between $C_1$ and the Lipschitz constant $C_L$ of $\phi $ follow from $\eqref{c_0}$. In fact if $\phi $ is intrinsic $C_L$-Lipschitz in $\mathcal O $ then
\begin{equation*}
\| q^{-1} q' \|\leq \| \mathbf{P}_\W (q^{-1} q') \| +\| \mathbf{P}_\HH (q^{-1} q') \| \leq (1+C_L)\| \mathbf{P}_\W (q^{-1} q') \|,
\end{equation*}
for all $q,q' \in \graph{\phi}$. Conversely if $\| q^{-1} q' \|\leq c_0 (1+C_L)\| \mathbf{P}_\W (q^{-1} q') \|$ then
\begin{equation*}
\| \mathbf{P}_\HH (q^{-1} q' )\|\leq C_L\| \mathbf{P}_\W (q^{-1} q') \|,
\end{equation*}
for all $q,q' \in \graph{\phi}$.
\end{rem}

We observe that in Euclidean spaces intrinsic Lipschitz maps are the same as Lipschitz maps. The converse is not true (see Example 2.3.9 in \cite{biblio21}) and if $\phi :\W \to \HH$ is intrinsic Lipschitz  then this does not yield the existence of a constant $C$ such that
\[
\|\phi (a)^{-1}\phi (b)\|\leq  C \|a^{-1}b\| \quad \mbox{for } a,b\in \W
\]
not even locally.  In Proposition 3.1.8 in \cite{biblio22} the authors proved that the intrinsic Lipschitz functions, even if non metric Lipschitz, nevertheless are H\"older continuous.
\begin{prop}\label{lip84} 
Let $\W ,\HH$ be complementary subgroups in $\G$ and $\phi :\mathcal O \subset \W \to \HH$ be an intrinsic $C_L$-Lipschitz function. Then, for all $r>0$,
\begin{enumerate}
\item there is $C_1= C_1(\phi, r)>0$ such that 
\[
\|\phi (a)\| \leq C_1 \quad \text{ for all $a\in \mathcal O$ with $\|a\|\leq r;$}
\]
\item there is $C_2= C_2(C_L, r)>0$ such that  $\phi$ is locally $1/2 $-H\"older continuous, i.e., 
\begin{equation*}
\|\phi (a)^{-1}\phi (b)\|\leq C_2 \|a^{-1}b\|^{1/2 }\quad \text{for all $a, b$ with $\|a\| ,  \|b\| \leq r.$}
\end{equation*}
%where $\kappa$ is the step of $\G$.
\end{enumerate}
\end{prop}

Now we present a result which we will use later:

%\begin{prop}[\cite{biblio21}, Theorem 4.2.9]\label{Theorem 4.2.9fms} 
%If $\hat \phi :\hat{ \mathcal O} \to \V$ is intrinsic Lipschitz then the subgraph $\E$ of $\hat \phi$  is a set with locally finite $\G$-perimeter.
%\end{prop}

\begin{prop}[Proposition 3.6 \cite{biblioDDD}]\label{lip1512DDD} 
Let $\HH$, $\W$ be complementary subgroups of $\G$ with $\HH$ hori\-zon\-tal. Let  $\mathcal O$ be open in $\W$ and $\phi :\mathcal O\to \HH$ be u.i.d. in $\mathcal O$. Then
\begin{enumerate}
\item $\phi$ is intrinsic Lipschitz continuous in every relatively compact subset of $\mathcal O;$
\item the function $a \mapsto d\phi_{a}$ is continuous in $\mathcal O$.
\end{enumerate}
\end{prop}

Finally, we recall Rademacher type theorem proved in \cite{biblio21}  in codimension 1 for a large class of Carnot groups which includes Carnot group of step 2. In this paper, we need it  for step 2 and so we does not introduce this specific class.  Moreover, recently, Vittone  generalized this result in low codimension when $\G= \HH^n$ (see  \cite{biblioVIT1}).
\begin{theo}\label{Theorem 4.3.5fms}
Let $\HH$ and $\W$ be complementary subgroups of a Carnot group $\G$ of step 2  with $\HH$ one dimensional and let $\hat \phi :\mathcal O \subset \W \to \HH $ be an intrinsic Lipschitz function.  % and $\psi :\omega  \to \R$ is the map associated to $\phi$ as in $\eqref{phipsi}$. 
Then $\hat \phi$ is intrinsic differentiable $ \mathcal{L}^{m+n-1}$-a.e. in $\mathcal O$. %Notice that $ \mathcal{L}^{m+n-1} \res \W $ is the Haar measure of $\W$.
\end{theo}

\section{The intrinsic gradient}

 %{\blue qui ho messo il fatto che si pu\'o considerare $\W$ e $\V$ cos\'i fatti senza perdita di generalit\'a}
 
Let $\G = (\R ^{m+n}, \cdot  , \delta_\lambda )$ be a Carnot group of step 2 as Section \ref{Carnotinizio}  and $\W$, $\V$ be complementary subgroups in $\G$ with $\V$ horizontal and one dimensional.
 
%It is possible to characterize uniformly  intrinsic differentiable functions $\phi:\W\to \V$ (see Definitions  \ref{d3.2.1} and \ref{Ndef1.1}) in terms of existence and continuity of the \emph{intrinsic derivatives} of the components of $\phi$. 
%
%These intrinsic derivatives are non linear first order differential operators that can be explicitly written in terms  of the matrices $\mathcal B^{(s)}$ only.

\begin{rem}\label{remIMPORT} To keep notations simpler,
through all this section we assume, without loss of ge\-ne\-ra\-li\-ty, that  the complementary subgroups $\W$, $\V$ are
\begin{equation}\label{5.2.0}
\begin{aligned}
\V & :=\{ (x_1,0\dots , 0) \,:\, x_1\in \R \},\\
 \W & :=\{ (0,x_2,\dots , x_{m}, y_{1},\dots , y_{n}) \,:\, x_i,y_j\in \R \mbox{ for }i=2,\dots, m,j=1,\dots, n \}.
\end{aligned}
\end{equation}
This amounts simply to a linear change of variables  in the first layer of the algebra $\mathfrak g$ (see Remark \ref{remChange} where $\mathcal{M}_2 = \mathcal{I}_n$ is the $n\times n$ identity matrix).

\end{rem}

When $\V$ and $\W$ are defined as in \eqref{5.2.0} there is a natural inclusion $i: \R^{m+n-1}\to \W$ such that, for all $(x_2,\dots x_m,y_1,\dots,y_n)\in \R^{m+n-1}$,  
\[
i((x_2,\dots x_m,y_1,\dots,y_n)):=(0, x_2,\dots x_m,y_1,\dots,y_n)\in \W.
\]
If $\mathcal O$ and $\phi$ are respectively an open set in $\R^{m+n-1}$ and a function $\phi:\mathcal O\to \R$ 
we denote  $\hat {\mathcal O}:= i(\mathcal O)\subset \W$ and $\hat \phi :\hat{\mathcal O}\to \V$ the function defined as
\begin{equation}\label{phipsi}
\hat\phi (i(a)) :=(\phi(a), 0,\dots ,0) 
\end{equation}
for all $a\in \mathcal O$.

From Theorem \ref{teo4.1} \eqref{teo4.1.1N}, if $\hat \phi :\hat{\mathcal O} \subset \W \to \V$ is such that $\graph {\hat \phi} $ is 
%a $\G$-regular hypersurface, i.e. if $\graph {\hat \phi} $ 
locally a non critical level set of  $f\in C^1_\G (\G, \R)$ with $X_1f\neq 0$, then $\hat\phi$ is u.i.d. in $\hat{\mathcal O}$ and  the following representation of the  intrinsic gradient $\nabla^{\hat \phi} \hat\phi$ holds
\begin{equation}\label{DPHI2}
\nabla^{\hat \phi} \hat\phi (p)=-\left (\frac{X_2f}{X_1f} ,\dots , \frac{X_{m}f}{X_1f}\right)(p\cdot \hat\phi (p))
\end{equation}
for all $p\in \hat{\mathcal O}$. 

In Proposition \ref{prop2.22} we prove  a different  explicit expression of  $\nabla^{\hat \phi} \hat\phi $, not involving  $f$, but only derivatives of the real valued function  $\phi$.  

\begin{prop}[Proposition 5.4, \cite{biblioDDD}]\label{prop2.22}
%With the notations of Definition \ref{def5.1.1}, 
Let $\G = (\R^{m+n}, \cdot  , \delta_\lambda )$ be a Carnot group of step $2$ and $\V$, $\W$ the complementary subgroups defined in \eqref{5.2.0}. Let $\mcal U$ be  open in $\G$,  $f\in C^1_\G(\mcal U, \R )$ with $X_1f>0$ and assume that $S:=\{ p\in \mcal U : f(p)=0\}$ is non empty. Then there are $\hat {\mcal O}$ open in $\W$ and  $\hat{\phi} :\hat{\mcal O}  \to \V$ such that $S=\graph {\hat \phi} $.  Moreover $\hat \phi$ is u.i.d. in $\hat{\mcal O}$ and the intrinsic gradient $\nabla^{\hat \phi} \hat\phi $ is the vector 
\[
\nabla^{\hat \phi} \hat\phi (i(a))=\left(D_2^{\phi} {\phi}(a),\dots, D_m^{\phi} {\phi}(a) \right)\]
for all $a\in \mathcal O$ where, for $j=2,\dots, m$,
 %, for all $A\in \omega$, the intrinsic gradient $D^\phi \phi (A)=\left(D_2^\phi \phi (A),\dots, D_m^\phi \phi (A)\right)$ takes the following form 
\begin{equation}\label{pde5}
%- \frac{X_j f}{X_1 f} \circ \Phi  
D_j^{\phi} {\phi} = X_j \phi+\phi  \sum_ {s=1 }^{n} b_ {j1 }^{(s)} Y_s\phi
\end{equation}
%where the equality has to be meant 
in distributional sense in $\mathcal O$, where $X_j, Y_s$ are defined as \eqref{vectorfields}.
%for $j=2,\dots, m$, where the equality \eqref{pde5} has to be meant in distributional sense.
%\item [(ii)]
%The subgraph $ \E:=\{ p\in \mcal U : f(p)<0\}$ has locally finite $\G$-perimeter in $\mcal U$ and its $\G$-perimeter measure $|\partial \mcal E|_\G$ has the integral representation
%\[
%|\partial \mcal E|_\G (\mathcal{F}) =\int _{\Phi ^{-1}(\mathcal{F})} \sqrt{1+  |\nabla^{\hat \phi} \hat\phi|_{\R^{m-1}}^2} \, \, d\mathcal{L}^{m+n-1}
%\]
%for every Borel set $\mathcal{F} \subset \mathcal U $ where $\Phi:\mathcal O\to \G$ is defined as $\Phi(a):=i(a)\cdot\hat \phi(i(a))$ for all $a\in\mathcal O$. 
%\end{enumerate}
 % i.e. for each $\rho \in \C^\infty _c (I_\delta )$
%\begin{equation*}
%\int_{I_\delta } \psi (A) D^\rho _j \rho(A)\, d\mathscr{L}^{m+n-1}(A) = -\int_{I_\delta } \frac{X_j f}{X_1 f} (\Phi (A)) \rho (A)\, d\mathscr{L}^{m+n-1}(A)
%\end{equation*}
%for $j=2,\dots , m$. 
\end{prop}

From Proposition \ref{prop2.22},   if $\graph{\hat \phi}$ is a $\G$-regular hypersurface, the intrinsic gradient of $\phi $ 
takes the explicit form given in \eqref{pde5}. 
%On the contrary if $\phi$ is only a continuous function we don't sure about the existence of intrinsic gradient. 
This motivates the definitions of the operators \emph{intrinsic horizontal  gradient} and  \emph{intrinsic derivatives}.
\begin{defi}\label{defintder}
Let $\mcal O $ be  open in $\R^{m+n-1}$,  $\phi :\mcal O \to \R$ be continuous in $\mathcal O$. The \emph{intrinsic derivatives}  $D^\phi _j $, for $j=2,\dots ,m$, are the differential operators with continuous coefficients
%continuous vector fields in $\W$%We denote as $\widetilde{X}_2,\dots , \widetilde{X}_m,  \widetilde{Y}_1, \dots,  \widetilde{Y}_n$ the left invariant vector fields  defined as restrictions of the vector fields in \eqref{5.1.0}
%\begin{equation}\label{opXY}
%\widetilde{X}_j :=  {X_j}_{| \W}\quad \mbox{and} \quad \widetilde{Y}_s:=  {Y_s}_{| \W}
%\end{equation}
%for all $j=1,\dots ,m$ and $s=1,\dots , n$. 
%associated with $\phi$   is  
%the family of $(m-1)$ first order differential operators
%, or continuous vector fields,
\begin{equation*}
\begin{split}
D^\phi _j  &:=\partial _{x_j }+\sum_ {s=1 }^{n} \left(\phi b_ {j 1}^{(s)} +\frac{1}{2 } \sum_ {l=2 }^{m}  x_l  b_ {jl}^{(s)} \right) \partial _{y_s } \\
&= {X_j}_{\vert \W}+ \phi\,\sum_ {s=1 }^{n}b_ {j 1}^{(s)} {Y_s}_{\vert \W}
\end{split}
\end{equation*}
where, in the second line with abuse of notation, we denote with the same symbols $X_j$ and $Y_s$ the vector fields acting on functions defined in $\mathcal O$.

If  $\hat\phi:=(\phi, 0,\dots, 0):\hat{\mathcal O}\to \V$, we denote
\emph{intrinsic horizontal gradient} $\nabla^{\hat\phi}$ the differential operator 
%The left invariant vector fields $X_j$ and $Y_s$ are the ones defined in \eqref{5.1.0}.
\[
\nabla^{\hat\phi} :=(D^\phi _2,\dots, D^\phi _m).
\]
%We also use the symbol $D^\phi$ for the intrinsic horizontal gradient of $\phi.$
\end{defi}

\begin{defi} 
$\mathbf{(Distributional \, \, solution)}$. Let $\mcal O \subset  \R^{m+n-1}$ be open and $w=\left(w_2,\dots,w_m\right)\in \mathcal{L}^\infty _{loc} (\mathcal O , \R^{m-1})$ . We say that $\phi\in C(\mcal O)$ is a \emph{distributional solution} in $\mcal O$ of the non-linear first order PDEs' system
\begin{equation}\label{sistema}
%\nabla^{\hat\psi}  \hat\psi := 
\left(D^{\phi}_2\phi,\dots, D^{\phi}_m\phi\right)= w \,\, \mbox{ in } \mathcal O, 
\end{equation}
if for every $\zeta \in C^1 _c (\mathcal O)$
\begin{equation*}\label{distrib}
 \int _\mathcal O \phi \left( \,X_j \zeta +\phi  \sum_{s=1}^n b_ {j1 }^{(s)} Y_s \zeta \right)  \, d\mathcal{L}^{m+n-1} = -\int_\mathcal O w_j \zeta \, d \mathcal L^{m+n-1}, \quad \mbox{for } j=2,\dots , m 
\end{equation*}
\end{defi} 

\begin{rem}
If the vector fields $D^{\phi}_2,\dots, D^{\phi}_m $ are smooth we know that it is possible to connect each couple of points $a$ and $b$ in $\mcal O$ with a piecewise continuous integral curve of horizontal vector fields. This means that there is an absolutely continuous curve $\gamma _h :[t_1,t_2]\to \mathcal O  $ from $a$ to $b$ such that $-\infty <t_1<t_2<+\infty $ and
\begin{equation*}
\dot \gamma _h (t) = \sum_{j=2}^{m} h_j(t) \, D_j^\phi (\gamma _h (t)) \hspace{0,5 cm } \mbox{a.e.}  \, \, t\in (t_1,t_2)
\end{equation*}
with $h=(h_2,\dots, h_m) :[t_1,t_2]\to \R^{m-1}$ a piecewise continuous function. In our case the vector fields $D_j^\phi $ are only continuous, and consequently it isn't sure the existence of $\gamma _h$.
\end{rem}

In \cite{biblioDDD2}, we show the following relationship between intrinsic Lipschitz function and distributional solution of \eqref{sistema} when the datum $w$ is a measurable map (see \cite{biblioADDDLD, biblioADDD} for the case of continuous datum): 
\begin{theo}\label{lemma5.4bcsc} 
Let $\G = (\R^{m+n}, \cdot  , \delta_\lambda )$ be a Carnot group of step $2$ and $\V$, $\W$ the complementary subgroups defined in \eqref{5.2.0}. Let $\hat \phi :\hat{\mathcal O} \to \V $ be a continuous map  where $\hat{\mathcal O} $ is open in $\W$ and $\phi :\mathcal O  \to \R$ is the map associated to $\hat \phi$ as in $\eqref{phipsi}$. If $\phi$ is locally $1/2 $-H\"older continuous along the vertical components, then the following conditions are equivalent:
\begin{enumerate}
\item $\hat \phi $ is locally intrinsic Lipschitz function; 
\item $\phi$ is a continuous distributional solution of $(D^{\phi}_2\phi,\dots, D^{\phi}_m\phi )  =w$ in $\mathcal O $ with\\ $w\in \mathcal{L}^\infty _{loc}(\mathcal O ,\R^{m-1}).$
\end{enumerate}

\end{theo}

By Theorem \ref{Theorem 4.3.5fms} and Theorem \ref{lemma5.4bcsc}, it immediately follows
\begin{coroll}\label{coro5.7bcsc}
Let $\G = (\R^{m+n}, \cdot  , \delta_\lambda )$ be a Carnot group of step $2$ and $\V$, $\W$ the complementary subgroups defined in \eqref{5.2.0}. Let $\hat \phi :\hat{\mathcal O} \to \V $ be a continuous map  where $\hat{\mathcal O} $ is open in $\W$ and $\phi :\mathcal O  \to \R$ is the map associated to $\hat \phi$ as in $\eqref{phipsi}$. We also assume that  
\begin{enumerate}
\item $\phi$ is a continuous distributional solution of $(D^{\phi}_2\phi,\dots, D^{\phi}_m\phi )  =w$ in $\mathcal O $ with $w\in \mathcal{L}^\infty _{loc}(\mathcal O ,\R^{m-1})$
\item $\phi$ is locally $1/2 $-H\"older continuous along the vertical components 
\end{enumerate}
Then $\hat \phi$ is intrinsic differentiable $\mathcal{L}^{m+n-1}$-a.e. in $\hat{\mathcal O}$ and $(D^{\phi}_2\phi (a),\dots, D^{\phi}_m\phi (a))  =w (a)$ a.e. $a\in \mathcal O$.
\end{coroll}

We conclude the section given some explicit examples of the intrinsic gradient of a continuous map $\phi$.

\begin{exa}[Intrinsic gradient on corank 1 Carnot groups]\label{derivateintrinsecorank1}
	Let $\G$ be a corank 1 Carnot group as in Example \ref{exaCorank1} and let $\W$ and $\mathbb V$ be complementary subgroups of $\mathbb{G}$ as in \eqref{5.2.0}. Given a continuous map $\hat \phi\colon \mathcal O \subseteq \mathbb W\to \mathbb V$ as in $\eqref{phipsi}$, the \emph{intrinsic derivatives}  $D^\phi _j $ of $\phi$ are given by
	\begin{equation}\label{operatoriproiettatiinG}
\begin{aligned}
D_{j}^\phi & = \partial _{x_j}-  \left(  b_{j1} \phi  +\frac 1 2 \sum_{ \ell=2 }^{m  } x_\ell b_{j\ell}\right) \partial _{y} =  {X_j}_{\vert \mathcal O}-  b_{j1}  \phi  {Y}_{\vert \mathcal O},  \quad \mbox{ for } j=2,\dots, m.\\
%D^\phi_{i} & = \partial_{y^*_i}  =  {Y'_{i}}_{\vert U},  \qquad \mbox{ for }  i=1,\dots, h.
\end{aligned}
\end{equation}
	Then, for each $j=2,\dots,m$, every integral curve $\gamma_j\colon I\to \mathbb R^{m}$ of $D^{\phi}_j$ has vertical component $y\colon I\to \mathbb R$  satisfying the following equation
	\begin{equation*}
	\begin{aligned}
	\dot y  (t)&=-b_{j1} \phi (x_2,\dots ,x_{j-1}, x_j+t,x_{j+1},\dots ,x_m, y(t))+\frac 1 2 \sum_{ \ell =2 }^{m  } x_\ell b_{j\ell }, \\
	\end{aligned}
	\end{equation*}
and the horizontal components of $\gamma_j(0)$ are $(x_2,\dots, x_m)$.
\end{exa}

\begin{exa}[Intrinsic gradient on free Carnot groups of step 2]\label{derivateintrinsecheperguppiliberi}
	Let $\mathbb{F}$ be a free Carnot group of step $2$ as in Example \ref{exaFree Step 2 Groups} and let $\W$ and $\mathbb V$ be complementary subgroups of $\mathbb{F}$ as in \eqref{5.2.0}. Given a continuous map $\hat \phi\colon \mathcal O \subseteq \mathbb W\to \mathbb V$ as in $\eqref{phipsi}$, the \emph{intrinsic derivatives}  $D^\phi _j $ of $\phi$ are given by
	\begin{equation}\label{operatoriproiettatiinF}
	\begin{aligned}
	D_j^\phi  = \partial _{j}-\phi \partial _{j1} + \frac 1 2 \sum_{  j<\ell\leq m } x_\ell\partial _{\ell j} -  \frac 1 2 \sum_{1<s <j} x_s \partial _{js}=  {X_j}_{| \mathcal O }-\phi  {Y_{j1}}_{|\mathcal O },  \quad &\mbox{ for } j=2,\dots, m.\\
	%D^\psi_{\ell s}  = \partial _{\ell s}  =  {Y_{\ell s}}_{\vert \W},  \hphantom{\frac 1 2 \sum_{  j<\ell\leq m } x_\ell\partial _{\ell j} -  \frac 1 2 \sum_{1<s <j} x_s \partial _{js}=  {X_j}_{\vert \W}-\psi  {Y_{j1}}_{|V}}\quad &\mbox{ for }  1\leq s<\ell\leq m.
	\end{aligned}
	\end{equation} 
	Then, for each $j=2,\dots,m$, every integral curve $\gamma_j\colon I\to \mathbb R^{m+n-1}$ of $D^{\phi}_j$ has vertical components $y= (y_{\ell s})_{1\leq s<\ell\leq m}\colon I\to \mathbb R^{\frac{m(m-1)}2}$  satisfying the following equations
	\begin{equation*}
	\begin{aligned}
	\dot y _{j1} (t)&=-\phi (x_2,\dots ,x_{j-1}, x_j+t,x_{j+1},\dots ,x_m, y(t)), \\
	\dot y _{\ell j} (t)&= \frac 12 x_\ell , \quad \quad  \qquad \mbox{ if }  j<\ell\leq m,\\
	\dot y _{js} (t)&=  -\frac 12 x_s, \quad  \,\qquad \mbox{ if }  1<s <j, \\
	\dot y_{\ell s} (t)&= 0,\quad \qquad  \quad \,\quad  \mbox{ otherwise,} 
	\end{aligned}
	\end{equation*}
where the horizontal components of $\gamma_j(0)$ are $(x_2,\dots, x_m)$.
\end{exa}

\begin{exa}[Intrinsic gradient on complexified Heisenberg group]
%The \emph{complexified Heisenberg group} $\HH^1_2$ is an example of H-type group and, consequently, of Carnot groups of step 2. More information on this group can be found in \cite{biblioReimannRicci}.
Let $\G$ be the complexified Heisenberg group $\HH^1_2$ defined as in Example \ref{exacomplexified Heisenberg group} and let $\V=\{(x_1,0,\dots,0)\,:\, x_1\in \R\}$ and $\W=\{(0,x_2\dots,x_6)\,:\, x_i\in \R, $ for $ i=2,\dots, 6\}$ be complementary subgroups of $\G$ as in \eqref{5.2.0}. Given a continuous map  $\hat \phi\colon \mathcal O \subseteq \mathbb W\to \mathbb V$  as in \eqref{phipsi}, then the \emph{intrinsic derivatives}  $D^\phi _j $ of $\phi$ are
\begin{equation*}
\begin{split}
D^\phi _2  &=\partial _{x_2 }+ \phi \partial _{y_1} +\frac{1}{2 } x_3 \partial _{y_2 }, \quad D^\phi _3 =\partial _{x_3 }+  \frac{1}{2 } x_4 \partial _{y_1 }  - \frac 1 2 x_2 \partial _{y_2 }, \quad  D^\phi _4 =\partial _{x_4 }+\phi \partial _{y_1 } -\frac{1}{2 }   x_3  \partial _{y_2 }. \\
\end{split}
\end{equation*}

Then for each $j=2,3,4$ and $(x_j, \hat x_j)\in \R^{3}$ fixed, every integral curve $\gamma_j\colon I\to \mathbb R^{m+n-1}$ of $D^{\phi}_j$ has vertical components  $(\gamma _{j1},  \gamma _{j2}):[-\delta , \delta ]\to \R^2$ satisfying the following equations
$$
\dot \gamma _{j1} (t)=\left\{
\begin{array}{l}
\phi (x_j+t,\hat x_j, \gamma _j(t)), \,\,\,\,  \mbox{ if } j=2\\
\\
\frac 1 2  x_4, \qquad \quad \qquad \qquad \mbox{ if } j=3\\
\\
-\frac 1 2 x_3, \qquad \qquad \qquad \, \mbox{ if } j=4\\
\end{array}
\right.\\
\qquad \quad \dot \gamma _{j2} (t)=\left\{
\begin{array}{l}
\frac 1 2  x_3, \qquad \quad \,\,\,\,\,\,\,\qquad \qquad  \mbox{ if } j=2\\
\\
-\frac 1 2 x_2, \quad \qquad \qquad \,\,\qquad \mbox{ if } j=3\\
\\
\phi (x_j+t,\hat x_j, \gamma _j(t)),   \quad \quad \mbox{ if } j=4\\
\end{array}
\right.\\
$$
where the horizontal components of $\gamma_j(0)$ are $(x_2, x_3,x_4)$.

 \end{exa}

\section{Existence of a Lagrangian type solution}\label{Lagrangian solution}
In this section we introduce and give some properties of so-called Lagrangian type solution of \eqref{sistema} (see Definition \ref{defiLagrangiana}).  This name follows  because it generalizes Lagrangian solution of \eqref{sistema} in the context of Heisenberg groups defined in \cite{biblio17}. Our setting is any Carnot groups of step 2 as in Section \ref{Carnotinizio}.  
%Finally, we prove the existence of it (see Proposition \ref{lemma4.2bcsc}).
 %We have now motivated the study for the concept of Lagrangian type solution: t

 The idea of Lagrangian type solution is that the reduction on characteristics is not required on
any characteristic as happened in broad* solution, but on a suitable subset of characteristics. Exhibiting a suitable set of characteristics and so proving that this definition is not empty
 is the topic of  Section \ref{Existence Lagrangian type parameterization}. 

\subsection{Definition of Lagrangian type solution}
{\bf It is convenient to introduce the following notations}: for $j=2,\dots ,m $, $s=1,\dots , n$, 
 $\hat x_j=( x_2, \dots ,  x_{j-1} , x_{j+1}, \dots ,  x_{m}) \in \R^{m-2}$ and\\ $\hat y_s=( y_1, \dots ,  y_{s-1} , y_{s+1}, \dots , y_{n}) \in \R^{n-1}$ fixed, we denote by $$(t,\hat x_j) :=( x_2, \dots , x_{j-1}, t , x_{j+1}, \dots ,  x_{m} )$$ and $$ (y_{s} ,\hat y_s):= (y_1,\dots , y_{s-1},y_s , y_{s+1},\dots , y_n).$$

 Moreover, let $\mathcal O \subset \R^{m+n-1}$,\\ %recalling that we denote a point of $\G$ and $\W$ as, respectively $p=(x_1,x,y) \in \G$ and $a=(0,x,y)=(x,y) \in \W$, where $(x_1,x)\in \R^{m}$ and $y\in \R^{n}$ we denote:\\
%$ \bullet $ If $(x,y)=(x_2,\dots ,x_m,y_1,\dots y_n)\in \R^{m-1}\times \R^n$, for given $j=2,\dots , m$ we denote\\ $\hat x_j =(x_2,\dots ,x_{j-1} ,x_{j+1} , \dots,x_m)$ and $\hat y_s =(y_1,\dots ,y_{s-1} ,y_{s+1} , \dots,y_n)$. When we will use this notation, we also denote a point $A=(x,y)\in \W$ as $A=(x_j, \hat x_j ,y)$.\\
$ \bullet $  for given $\hat x _j \in \R^{m-2},$ we will denote $x_l$ instead of $(\hat x _j)_l$.\\
$ \bullet $  for given $\hat x _j \in \R^{m-2}$ and $y\in \R^{n}$, we will denote $\mathcal O_{\hat x _j, y}:= \{ t \in \R \, :\, (t,\hat x_j, y)\in \mathcal O \}$.\\
$ \bullet $  for given $x  \in \R^{m-1}$, we will denote $\mathcal O _{ x }:= \{ y\in \R^{n} \, :\, (x, y)\in \mathcal O \}$.\\
$ \bullet $ for given $\hat x _j \in \R^{m-2}$, we will denote $\mathcal O _{\hat x _j}:= \{ (t,y) \in \R \times \R^n \, :\, (t,\hat x_j, y)\in \mathcal O \}$.\\
$ \bullet $ for given $t\in \R$ and  $x  \in \R^{m-1},$ we will denote $\mathcal O _{y_s}:= \{ (t,\hat x_j,y_1,\dots , y_{s-1},y_{s+1},\dots , y_n) \in \R^{m-1} \times \R^{n-1}\, :\, (t,\hat x_j, y)\in \mathcal O \}$.\\
$ \bullet $  for given $t\in \R,$ we will denote $\mathcal O _{t}:= \{ (\hat x_j,y) \in \R^{m-2} \times \R^n \, :\, (t,\hat x_j, y)\in \mathcal O \}$.\\
%$ \bullet $ Let $\mathcal O \subset \R^{N-1}$ and $\psi :\mathcal O \to \R$ be a function, for given $\hat x _j \in \R^{m-2}$, we will denote $\psi _{\hat x _j}:\mathcal O _{\hat x _j} \to \R$ the map defined as $\psi _{\hat x _j} (t,y):=\psi (t,\hat x _j,y)$.

We begin recalling that a set $\mathcal A \subset \R^{m+n}$  is universally measurable if it is measurable w.r.t. every Borel measure, (see \cite{Librocheservefinale}, Section 5.5). For example the open sets of $\R^{m+n}$ are  universally measurable.

Universally measurable sets constitute a $\sigma$-algebra, which includes analytic sets. A function $f:\R^{m+n} \to \R$ is said universally measurable if it is measurable w.r.t. this $\sigma$-algebra. In particular, it will be measurable w.r.t. any Borel measure. Notice that Borel counterimages of universally measurable sets are universally measurable. Then the composition $\phi \circ \psi $ of any universally measu\-ra\-ble function $\phi$ with a Borel function $\psi$ is universally measurable. This composition would be nasty with $\psi$ just Lebesgue measurable. Since restrictions of Borel functions on Borel sets are Borel, all the terms in the following definition are thus meaningful.

\begin{defi}\label{defiA}  
$\mathbf{Lagrangian \, \,  type \,\, parameterization}$. For $j=2,\dots , m$, a family of partial Lagrangian  type  parameterizations associated to a continuous function $\phi : \mathcal O  \to \R$ and to the system $(D^{\phi}_2\phi,\dots, D^{\phi}_m\phi )= w $ is a family of $(\tl {\mathcal O _j}, \chi ^1_j,\dots ,   \chi ^n_j)=(\tl {\mathcal O _j}, \chi _j)$  with $\tl {\mathcal O _j}\subset \R^{m+n-1}$ open sets and,  for $s=1,\dots , n$, $\chi  _{js}: \tl {\mathcal O _j} \to \R$ are Borel functions such that\\
$(L.1)$  the map $\Upsilon _j:\tl {\mathcal O _j}   \to \R^{m+n-1}$, $\Upsilon_j(x,y)=(x,  \chi _{j1} (x,y),\dots ,  \chi  _{jn}(x,y))$ is valued in $\mathcal O  $;\\
$(L.2)$ for each $x\in \R^{m-1}$ and for all $s=1,\dots , n$, the function $\tl {\mathcal O} _{j(x,\hat y_s)} \ni y_s \mapsto \chi  _{js}(x,y_{s} ,\hat y_s)$ is non-decreasing;\\
$(L.3)$ for each $x  \in \R^{m-1}$, $y\in \R^{n}$ and $(x_j-\delta,x_j+\delta ) \subset \tl {\mathcal O} _{j(\hat x _j,y)}$ the function $(-\delta , \delta  ) \ni t \mapsto \Upsilon_j (x_j+t, \hat x _j,y)$ is absolutely continuous and $\Upsilon_j $ is an integral curve of the vector field $D^\phi _j$, i.e.
\begin{equation}\label{vector}
\frac{\partial \Upsilon_j}{\partial t} (x_j+t, \hat x _j,y)= D^\phi _j (\Upsilon_j (x_j+t, \hat x _j,y)), \qquad t \in (-\delta , \delta ).
\end{equation}
 
 We call it a family of $($full$)$ Lagrangian type  parameterizations if $\chi _{js} : \tl {\mathcal O } _{j(x,\hat y_s)} \to \{ y_s \in \R\,:\\(x,y_s,\hat y_s) \in  \mathcal O \}  $ is onto. %the section $\mathcal O _x$ for all $x\in \R^{m-1}$. {\red ? mettere che T \'e continua} 
\end{defi} 

\begin{defi}\label{defiLagrangiana} 
$\mathbf{Lagrangian \, \,  type \,\, solution}$. A continuous function $\phi :\mathcal O  \to \R$ is a Lagrangian  type solution of 
\begin{equation}\label{solLagrangiana}
(D^{\phi}_2\phi,\dots, D^{\phi}_m\phi )= w \quad \mbox{in } \mathcal O, 
\end{equation}
  if there exists a family of Lagrangian  type parameterizations $(\tl {\mathcal O } _j , \chi _j)\, ($for $j=2,\dots , m)$ associated to $\phi $ and \eqref{solLagrangiana}, and a family of universally measurable functions $\bar w_{\chi _j}\in  \mathcal{L}^{\infty }(\mathcal O) $ for $j=2,\dots,m$, such that for all $ x\in \R^{m-1}$, $y\in \R^{n}$ and $(x_j-\delta,x_j+\delta )\subset \tl {\mathcal O } _{\hat x_j,y}$ it holds that\\
$(LS1)$  the function $(-\delta,\delta)\ni t \mapsto \phi (\Upsilon_j(x_j+t,\hat x _j ,y))$ is absolutely continuous and
\begin{equation*}\label{vector1}
\frac{d }{d t} \phi (\Upsilon_j(x_j+t,\hat x _j ,y)) = \bar w _{\chi _j}(\Upsilon _j(x_j+t,\hat x _j ,y)) \qquad \mathcal{L}^1\mbox{-a.e }t \in (-\delta,\delta).
\end{equation*}
$(LS2)$ for $j=2,\dots , m$ if there is $\alpha _j \in \R$ such that for every integral curve $\Gamma _j :I\to \mathcal O$ of $D^\phi _j$ with $\Gamma _j (0)=(x,y)$ and
\[
\frac{d}{dt} \phi (\Gamma _j (t))|_{t=0} =\lim _{t\to 0 } \frac{\phi (\Gamma _j(t)) -\phi (\Gamma _j(0))}{t} =\alpha _j,
\]
then $\alpha _j =D^\phi _j \phi (x,y)=\bar w _{\chi _j}(x,y)$.\\
% then $\bar w _{\chi _j}(x,y)= D^\phi _j \phi (x,y)$, for every $ j=2,\dots, m$.\\
$(LS3)$ $\bar w_{\chi _j}= w_j$  $\mathcal{L}^{m+n-1}$-a.e. in $\mathcal O $, for all $j=2,\dots , m$.
\end{defi} 

%\begin{defi} 
%Let $\phi : \mathcal O \to \R$ be a continuous function and let $a\in \mathcal O$. We say that $\phi $ admits $D^\phi _j$-derivative at $a$ for $j=2,\dots , m$ if there is $\alpha _j \in \R$ such that for every integral curve $\Gamma _j :(-\delta , \delta )\to \mathcal O$ of $D^\phi _j$ with $\Gamma _j (0)=a$ and
%\[
%\frac{d}{dt} \phi (\Gamma _j (t))|_{t=0} =\lim _{t\to 0 } \frac{\phi (\Gamma _j(t)) -\phi (\Gamma _j(0))}{t} =\alpha _j
%\]
%and we will denote $\alpha _j :=D^\phi _j \phi (a)$.
%\end{defi}

%\begin{rem} 
%Note that if $\graph{\phi}$ is a $\G$-regular hypersurface then for all $j=2,\dots, m$ and $\phi $ admits $D^\phi _j$-derivative at $a$, for each $a \in \mathcal O $ (see \cite{biblioDDD}). %and $D^\phi _j \phi (A)= d_\W \phi (A)(e_{j-1})$, where $(e_1,\dots , e_{m-1} )$ is the standard basis of $\R^{m-1}$.
%\end{rem} 

\subsection{Existence of (partial) Lagrangian  type  parameterization}\label{Existence Lagrangian type parameterization}
 In this section we show that the definition of the Lagrangian  type  parameterization is not empty.

%Firstly, in Lemma \ref{lemma4.1bcsc}  we show that if one takes the integral curves through an $t$-section of $\mathcal O$ which are minimal, in the sense that any other curve through that point lies on its right side, then we get a partial Lagrangian parameterization. The same happens when selecting the maximal ones. Here the result is true for a Carnot groups of step 2. Extending it to a full one will be matter of Proposition \ref{AppendiceA.1}
%
%VEDERE DI SCRIVEREEEEEEEEE QUALCOSA A RIGUARDOOOOOOOO 
%
%and in order to obtain this extension we restrict the setting in a subclass of Carnot groups of step 2. A full Lagrangian type parameterization basically amounts to an order preserving parameterization, with a real valued parameter, of non-crossing curves through each point of the plane. 

\begin{lem}\label{lemma4.1bcsc} 
 Let $\mathcal O \subset \R^{m+n-1} $ be a relatively open and bounded set and let $\phi : clos( \mathcal O ) \to \R$ be a continuous function with $\{ y=0\}\subset \mathcal O$. Then there exist domains $\tl {\mathcal O}_m$, $\tl {\mathcal O} _M$ associated to the functions
\begin{equation}\label{minode1}
\begin{aligned}
\chi _m (x_j+t,\hat x _j, y ):=\min \Big\{  & \gamma _j (t) = (\gamma _{j1}(t),\dots ,  \gamma _{jn}(t) )\, : \,  (x_j+r,\hat x _j, \gamma _{j}(r))\in \mbox{clos} ( \mathcal O)  , \\ &  ( \gamma _{j1}(0 ), \dots \gamma _{jn}(0 ))=(y_1,\dots ,y_n) \\ \, \, &   \dot \gamma _{js} (r)= b^{(s)}_{j1}\phi (x_j+r,\hat x_j, \gamma _j(r))+ \frac{1}{2}\sum_{l=2}^{m}b^{(s)}_{jl}(\hat x_j )_l  \quad \mbox{for all } s=1,\dots ,n\, \Big\}
\end{aligned}
\end{equation}
\begin{equation}\label{maxode2}
\begin{aligned}
\chi _M (x_j+t,\hat x _j, y ):=\max \Big\{  & \gamma _j (t)  = (\gamma _{j1}(t),\dots ,  \gamma _{jn}(t) )\, : \,  (x_j+r,\hat x _j, \gamma _{j}(r))\in \mbox{clos} ( \mathcal O )  , \\ &  ( \gamma _{j1}(0 ), \dots \gamma _{jn}(0 ))=(y_1,\dots ,y_n) \\ \, \, &   \dot \gamma _{js} (r)= b^{(s)}_{j1}\phi (x_j+r,\hat x_j, \gamma _j(r))+ \frac{1}{2}\sum_{l=2}^{m}b^{(s)}_{jl}(\hat x_j )_l  \quad \mbox{for all } s=1,\dots ,n\, \Big\}
\end{aligned}
\end{equation}
for which $(\tl {\mathcal O}_m , \mathcal O _m)$, $(\tl {\mathcal O}_M , \mathcal O_M)$ are partial Lagrangian  type  parameterization relative to $\phi$.
\end{lem}

\begin{proof}
We split the proof in two steps. In the first step we refine the technique used in \cite[Lemma 4.1]{biblio17} in the context of $\HH^n$; while the second step is the main difference w.r.t. the case of Heisenberg groups.

$\mathbf{Step \, 1.}$ We consider the case $n=1$. Fix $j=2,\dots , m$ and $\hat x _j\in \R^{m-2}$. For simplicity, we choose $x_j=0.$ For each $(\bar t , \hat x_j, \bar y)\in \mathcal O$ we could consider the minimal and the maximal curve satisfying on clos$(\mathcal O )$ the ODE for characteristics $\eqref{vector}$ and passing through that point: in fact the functions
\begin{equation}\label{ode1}
\begin{aligned}
\gamma_{ (\bar t , \hat x_j, \bar y)} (t ):=\min  \Big\{  & \gamma _j (t) \, : \,  (r,\hat x _j, \gamma _{j}(r) )\in \mbox{clos} ( \mathcal O )  , \,   \gamma _j(\bar t )=\bar y  \\ \, \, &   \dot \gamma _{js} (r)= b^{(s)}_{j1}\phi (r,\hat x_j, \gamma _j(r))+ \frac{1}{2}\sum_{i=2}^{m}b^{(s)}_{ji}(\hat x_j )_i  \quad \mbox{for all } s=1,\dots ,n\, \Big\} \\
\end{aligned}
\end{equation}
\begin{equation}\label{ode2}
\begin{aligned}
\gamma^{ (\bar t , \hat x_j, \bar y)} (t ):=\max  \Big\{  & \gamma _j (t) \, : \,  (r,\hat x _j, \gamma _{j}(r) )\in \mbox{clos} ( \mathcal O )  , \,   \gamma _j(\bar t )=\bar y  \\ \, \, &   \dot \gamma _{js} (r)= b^{(s)}_{j1}\phi (r,\hat x_j, \gamma _j(r))+ \frac{1}{2}\sum_{i=2}^{m}b^{(s)}_{ji}(\hat x_j )_i  \quad \mbox{for all } s=1,\dots ,n\, \Big\} \\
\end{aligned}
\end{equation}
are well defined, Lipschitz and because of the continuity of $\phi $ they are still integral curves.

Moreover, denoting by r.i. the relative interior of a set, we define the domains 
\begin{equation*}
\begin{aligned}
\tl {\mathcal O}_{min} =\mbox{ r.i.} \{ ( t , \hat x_j, y)\in \R\times \mathcal O_{j,0} \, : \, (\hat x _j , y)\in \mathcal O_{j,0} , \, (t,\hat x _j ,\gamma _{(0,\hat x_j ,y)}(t))\in \mathcal O \}, \\
\tl {\mathcal O}_{max} =\mbox{ r.i.} \{ ( t , \hat x_j, y)\in \R\times \mathcal O_{j,0} \, : \, (\hat x _j , y)\in \mathcal O_{j,0} , \, (t,\hat x _j ,\gamma ^{(0,\hat x_j ,y)}(t))\in \mathcal O \}, \\
\end{aligned}
\end{equation*}
 where $\mathcal O _{j,0} = \{( \hat x _j ,r)\in \R^{m-1} \, : \, (0, \hat x _j ,r)\in \mathcal O \} $. Note that $\tl {\mathcal O}_{min}$ and $\tl {\mathcal O}_{max}$ are both nonempty because, by hypothesis,  $(0,\hat x _j ,0)\in \mathcal O$.
 
Using the fact that $\gamma_{ (\bar t , \hat x_j, \bar y)} (t )$ and $\gamma^{ (\bar t , \hat x_j, \bar y)} (t )$ are Lipschitz solutions to the ODE with continuous coefficients, we have that $\chi _{min}  ( t , \hat x_j, y)$, $\chi _{max}  ( t , \hat x_j, y)$ in the statement are $C^1 (\mathcal O \cap \{ \tau =y\} )$ in the $t$ variable for each $y $ fixed. 

Now it remains to show that $\chi _{min}$, $\chi _{max}$ are monotone w.r.t. $y$. Indeed, noticing that $\chi _{min}$, $\chi _{max}$ are jointly Borel in $( t , \hat x_j, y)$ by continuity in $t$ and the monotonicity in $y,$ we obtain the thesis, as desired.

Hence, we prove the monotonicity w.r.t. $y$ of $\chi _{min} (t, \hat x_j, y )$ for every $t$ fixed. This is a direct consequence of the following semigroup property: for $h_1, h_2 >0$ for example for $\eqref{ode1}$
\begin{equation*}
\gamma_{ (0, \hat x_j, y)} (h_1 )=: y_1, \quad \gamma_{ (h_1, \hat x_j, y_1)} (h_1 + h_2 )=: y_2  \quad \Longrightarrow \quad \gamma_{ (0, \hat x_j, y)} (h_1 + h_2 )=y_2 .
\end{equation*}
Indeed, if we put  $\chi _{min} (t, \hat x_j, y ):= \gamma _{(0, \hat x_j, y)} (t),$ we have that if $y _1< y _2$ and $\gamma_{ (0, \hat x_j, y_1)} ( t_1 )\geq \gamma_{ (0, \hat x_j, y_2)} ( t_1 )$ at a certain $t_1 >0$, by the continuity of the curves there is $ \bar t>0$ with $\bar t < t_1$ when $\gamma_{ (0, \hat x_j, y_1)} (\bar t ) = \gamma_{ (0, \hat x_j, y_2)} (\bar t )$. But then the curve
 \begin{equation*}
\gamma (t)=  \left\{
\begin{array}{l}
\gamma_{ (0, \hat x_j, y_1)} (t ), \quad \mbox{for } t\leq \bar t,
\\
\gamma_{ (0, \hat x_j, y_2)} (t ), \quad \mbox{for } t\geq \bar t,
\end{array}
\right.
\end{equation*}
is a good competitor for the definition of $\gamma_{ (0, \hat x_j, y_1)} $, which implies 
 \begin{equation*}
 \gamma_{ (0, \hat x_j, y_2)} ( t_1 )=\gamma (t_1)\geq \gamma_{ (0, \hat x_j, y_1)} (t _1 ),
\end{equation*}
and so we get the equality. In a similar way, we prove the statement for $\gamma ^{(0,\hat x _j,y )}$. 

$\mathbf{Step \, 2.}$ We consider the case $n>1$. We want to reduce to the step 1.

Fix $j=2,\dots , m$ and $\hat x _j\in \R^{m-2}$. If $b^{(s)}_{1j}= 0$ for some $s=1,\dots, n$ it is evident that there is a unique solution of the Cauchy problem 
\begin{equation*}
\left\{
\begin{array}{l}
 \dot \gamma _{js} (t)= \frac{1}{2}\sum_{i=2}^{m}b^{(s)}_{ji}x_i ,
\\
\gamma _{js} (0)= y_s,
\end{array}
\right.
\end{equation*}
Obviously, it is
\begin{equation}\label{Gammaovvio}
\gamma _{js} (t)=\frac{1}{2} t \sum_{l=2}^{m}b^{(s)}_{jl}x_l + y_s.
\end{equation}

Henceforth we consider the curves $\gamma _{js} $ when $b^{(s)}_{1j} \ne 0$. 

Without loss of generality we can suppose that $b^{(1)}_{1j}\ne 0$. Let $\alpha_{1s}:=\frac{b^{(s)}_{j1}}{b^{(1)}_{j1}}.$ The key observation is that we can reduce $\gamma _{j2}, \dots , \gamma _{jn}$  satisfying
\begin{equation}\label{gammaRISOLVE}
\begin{aligned}
 \dot \gamma _{js} (t)= b^{(s)}_{j1}\phi (t,\hat x_j, \gamma _j(t))+ \frac{1}{2}\sum_{l=2}^{m}b^{(s)}_{jl}x_l,
\end{aligned}
\end{equation}
with $b^{(s)}_{1j} \ne 0$ and $\gamma_{ js} (0)=y_s$ as
\begin{equation*}
\begin{aligned}
\dot \gamma _{js}(t)= \alpha_{1s} \dot \gamma _{j1} (t) + \frac{1}{2}  \sum_{l=2}^{m}x_l \Big( b^{(s)}_{jl}-\alpha_{1s} b^{(1)}_{jl} \Big), %\qquad \mbox{for } s=2,\dots, n
\end{aligned}
\end{equation*}
for every $s=2,\dots, n$ with $b^{(s)}_{j1} \ne 0$ and, consequently, 
\begin{equation}\label{infunzgamma}
\begin{aligned}
\gamma _{js}(t)= \alpha_{1s}\gamma _{j1} (t) + \frac{1}{2} t\, \sum_{l=2}^{m}x_l \Big( b^{(s)}_{jl}-\alpha_{1s} b^{(1)}_{jl} \Big) +\left( y_s -\alpha_{1s} y _1\right), %\qquad \mbox{for } s=2,\dots, n
\end{aligned}
\end{equation}
for every $s=2,\dots, n$ with $b^{(s)}_{j1} \ne 0$. Hence the function $\phi $ depends only on $\gamma _{j1}$ and we can apply the case $n=1$.  
%Consequently, there exist $\gamma_{ (\bar t , \hat x_j, \bar \tau)} $ and $\gamma^{ (\bar t , \hat x_j, \bar \tau)} $ associated to $\gamma _{j1}$ satisfying $\eqref{minode1}$ and $\eqref{maxode2}$. We define
%\begin{equation*}
%\begin{aligned}
%\Gamma_{ (\bar t , \hat x_j, \bar \tau)} & := \Big(\gamma_{ (\bar t , \hat x_j, \bar \tau)} ,\gamma^2_{ (\bar t , \hat x_j, \bar \tau)} ,\dots , \gamma^n_{ (\bar t , \hat x_j, \bar \tau)} \Big)\\
%\Gamma^{ (\bar t , \hat x_j, \bar \tau)} & := \Big(\gamma^{ (\bar t , \hat x_j, \bar \tau)} ,\gamma_2^{ (\bar t , \hat x_j, \bar \tau)} ,\dots , \gamma_n^{ (\bar t , \hat x_j, \bar \tau)} \Big)\\
%\end{aligned}
%\end{equation*}
%where if $b^{(s)}_{1j}= 0$ the curve $\gamma^s_{ (\bar t , \hat x_j, \bar \tau)}$ is defined as $\eqref{Gammaovvio}$ and %if $b^{(s)}_{1j}\ne 0$ we have
%\begin{equation}\label{relazionegamma}
%\begin{aligned}
%\gamma^s_{ (\bar t , \hat x_j, \bar \tau)} (t)&= a_{1s} \gamma_{ (\bar t , \hat x_j, \bar \tau)} (t) +  \frac{1}{2} t\, \sum_{i=2}^{m}(\hat x_j )_i \Big( b^{(s)}_{ji}- a_{1s}b^{(1)}_{ji} \Big) +\big(\tau _s - a_{1s}\tau _1\big)\\
%\gamma_s^{ (\bar t , \hat x_j, \bar \tau)} (t)&= a_{1s} \gamma^{ (\bar t , \hat x_j, \bar \tau)}(t) +  \frac{1}{2} t\, \sum_{i=2}^{m}(\hat x_j )_i \Big( b^{(s)}_{ji}-  a_{1s} b^{(1)}_{ji} \Big) +\big(\tau _s -a_{1s} \tau _1\big)\\
%\end{aligned}
%\end{equation}
%for $s=2,\dots , n$ with $b^{(s)}_{1j}> 0$. 
Consequently, there exist $\gamma_{ (0 , \hat x_j, y)} $ and $\gamma^{ (0, \hat x_j, y)} $ associated to $\gamma _{j1}$ satisfying $\eqref{minode1}$ and $\eqref{maxode2}$. 

Now we define
\begin{equation*}
\begin{aligned}
\Upsilon _{ (0 , \hat x_j,  y)} & := \Big(\gamma_{ (0, \hat x_j, y)} ,\gamma^2_{ (0 , \hat x_j, y)} ,\dots , \gamma^n_{ (0 , \hat x_j,y)} \Big),\\
\Upsilon^{ (0 , \hat x_j,  y)} & := \Big(\gamma^{ (0 , \hat x_j,  y)} ,\gamma_2^{ (0 , \hat x_j, y)} ,\dots , \gamma_n^{ (0 , \hat x_j, y)} \Big),\\
\end{aligned}
\end{equation*}
where if $b^{(s)}_{j1}= 0$ the curve $\gamma^s_{(0 , \hat x_j, y)}=\gamma _s^{(0 , \hat x_j, y)}$ is defined as $\eqref{Gammaovvio}$ and if $b^{(s)}_{1j}\ne 0$ we have two cases: if $\alpha_{1s} >0$, then
\begin{equation}\label{relazionegamma.1}
\begin{aligned}
\gamma^s_{ (0 , \hat x_j, y)} (t)&= \alpha_{1s} \gamma_{ (0 , \hat x_j,y)} (t) +  \frac{1}{2} t\, \sum_{l=2}^{m}x_l \Big( b^{(s)}_{jl}- \alpha_{1s}b^{(1)}_{jl} \Big) +\big(y_s - \alpha_{1s} y_1\big),\\
\gamma_s^{(0 , \hat x_j, y)} (t)&= \alpha_{1s} \gamma^{ (0 , \hat x_j,y)}(t) +  \frac{1}{2} t\, \sum_{l=2}^{m}x_l \Big( b^{(s)}_{jl}-  \alpha_{1s} b^{(1)}_{jl} \Big) +\big(y _s -\alpha_{1s} y_1\big).\\
\end{aligned}
\end{equation}
On the other hand, if $\alpha_{1s} <0$, then
\begin{equation}\label{relazionegamma.2}
\begin{aligned}
\gamma^s_{ (0 , \hat x_j, y)} (t)&= \alpha_{1s}  \gamma^{ (0 , \hat x_j,y)} (t) +  \frac{1}{2} t\, \sum_{l=2}^{m}x_l \Big( b^{(s)}_{jl}- \alpha_{1s}b^{(1)}_{jl} \Big) +\big(y_s - \alpha_{1s} y_1\big),\\
\gamma_s^{(0 , \hat x_j, y)} (t)&= \alpha_{1s}\gamma_{ (0 , \hat x_j,y)} (t) +  \frac{1}{2} t\, \sum_{l=2}^{m}x_l \Big( b^{(s)}_{jl}-  \alpha_{1s} b^{(1)}_{jl} \Big) +\big(y _s -\alpha_{1s} y_1\big).\\
\end{aligned}
\end{equation}

We would like to prove that $\Upsilon_{ (0 , \hat x_j, y)} (t) $ satisfying $\eqref{minode1}$. 

We consider the case $\alpha_{1s} >0$ and we prove that $\gamma^s_{ (0 , \hat x_j, y)}$ is the minimal curve associated to $\gamma _{js}$, i.e. if there is $\hat \gamma^s_{ (0 , \hat x_j,  y)}$ satisfied \eqref{gammaRISOLVE} %with $\gamma^s_{ js} (0)=\tau _s$ 
and such that  
\begin{equation*}\label{8nov1}
\hat \gamma^s_{ (0 , \hat x_j, y)} (r)\leq \gamma^s_{ (0 , \hat x_j, y)}(r),
\end{equation*}
for some $r>0$, then 
\begin{equation}\label{assurdo1}
\hat \gamma^s_{ (0 , \hat x_j, y)} (r)= \gamma^s_{ (0 , \hat x_j, y)}(r).
\end{equation}
 We consider the curve $\hat \gamma^1_{ (0 , \hat x_j, y)} $ associated to $\hat \gamma^s_{ (0 , \hat x_j,y)} $ as in $\eqref{infunzgamma}$; recall that $b^{(s)}_{j1} \ne 0,$ it easy to see that
\begin{equation}\label{8nov2}
\hat \gamma^1_{ (0 , \hat x_j, y)}(t) = \alpha^{-1}_{1s}\hat \gamma^s_{ (0 , \hat x_j, y)}  (t) +\frac{1}{2} t\, \sum_{l=2}^{m}x_l \Big(b^{(1)}_{jl}  -\alpha^{-1}_{1s}b^{(s)}_{jl} \Big) +\Big(y_1 -\alpha^{-1}_{1s}y_s \Big).
\end{equation}
We know that $\gamma_{ (0 , \hat x_j, y)}$ is the minimal curve associated to $\gamma _{j1}$ satisfying on clos$(\mathcal O )$ the ODE for characteristics $\eqref{vector}$ and consequently 
\begin{equation}\label{8nov8}
\gamma_{ (0 , \hat x_j, y)} (t)\leq \hat \gamma^1_{(0 , \hat x_j, y)}(t),
\end{equation}
for all $t>0$. Now using $\eqref{relazionegamma.1},$ it follows
\begin{equation*}\label{8nov3}
\gamma_{ (0 , \hat x_j,y)} (t)= \alpha^{-1}_{1s}\gamma^s_{ (0 , \hat x_j, y)}  (t)+ \frac{1}{2} t\, \sum_{l=2}^{m}x_l \Big(b^{(1)}_{jl}  -\alpha^{-1}_{1s}b^{(s)}_{jl} \Big) +\Big(y_1 -\alpha^{-1}_{1s}y_s \Big),
\end{equation*}
for each $t>0$ and, consequently, recall that $\alpha_{1s}>0$ and putting together $\eqref{8nov2}$ and $\eqref{8nov8}$ with $t=r$ we obtain that
\begin{equation*}
 \gamma^s_{ (0 , \hat x_j, y)}(r)\leq \hat \gamma^s_{(0 , \hat x_j, y)} (r),
\end{equation*}
Hence $\eqref{assurdo1}$ holds, as wished. In the similar way, we can prove the other cases.

The last step is to prove that $\gamma^s_{ (0 , \hat x_j, y)}$ is non decreasing in $y _s$. If $b^{(s)}_{j1} = 0$, it is clear. On the other hand, if $b^{(s)}_{j1} \ne 0$, we can use the same argument in Step 1.

%Let $t_1>t_2\geq 0$. 
%%\begin{equation}\label{assurdo2}
%%\gamma^s_{ (\bar t , \hat x_j, \bar \tau)}(t_1)\geq \gamma^s_{ (\bar t , \hat x_j, \bar \tau)}(t_2).
%%\end{equation}
%We know that  $\gamma_{ (\bar t , \hat x_j, \bar \tau)}(t_1)\geq \gamma_{ (\bar t , \hat x_j, \bar \tau)}(t_2)$ and consequently recall that $a_{1s}>0$ and using again $\eqref{relazionegamma}$ if $ \frac{1}{2} \, \sum_{i=2}^{m}(\hat x_j )_i \Big(b^{(1)}_{ji}  -a^{-1}_{1s}b^{(s)}_{ji} \Big) <0$, we have
%\begin{equation*}
%\begin{aligned}
%a^{-1}_{1s} \gamma^s_{ (\bar t , \hat x_j, \bar \tau)}  (t_1) +\frac{1}{2} t_1\, \sum_{i=2}^{m}(\hat x_j )_i \Big(b^{(1)}_{ji}  -a^{-1}_{1s}b^{(s)}_{ji} \Big)
%& \geq a^{-1}_{1s} \gamma^s_{ (\bar t , \hat x_j, \bar \tau)}  (t_2) +\frac{1}{2} t_2\, \sum_{i=2}^{m}(\hat x_j )_i \Big(b^{(1)}_{ji}  -a^{-1}_{1s}b^{(s)}_{ji} \Big) \\
%& \geq a^{-1}_{1s} \gamma^s_{ (\bar t , \hat x_j, \bar \tau)}  (t_2) +\frac{1}{2} t_1\, \sum_{i=2}^{m}(\hat x_j )_i \Big(b^{(1)}_{ji}  -a^{-1}_{1s}b^{(s)}_{ji} \Big)
%\end{aligned}
%\end{equation*}
%and so
%\begin{equation}\label{assurdo2}
%\gamma^s_{ (\bar t , \hat x_j, \bar \tau)}(t_1)\geq \gamma^s_{ (\bar t , \hat x_j, \bar \tau)}(t_2),
%\end{equation}
%as desired. On the other hand if $ \frac{1}{2}\, \sum_{i=2}^{m}(\hat x_j )_i \Big( b^{(s)}_{ji}-a_{1s} b^{(1)}_{ji} \Big) \geq 0$, $\eqref{assurdo2}$  is also satisfied. Consequently $\Gamma_{ (\bar t , \hat x_j, \bar \tau)} $ satisfying $\eqref{minode1}$. It is similar for $\Gamma ^{(\bar t,\hat x _j,\bar \tau )}$.

\end{proof}

%\begin{prop}\label{lemma4.2bcsc} 
%For each $j=2,\dots ,m$ and $\hat x _j \in (0,1)^{m-2}$ fixed, set if $n>1$:
%\begin{itemize}
%\item if $b^{(s)}_{j1}= 0$, $K_s \in \N$ such that $ \min _{K_s \in \N} K_s -2|\frac{1}{2}  \sum_{l=2}^{m}b^{(s)}_{jl}x_l | >0$;
%\item if $ b^{(s)}_{j1}/b^{(\bar s)}_{j1} >0$ and $b^{(\bar s)}_{j1}\ne 0$, $K_s \in \N$ such that $\min _{K_s \in \N} K_s -2 -2|\frac{1}{2} \sum_{l=2}^{m}x_l \left( b^{(s)}_{jl}-\alpha_{1s} b^{(1)}_{jl} \right)| >0$;
%\item if $ b^{(s)}_{j1}/b^{(\bar s)}_{j1}<0$ and $b^{(s)}_{j1}\ne 0$,  $K_s \in \N$ such that $\min _{K_s \in \N} K_s -2 -2|\frac{1}{2} \sum_{l=2}^{m}x_l \left( b^{(s)}_{jl}-\alpha_{1s} b^{(1)}_{jl} \right)| >0$.
%\end{itemize}
%On the other hand, set $K_1=0$ if $n=1.$
%
%Then, there exists a full Lagrangian pa\-ra\-me\-teri\-zation associated to a continuous function $\phi :(0, 1)^{m-1} \times (0, K_1+1)\times \dots \times (0,K_n+1) \to \R$ which is also continuous on the closure $[0, 1]^{m-1} \times [0, K_1+1]\times \dots \times [0,K_n+1]$. 
%\end{prop} 

In Proposition \ref{AppendiceA.1} we show how to make a partial parameterization $\chi$  surjective. %thus we cover $O$ by a family of characteristic curves which includes the ones of $\chi $. 
In the following lemma, w.l.o.g. in a simpler setting, we provide instead a full Lagrangian  type parameterization, defined at once instead of extending a given one.

\begin{prop}\label{lemma4.2bcsc} 
Let $\G$ be a Carnot group of step $2$ and $\B= \max\{ b^{(s)}_{j\ell}\,:\, s=1,\dots , n , j,\ell =1,\dots , m\}.$ Then, if $n>1,$ there exist a suitable open subset $U$ of $\R^{m+n-1}$ which depends on $\B$ and a full Lagrangian type  pa\-ra\-me\-teri\-zation associated to a continuous function $\phi :U \to \R$ which is also continuous on the closure of $U$. On the other hand, if $n=1$, the statement holds with $U=(0,1)^{m}.$
\end{prop}

\begin{proof}We split the proof in several steps. In the first three steps we refine the technique used in \cite[Lemma 4.2]{biblio17} in $\HH^n$; while the last step is the main difference w.r.t. the case of Heisenberg groups.

$\mathbf{At\,\, first,\,\, we \,\, consider\,\,  \G \,\,  a \,\, corank\,\,  1\,\, Carnot\,\, group}$ as in Example \ref{exaCorank1}. Fix $j=2,\dots , m$ and $\hat x_j\in [0,1]^{m-2}$. For simplicity, we write $ b_{j1}$  instead of  $ b^{(1)}_{j1}$ and we choose $x_j=0$. We define the map $$f_1(r, \gamma _j(r) ) := b_{j1}\phi (r,\hat x_j, \gamma _j(r))+ \frac{1}{2}\sum_{l=2}^{m}b_{jl}x_l,$$ where $\phi (\cdot , \hat x _j , \cdot)$ be a continuous function on $[0,1] \times [0,1]$. We want to give a Lagrangian type  pa\-ra\-me\-teri\-zation for its restriction to $(0,1)^2$ as we define it on an open set. 
 
 $\mathbf{Step \, 1.}$ We show that we can assume $f_1$ is compactly supported in $(t,y)\in [0,1]\times (0,1)$. On the contrary, one can extend it to a compactly supported function $\bar f_1$ on $[0,1]\times (-1,2)$: restricting the Lagrangian type  parameterization $(\tl {\mathcal O} ^{\bar f_1}, \chi ^{\bar f_1})$ for $\bar f_1$, defined as described below, to the open set
\begin{equation*}
\tl {\mathcal O} :=\{(t,\hat x_j,y ) \in \tl {\mathcal O} ^{\bar f_1} \, :\, \chi ^{\bar f_1}_j (t,\hat x_j, y) \in (0,1)\},
\end{equation*}
one will get a Lagrangian  type parameterization for $f_1$. The assumption of $y$ compactly supported in $(0,1)$ implies that there are two characteristics, one starting from $(0,0)$ and one from $(0,1)$ which satisfy $$ \dot \gamma _j(t)= b_{j1}\phi (t,\hat x_j, \gamma _j(t))+ \frac{1}{2}\sum_{l=2}^{m}b_{jl}x_l \equiv 0.$$
%where $\gamma (t):=\gamma _{j1} (t,\hat x_j)$. 
This means respectively $\gamma _j(t)\equiv 0$ and $\gamma _j(t)\equiv 1$ for each $t\in [0,1]$.

After this simplification, we associate to each point $(\bar t ,\hat x_j, \bar y )$ with $ (\bar t , \bar y )\in [0,1]^2$ a curve $\gamma _j$ which is minimal forward in $t$ and maximal backward:
 \begin{equation}\label{curvadef7}
\gamma _j(t,\bar t, \bar y ) =  \left\{
\begin{array}{l}
\gamma_{ (\bar t, \hat x_j, \bar y)} (t ), \quad \mbox{for } t\geq \bar t,
\\
\quad
\\
\gamma^{ (\bar t, \hat x_j, \bar y)} (t ),  \quad \mbox{for } t< \bar t,
\end{array}
\right.
\end{equation}
where $\gamma_{ (\bar t, \hat x_j, \bar y)} (t ) $ and $\gamma^{ (\bar t, \hat x_j, \bar y)} (t ) $ were defined in $\eqref{ode1}$ and $\eqref{ode2}$. Observe that the curve $\gamma _j (\cdot , \bar t , \bar y)$ is defined on the whole $[0,1]$.

$\mathbf{Step \, 2.}$  Now we consider the set
\begin{equation*}
\mathcal{C}:= \{ \gamma _j(\cdot , \bar t , \bar y )\, :\, [0,1]\to [0,1], \, (\bar t , \bar y ) \in [0,1]^2 \}.
\end{equation*}
We will endow $\mathcal{C}$ by the topology of uniform convergence on $[0,1]$ and the following total order relation
\begin{equation}\label{relazioneordine}
\gamma _j(\cdot , t _1, y _1) \preceq \gamma _j(\cdot ,  t_2 , y_2) \quad \Longleftrightarrow \quad \gamma _j(t , t _1, y_1) \leq \gamma _j(t,  t_2 , y_2), \quad \forall t \in [0,1]
\end{equation}

Let us denote by $\mathcal{C}^*$ the closure of $\mathcal{C} \subset C([0,1])$ endowed with the topology of uniform convergence and satisfy the following properties:
\begin{itemize}
		\item[(a)] $\mathcal{C}^*$ is compact;
		\item[(b)] The total order relation  \eqref{relazioneordine}  still applies in $\mathcal{C}^*$;
		\item[(c)] $\mathcal{C}^*$ is connected;
		\item[(d)] $\mathcal{C}^*$  is still a family of characteristic curves for $f.$
\end{itemize}
That means for each $\gamma _j (\cdot , t _1, y _1) \in \mathcal{C}^* $ we have $\dot \gamma _j (t)=b_{j1}\phi (t,\hat x_j, \gamma _j(t))+ \frac{1}{2}\sum_{l=2}^{m}b_{jl}x_l.$ %with $\gamma _j (t_1)=y_1$ for all $t\in [0,1]$.

$\mathbf{(a).}$ Since $\mathcal{C}$ is a family of equi-Lipschitz continuous and bounded functions of $C([0,1])$, then, from Arzel\'a-Ascoli's Theorem $\mathcal{C}^*$  is compact, i.e., (a) holds.

$\mathbf{(b).}$ Let $\gamma _j, \gamma '_j \in \mathcal{C}^*$, we would like to show that $\gamma _j \preceq \gamma '_j $ or $\gamma '_j \preceq \gamma _j$. By definition there are two sequence $\{ \gamma _j^h\}_h, \{ \gamma _j^{'h}\}_h \subset \mathcal{C}$ such that $\gamma _j^{h} \to \gamma _j$ and $\gamma _j^{'h} \to \gamma '_j$ uniformly in $[0,1]$. If $ \gamma _j  \ne \gamma '_j$, then there is $t_0 \in [0,1] $ such that
\begin{equation*}
\gamma _j (t_0) <\gamma '_j (t_0) \quad \mbox{or} \quad \gamma '_j (t_0)< \gamma _j (t_0).
\end{equation*}
Suppose $\gamma _j (t_0) <\gamma '_j (t_0) $. Then we will show that 
\begin{equation}\label{bcsc4.7}
\gamma _j (t) \leq \gamma '_j (t) \quad \mbox{for all } t\in [0,1].
\end{equation}
Let $0<\varepsilon < \frac{ \gamma _j(t_0)-\gamma ' _j(t_0) }{2}$, these is $\bar h=\bar h(\varepsilon) \in \N$ such that
\begin{equation*}
|\gamma _j(t) -\gamma _j^{h} (t)| <\varepsilon \quad \mbox{ and }  \quad |\gamma '_j(t) -\gamma _j^{'h} (t)| <\varepsilon  \quad  \mbox{for all } t\in [0,1], h> \bar h. 
\end{equation*}
As a consequence, 
\begin{equation*}
\gamma _j^{h} (t_0) <\gamma _j (t_0) +\varepsilon < \gamma '_j (t_0) -\varepsilon < \gamma _j^{'h} (t_0) \quad  \mbox{for all } h> \bar h,
\end{equation*}
and so, recall that $\gamma _j^{h}, \gamma _j^{'h}$ are ordered, we get that $\gamma _j^{h} (t) \leq  \gamma _j^{'h} (t)$ for all $t\in [0,1]$, $ h> \bar h$. The inequality \eqref{bcsc4.7} holds passing to the limit as $h\to \infty.$

$\mathbf{(c).}$ By contradiction, we suppose that $\mathcal{C}^* =\mathcal{C}^*_1 \cup \mathcal{C}^*_2$ and $\mathcal{C}^*_1 \cap \mathcal{C}^*_2= \emptyset$, where $\mathcal{C}^*_1, \mathcal{C}^*_2$ are non empty, closed sets in $C([0,1])$. It is well-known that, from (a) and (b), for each subset $\mathcal{A} \subset \mathcal{C}^*$ there exists the least upper bound (or supremum) $\sup \mathcal{A}$ and greatest lower bound (or infimum) $\inf \mathcal{A}$ of $\mathcal{A}$. More precisely, we have 
\begin{equation}
\gamma _{l_1}:= \inf \mathcal{C}^*_1 \leq \gamma _{L_1}:= \sup \mathcal{C}^*_1 \qquad \mbox{and} \qquad \gamma _{l_2}:= \inf \mathcal{C}^*_2 \leq \gamma _{L_2}:= \sup \mathcal{C}^*_2,
\end{equation}
where  $\gamma _{l_i}, \gamma _{L_1} \in \mathcal{C}^*_i$ for $i=1,2$ because $\mathcal{C}^*_1$ and $\mathcal{C}^*_2$ are closed. Moreover since $\mathcal{C}^*_1 \cap \mathcal{C}^*_2= \emptyset$, it follows  $\gamma _{L_1} \preceq \gamma _{l_2}$ or $\gamma _{L_2} \preceq \gamma _{l_1}$. If we assume, for instance, that $\gamma _{L_1} \preceq \gamma _{l_2},$ then $\gamma _{L_1}(t)\leq  \gamma _{l_2} (t)$ for all $t\in [0,1]$ and $\gamma _{L_1}(t_0) <  \gamma _{l_2} (t_0)$  for a suitable $t_0 \in [0,1]$. Now we obtain a contradiction because if we put 
\begin{equation}
\bar t:= t_0, \quad \bar y :=\frac{\gamma _{L_1}(t_0)+\gamma _{l_2}(t_0) }{2}, \quad \gamma (t):=\gamma _j (t, \bar t , \bar \tau), \mbox{ for } t\in [0,1],
\end{equation}
then by definition $\gamma (t) \in \mathcal{C} \subset \mathcal{C}^*,$ but $\gamma _{L_1} \preceq \gamma \preceq  \gamma _{l_2},$ i.e.,$\gamma $ doesn't belong to $\mathcal{C}^*_1 \cup \mathcal{C}^*_2 =\mathcal{C}^*$.

$\mathbf{(d).}$  Let $\gamma _j \in \mathcal{C}^*$, then by definition there is a sequence $\{\gamma _j^{h}\}_h \subset  \mathcal{C}$ such that $\gamma _j^{h} \to \gamma _j$ uniformly in $[0,1]$. Moreover
\begin{equation*}
\gamma _j^{h} (t) -\gamma _j^{h} (\bar t) = b_{j1} \int _0^{t-\bar t} \phi (r,\hat x_j, \gamma _j^{h}(r))\, dr+ \frac{1}{2} (t-\bar t)\sum_{l=2}^{m} b_{jl}x_l ,
\end{equation*}
and so passing to the limit as $h\to \infty $ in the previous identity we get the last claim (d).

$\mathbf{Step \, 3.}$ Now we are able to give a full Lagrangian  type parametrization associated to $\phi.$ In order to do it, we consider the map $\theta _j:\mathcal{C}^* \to \R$ given by
\begin{equation*}
\theta _j (\gamma _j) :=\sum_{l=0}^\infty \frac{1}{2^l} \gamma _j(r_l),
\end{equation*}
where $(r_l)_{l\in \N}$ is an enumeration of $\mathbb{Q} \cap [0,1]$. Notice that $\theta _j$ satisfies the following properties:
\begin{itemize}
		\item[(a.1)] By definition, $\theta _j$ is continuous.
		\item[(b.1)] By definition,  $\theta _j$ is strictly order preserving, that is $\theta _j(\gamma _j) < \theta _j(\gamma ' _j) \, \mbox{ if } \,  \gamma _j  \prec \gamma ' _j.$ %this follows from the definition of $\theta _j.$
		\item[(c.1)] $\theta _j (\mathcal{C}^* ) =[0,2];$ indeed, noticing that $\theta _j (\gamma _j(\cdot , 0, 0))=0$ and $\theta _j (\gamma _j (\cdot , 0, 1))=2,$ the equality follows from (c), (a.1) and (b.1). 
		\item[(d.1)] Using  (a), (b.1) and (c.1), we have that there exists  $\theta^{-1} _j :[0,2]\to \mathcal{C}^*$  continuous.
\end{itemize}

Therefore for each $j=2,\dots , m$ and $\hat x_j \in [0,1]^{m-2}$ fixed, we consider the map $\chi _{\hat x_j}: [0, 1] \times [0, 2] \to [0, 1]$ defined as
\begin{equation}
\chi _{\hat x_j}(t, y):= \theta _j^{-1} (y) (t ) \qquad \mbox{for } (t,y)\in [0, 1] \times [0, 2].
\end{equation}
which is  a continuous function because of (d.1). Hence, we can define $\chi _j: [0,1]^{m-1} \times [0,2] \to [0,1]$ as $$\chi _j (t,\hat x_j ,y):=\chi _{\hat x_j} (t,y).$$  
Finally, for all $j=2,\dots , m$  the map $\Upsilon_j : (0, 1)^{m-1} \times (0, 2) \to (0, 1)^{m},$  given by $$\Upsilon_j(t,\hat x_j, y ) := (t,\hat x_j, \chi _j(t,\hat x_j, y)),$$ turns out to be a full Lagrangian  type parameterization associated to $\phi: (0, 1)^m\to \R,$ as desired.
\medskip

$\mathbf{Step \, 4.}\, \mathbf{We\,\, consider\,\, the\,\, case\,\, n>1.}$ Fix $j=2,\dots , m$ and $\hat x _j\in [0,1]^{m-2}$. If $b^{(s)}_{1j}= 0$ for some $s=1,\dots, n$ it is evident that \begin{equation}\label{GammaovvioDIC}
\gamma _{js} (t)=\frac{1}{2}( t-\bar t) \sum_{l=2}^{m}b^{(s)}_{jl} x_l  + \bar y_s.
\end{equation}

As a consequence, we consider the curves $\gamma _{js} $ when $b^{(s)}_{j1} \ne 0$.  

Thanks to Remark \ref{remChange}, without loss of generality, we can suppose that  $b^{(s)}_{j1}> 0$ for $s=1,\dots , n$ and we also assume that $b^{(1)}_{j1} =\max \{b^{(s)}_{j1}\,:\, s=1,\dots ,n \}$. As in Step 2 of Lemma \ref{lemma4.1bcsc}, the key observation is that we can reduce $\gamma _{j2}, \dots , \gamma _{jn}$ satisfying $ \dot \gamma _{js} (t)= b^{(s)}_{j1}\phi (r,\hat x_j, \gamma _j(t))+ \frac{1}{2}\sum_{l=2}^{m}b^{(s)}_{jl}x_l $ with $\gamma _{js} (\bar t)=\bar y _s$ as
\begin{equation}\label{infunzgammaDIC}
\begin{aligned}
\gamma _{js}(t)= \alpha_{1s}\gamma _{j1} (t) + \frac{1}{2} (t-\bar t)\, \sum_{l=2}^{m}x_l \left( b^{(s)}_{jl}-\alpha_{1s} b^{(1)}_{jl} \right) +\left(\bar y _s -\alpha_{1s} \bar y_1\right), %\qquad \mbox{for all } s=2,\dots, n
\end{aligned}
\end{equation}
for $s=2,\dots, n$ such that $b^{(s)}_{j1} \ne 0,$ where $\alpha_{1s}:=\frac{b^{(s)}_{j1}}{b^{(1)}_{j1}} \in (0,1]$. Hence the function $\phi $ depends only on $\gamma _{j1}$ and we can apply the case $n=1$. Consequently, for each $j=2,\dots , m$ and $\hat x_j \in [0,1]^{m-2}$ fixed, there is a map $\chi _{j1}(t,\hat x_j, y) $ associated to $\gamma _{j1}$. 

Then we define the map $\Upsilon_j : \tl {\mathcal O _j} \to \R^{m+n-1}$ given by $$\Upsilon _j(t,\hat x_j, y) := (t,\hat x_j, \chi _{j1}(t,\hat x_j, y), \chi _{j2}(t,\hat x_j, y ),\dots, \chi _{jn}(t,\hat x_j, y)),$$
where the curve $\chi _{js}( t , \hat x_j, \tau)$ is defined as $\eqref{GammaovvioDIC}$ if $b^{(s)}_{1j}= 0$ or
\begin{equation*}
\begin{aligned}
\chi _{js}( t , \hat x_j, y)= \alpha_{1s} \chi _{j1}( t , \hat x_j,y) + \frac{1}{2} (t-\bar t)\, \sum_{l=2}^{m}x_l \left( b^{(s)}_{jl}-\alpha_{1s} b^{(1)}_{jl} \right) +\left(\bar y_s -\alpha_{1s} \bar y_1\right), %\qquad \mbox{for all } s=2,\dots, n
\end{aligned}
\end{equation*}
 if $b^{(s)}_{j1} \ne0$. Now the idea is that for a suitable choice of $y_s,$ we have that 
 \begin{equation*}
\chi _{js}( t , \hat x_j, y) \in (0,K(\B))\quad  \mbox{where}\quad K(\B)>0 \quad not \mbox{ depends on $j.$}
\end{equation*}

 More precisely, recall that $\B= \max\{ b^{(s)}_{j\ell}\,:\, s=1,\dots , n , j,\ell =1,\dots , m\}$ is larger than zero because the matrices are skew- symmetric, we define an open set  $$ \tl {\mathcal O _j} =\{ (t,\hat x_j, y) \in (0,1)\times (0,1)^{m-2}\times \R^{n} \}$$ as follows:
\begin{itemize}
\item $y_1\in (0,2);$
\item if $b^{(s)}_{j1}= 0$, then $y _s \in \left(|\frac{1}{2}  \sum_{l=2}^{m}b^{(s)}_{jl}x_l |,   (m-1) \B +1 -|\frac{1}{2}  \sum_{l=2}^{m}b^{(s)}_{jl}x_l |\right)$; %where $K_s \in \N$ such that $ \min _{K_s \in \N} K_s -2|\frac{1}{2}  \sum_{l=2}^{m}b^{(s)}_{jl}x_l | >0$;
\item  if $b^{(s)}_{j1}\ne 0$, then \\  
$y _s \in \left(| \frac{1}{2} \sum_{l=2}^{m}x_l \left( b^{(s)}_{jl}-\alpha_{1s} b^{(1)}_{jl} \right) | +2, (m-1)\B +3 - \alpha_{1s} -|\frac{1}{2} \sum_{l=2}^{m}x_l \left( b^{(s)}_{jl}-\alpha_{1s} b^{(1)}_{jl} \right)| \right).$ %where $K_s \in \N$ such that $\min _{K_s \in \N} K_s -2 -2|\frac{1}{2} \sum_{l=2}^{m}x_l \left( b^{(s)}_{jl}-\alpha_{1s} b^{(1)}_{jl} \right)| >0$;
%\item if $ \alpha_{1s} <0$ and $b^{(s)}_{j1}\ne 0$, then $y _s \in (1+ | \frac{1}{2} \sum_{l=2}^{m}x_l \left( b^{(s)}_{jl}-\alpha_{1s} b^{(1)}_{jl} \right) |, K_s -1- |\frac{1}{2} \sum_{l=2}^{m}x_l \left( b^{(s)}_{jl}-\alpha_{1s} b^{(1)}_{jl} \right)|)$, where $K_s \in \N$ such that $\min _{K_s \in \N} K_s -2 -2|\frac{1}{2} \sum_{l=2}^{m}x_l \left( b^{(s)}_{jl}-\alpha_{1s} b^{(1)}_{jl} \right)| >0$.
\end{itemize}

It is easy to see that $\Upsilon_j $ is onto and, in particular, $$\Upsilon_j ( \tl {\mathcal O _j}) = (0, 1)^{m-1} \times (0,1)\times (0, \B_2)\times \dots \times (0,\B _n),$$ where 
 \begin{equation}\label{curvadef7}
\B_s =  \left\{
\begin{array}{l}
(m-1)\B +1, \quad \mbox{if } b^{(s)}_{j1}= 0,
\\
 (m-1)\B + 3,  \quad \mbox{if } b^{(s)}_{j1}\ne 0.
\end{array}
\right.
\end{equation}

Hence, if we put $$ U:=  (0, 1)^{m-1} \times (0,1)\times (0, \B_2)\times \dots \times (0,\B _n),$$ then for each $j=2,\dots, m$, we proved that there is $\Upsilon _j :\tl {\mathcal O _j} \to U$ which is a full Lagrangian type  parameterization associated to $\phi :U\to \R$, as desired. 

\end{proof}

\section{Intrinsic Lipschitz map vs. Lagrangian type  solution}
In this section we prove the main theorem \ref{teoremaFinale1} of this paper, i.e.,
\begin{equation*}
\begin{aligned}
\phi \mbox{ is intrinsic Lipschitz} & \quad  \Longleftrightarrow  \quad \phi \mbox{ is a Lagrangian type  solution of } \eqref{solLagrangiana}. \\
%\phi \mbox{ is Distributional solution of } \eqref{solLagrangiana} & \quad \Longleftrightarrow \quad   \phi \mbox{ is Lagrangian type  solution of } \eqref{solLagrangiana}\\
\end{aligned}
\end{equation*}
Actually, we will show that
\begin{equation*}
\begin{aligned}
%\phi \mbox{ is intrinsic Lipschitz} & \quad  \Longleftrightarrow  \quad \phi \mbox{ is Lagrangian type  solution of } \eqref{solLagrangiana} \\
\phi \mbox{ is a Distributional solution of } \eqref{solLagrangiana} & \quad \Longleftrightarrow \quad   \phi \mbox{ is a Lagrangian type  solution of } \eqref{solLagrangiana}\\
\end{aligned}
\end{equation*}
when $\phi :\W \to \V$ is locally $1/2 $-H\"older continuous along the vertical components. We underline that it is not trivial fact that a weak solution of \eqref{solLagrangiana}  is locally $1/2 $-H\"older continuous along the vertical components; indeed, for instance, in the context of step 2 Carnot groups, this problem is open for continuous datum $w$ when $\V$ has dimension larger than one. On the other hand, in \cite{biblioKOZHEVNIKOV}, the author prove that we can not drop the H\"older regularity when $\G$ is a Carnot group of step 3. Finally, in \cite{biblioADDD}, we give a positive answer in step 2 with $\V$ 1-dimensional. 

In order to obtain Theorem \ref{teoremaFinale1}, we restrict to a subclass of Carnot groups of step 2 defined in Section \ref{Setting}.  We underline that thanks to the "good" structure of our intrinsic derivatives in this  "good"  subclass of step 2, we can reduce the study of PDE system \eqref{solLagrangiana} to ODE (see Proposition \ref{Lemma6.2bcsc}).

Roughly speaking, an
$\mathcal{L}^\infty$-representative $w$ for the datum of the ODEs \eqref{equationODECONBARW}  provided by taking the $t$-derivative of $\phi(t,\hat x_j, \chi _{j1}(t,\hat x_j,y),\dots ,  \chi _{jn}(t,\hat x_j,y))$, which is $\phi$ evaluated along the characteristics $\chi _j (t, \tau)$ of the Lagrangian type parameterization; by construction, it
coincides with the second $t$-derivative of $\chi _{j\bar s}(t,\hat x_j,y)$ where $\bar s$ is the only one index such that $b_{j1}^{(\bar s)} \ne 0$. Here we have hidden the fact that we need to come back from $(t,y)$
to $(t,\tau)$, a change of variable that is not single valued, not surjective because the second derivative was defined only
almost everywhere, and which may map negligible sets into positive measure sets. We overcome the difficulty
showing that it is enough selecting any value of the second derivative when present. However, if one changes the set
of characteristics in general one arrives to a different function $\bar w\in \mathcal{L}^\infty$, which however identifies the same
distribution as $w$ %given by distributional solution Theorem \ref{teoremaFinale1} (2) 
 (see Theorem \ref{teoremaFinale1} $(2) \implies  (3)$).

\subsection{Setting}\label{Setting} We consider $\G = (\R^{m+n}, \cdot  , \delta_\lambda )$ a Carnot group of step $2$ where the group operation is given by \eqref{5.1.0} for $m\times m$ real matrices $\mathcal{B}^{(1)}, \dots , \mathcal{B}^{(n)}$ which satisfy the following property if $n>1$:
\begin{equation}\label{avereunaphisola}
b_{j1}^{(s)} b_{j1}^{(k)} =0, \quad \mbox{for all } s,k=1,\dots, n \mbox{ and } s\ne k.
\end{equation}
%Indeed by Remark \ref{remChange} with $\mathcal{M}_1=\mathcal{I}_m$ (that is the $m\times m$ identity matrix) and $\mathcal{M}_2$ a suitable lower triangular matrix, %(that is a  matrix such that $c_{sk}=0$ if $k>s$),  it is possible consider $(\R^{m+n}, \star, \delta_\lambda)$ a Carnot groups of step 2 isomorphic to $\G = (\R^{m+n}, \cdot, \delta_\lambda)$ where the corresponding matrices $\mathcal{ \tilde B}^{(1)}, \dots , \mathcal{\tilde B}^{(n)}$ satisfy the  condition \eqref{avereunaphisola}. 
Here for each fixed $j=2,\dots ,m$ and $(x_j,\hat x_j)\in \R^{m-1}$ the vertical characteristic line $\gamma _{j}=(\gamma _{j1}, \dots , \gamma _{jn}):[-\delta , \delta ]\to \R^n$ satisfying
\begin{equation*}\label{carattere2}
\begin{aligned}
\dot \gamma _{js} (t) & = b_{j1}^{(s)} \phi (x_j+t,\hat x_j, \gamma _j(t))+ \frac{1}{2}\sum_{l=2}^{m}b^{(s)}_{jl}x_l , \quad \mbox{at most one } \bar s \in \{1,\dots , n \},\\
\dot \gamma _{js} (t) & =  \frac{1}{2}\sum_{l=2}^{m}b^{(s)}_{jl} x_l, \quad \mbox{for all } s=1,\dots, n  \mbox{ with } s\ne \bar s.  \\ \\
\end{aligned}
\end{equation*}
Consequently, we have that every vertical component is a strict line except one and this fact allows us to "immerse" in the context of Heisenberg groups where we just have a non linear component. Therefore, it is natural that the strategy will be to adapt the technique valid in $\HH^n$ proved in \cite{biblio17} to groups satisfy \eqref{avereunaphisola} .

We conclude this section given some examples of the subclass of Carnot groups studied.
\begin{exa}
Corank 1 Carnot groups, and so also Heisenberg groups, as in Example \ref{exaCorank1} belong to our setting because $n=1$.
\end{exa}

\begin{exa}
Free Carnot groups of step 2 as in Example \ref{exaFree Step 2 Groups} satisfy the condition \eqref{avereunaphisola}.
\end{exa}

\begin{exa}
The complexified Heisenberg group as in Example \ref{exacomplexified Heisenberg group} satisfies the condition \eqref{avereunaphisola}.
\end{exa}

 \subsection{The main theorem}  From now on, $\G $ will be as in the setting \ref{Setting} and we will take $x_j=0$ for simplicity. %In this section we prove the main theorem of this paper, i.e.,

\begin{theo}\label{teoremaFinale1} 
Let $\G $ as in the setting $\ref{Setting}$, let $\V$, $\W$ be complementary subgroups defined in \eqref{5.2.0} and let $U\subset \R^{m+n-1}$ as in Proposition $\ref{lemma4.2bcsc}$. Let $\hat \phi $ be a continuous map  where $\phi :U\to \R$ is the map associated to $\hat \phi$ as in $\eqref{phipsi}$. If $\phi$ is locally $1/2 $-H\"older continuous along the vertical components, then the following statements are equivalent:
%Let $\phi :[0,1]^{m+n-1}\to \R$ be a  continuous function and locally $1/2 $-H\"older continuous along the vertical components. Then the following statements are equivalent:
\begin{enumerate}
\item  $\hat \phi $ is a locally intrinsic Lipschitz function.
\item $\phi$ is a continuous distributional solution of $(D^{\phi}_2\phi,\dots, D^{\phi}_m\phi )  =w$ in $U $ with $w\in \mathcal{L}^\infty _{loc}(U ,\R^{m-1})$.
\item $\phi$ is a Lagrangian type  solution to the equation  $(D^{\phi}_2\phi,\dots, D^{\phi}_m\phi )  = \bar w$  in $U$ for a bounded function $\bar w :U \to \R^{m-1}$ such that $\bar w =w$  $\mathcal{L}^{m+n-1}$-a.e. on $U$.
\end{enumerate}

\end{theo}

\begin{rem}
We don't know if Theorem \ref{teoremaFinale1} $(2) \Leftrightarrow (3)$ is true in general Carnot group of step 2. We underline that  Lagrangian type  solution  implies a suitable monotonicity of the vertical components of  the maps $ \chi _{js}$ and we don't know if the following results as Proposition \ref{propcontinuita}, \ref{AppendiceA.1} and \ref{Lemma6.2bcsc} preserve these properties in general Carnot group of step 2.
\end{rem}

\begin{rem}
In Example 1.3 in \cite{biblio17}, the authors show that the datum $w$ in (2) in general is not equal to the datum $\bar w$ in (3); but, in general, we have that $\bar w = w$ almost everywhere on $U$.
\end{rem}

The proof of Theorem \ref{teoremaFinale1} requires on the following statements
\begin{lem}[Lemma 7.3, \cite{biblioDDD2}]\label{lemma5.1bcsc} 
Let $\G = (\R^{m+n}, \cdot  , \delta_\lambda )$ be a Carnot group of step $2$ and $\V$, $\W$ the complementary subgroups defined in \eqref{5.2.0}. Let $\hat \phi :\hat{\mathcal O} \to \V $ be a continuous map  where $\hat{\mathcal O} $ is open in $\W$ and $\phi :\mathcal O  \to \R$ is the map associated to $\hat \phi$ as in $\eqref{phipsi}$.  We also assume that 
\begin{enumerate}
\item $\phi$ is a continuous distributional solution of $(D^{\phi}_2\phi,\dots, D^{\phi}_m\phi )  =w$ in $\mathcal O $ with $w\in \mathcal{L}^\infty _{loc}(\mathcal O ,\R^{m-1})$
\item $\phi$ is locally $1/2 $-H\"older continuous along the vertical components %with H\"older's constant $C_h>0$
\end{enumerate}
Then, for all $j=2,\dots , m$ there is $C=C(\| w_j\|_{\mathcal{L}^\infty (\mathcal O, \R )}, b_ {j1 }^{(s)})>0$ such that $\phi$ is a C-Lipschitz map along any characteristic line $\gamma _{j}=(\gamma _{j1}, \dots , \gamma _{jn}):[-\delta , \delta ]\to \R^n$ satisfying
\begin{equation*}\label{carattere}
\dot \gamma _{js} (t)= b_{j1}^{(s)} \phi (t,\hat x_j, \gamma _j(t))+ \frac{1}{2}\sum_{l=2}^{m}b^{(s)}_{jl} x_l, \quad \mbox{ for all } s=1,\dots, n,
\end{equation*}
with $t\in [-\delta , \delta ]$ and $\hat x_j \in  \R^{m-2}$ fixed. %In particular if $n=1$ then $C=\| w_j\|_{\mathcal{L}^\infty (\mathcal O )}$.
\end{lem}

About regularity of a full Lagrangian type  parametrization, we have the following result:
\begin{prop}\label{propcontinuita}
For all $j=2,\dots , m$, let $\Upsilon_j : \tl {\mathcal O _j} \to \mathcal O$ be a full Lagrangian type  parametrization  of a continuous map $\phi :  \mathcal O \to \R$. Then for a given $ ( \hat x_j, \hat y_{ s} )\in \R^{m-2} \times \R^{n-1},$ the map $ \chi _{js} = \chi _{js} (\cdot  , \hat x_j,\cdot , \hat y_{ s}): \tl {\mathcal O}_{j ( \hat x_j, \hat y_{ s} )} \to \R $  is continuous  for every $s=1,\dots , n$.
%Then $ \chi _{js}$ is smooth for every $s=1,\dots , n$ except for at most one $\bar s$ and, in this case, $ \chi _{j\bar s} = \chi _{j\bar s} (\cdot  , \hat x_j,\cdot , \hat y_{\bar s}): \tl {\mathcal O}_{j ( \hat x_j, \hat y_{\bar s} )} \to \R $  is continuous. % {\mathcal O}_{j (t, \hat x_j, \hat y_{\bar s} )}$ is continuous. 
\end{prop}

\begin{proof}
%Without loss of generality, we consider $\G = (\R^{m+n}, \cdot  , \delta_\lambda )$ a Carnot group of step $2$ where the group operation is given by \eqref{1} for $m\times m$ real matrices $\mathcal{B}^{(1)}, \dots , \mathcal{B}^{(n)}$ which satisfy the property \eqref{avereunaphisola}. Then $ \chi _{js}$ is smooth for every $s=1,\dots , n$ except for at most one $\bar s$.

We just consider the case $b_ {j1 }^{(\bar s)}\ne 0$. Fix $j=2,\dots , m$ and a point $(t_0,\hat x_j,y_0,\hat y_{\bar s}) \in  \tl {\mathcal O _j}$. Let $\delta _0>0$ such that
\begin{equation*}
\mathcal{R}=[t_0-\delta _0, t_0+\delta _0]\times \{ \hat x_j \} \times [y_{0}-\delta _0, y_{0}+\delta _0]\times  \{ \hat y_{\bar s} \} \subset  \tl {\mathcal O _j}. 
\end{equation*}
 Because $\mathcal O _j$ is open set and $ [y_{0}-\delta _0, y_{0}+\delta _0] \subset \tl {\mathcal O}_{j (t_0, \hat x_j, \hat y_{\bar s} )} $, it follows that there is $\sigma_0=\sigma(t_0)>0$ such that
\begin{equation}
\left[  \chi  _{j\bar s} (t_0,\hat x_j,y_0,\hat y_{\bar s}) ,\chi  _{j\bar s} (t_0,\hat x_j,y_0,\hat y_{\bar s})+\sigma _0\right] \subset \mathcal O _{j (t_0, \hat x_j, \hat y_{\bar s})}.
\end{equation}
Moreover, by definition of full Lagrangian type  parametrization, $\chi_{j \bar s} (t_0,\hat x_j,\cdot ,\hat y_{\bar s}) $ is monotone and surjective; and so
\begin{equation}\label{3.6bcsc}
\exists \lim_{y\to y_0^{-}}  \chi  _{j\bar s} (t_0,\hat x_j,y,\hat y_{\bar s}) = \chi  _{j\bar s} (t_0,\hat x_j,y_0,\hat y_{\bar s}).
\end{equation}
Indeed, if this is not true, one of the inequalities below would be strict
\begin{equation*}
\lim_{y\to y_0^{-}}  \chi  _{j\bar s} (t_0,\hat x_j,y,\hat y_{\bar s}) \leq  \chi  _{j\bar s} (t_0,\hat x_j,y_0,\hat y_{\bar s})  \leq \lim_{y\to y_0^{+}}  \chi  _{j\bar s} (t_0,\hat x_j,y,\hat y_{\bar s}) 
\end{equation*}
and we get a contradiction given by the surjectivity. On the other hand, since the map $[t_0-\delta _0, t_0+\delta _0] \ni t \mapsto  \chi  _{j\bar s} (t,\hat x_j,y_0 \pm \delta _0,\hat y_{\bar s}) $
is continuous and using again the monotony of $\chi_{j \bar s}$, it follows that the map $\Upsilon_j : \mathcal{R} \to  \mathcal O$ is bounded. Then by $(L.3)$ of Definition \ref{defiA}, we get that $ \chi  _{j\bar s} (\cdot ,\hat x_j,y,\hat y_{\bar s}) :  [t_0-\delta _0, t_0+\delta _0] \to \R$  is Lipschitz continuous uniformly with respect to $ y\in [y_{0}-\delta _0, y_{0}+\delta _0]$ and, consequently, using also \eqref{3.6bcsc}, $ \chi  _{j\bar s}$ is continuous at point $(t_0,\hat x_j,y_0,\hat y_{\bar s}) ,$ as desired.

Notice that we do not prove anything about the global continuity of $\chi_{j \bar s}$.
\end{proof}

In the previous section, we give an explicit example of the partial Lagrangian type parameterization. Following \cite{biblio17}, and, consequently, using the specific form of our intrinsic derivatives we can extended the partial Lagrangian type parameterization to a full one.
\begin{prop}\label{AppendiceA.1}
Any partial Lagrangian type parameterization can be extended to a full one.
\end{prop}

\begin{proof}
The proof follows as in \cite{biblio17}, Appendix A.1  in the context of Heisenberg groups.
\end{proof}

The following proposition states one can reduce the PDE to ODEs along a selected family of characteristics constituting a Lagrangian type parameterization.

\begin{prop}\label{Lemma6.2bcsc}
Let $\phi : U\to \R$ be a continuous function where $U$ is given by Proposition $\ref{lemma4.2bcsc}$. For each $j=2,\dots , m$ and $\hat x_j \in [0,1]^{m-2}$ fixed, consider a Lagrangian type  parametrization $(\tl {\mathcal O _j},  \chi _{j1}(t,\hat x_j,y),\dots ,  \chi _{jn}(t,\hat x_j,y) )$ and assume that $[0,1] \ni t \mapsto \phi (t,\hat x_j, \chi _{j1}(t,\hat x_j,y),\dots ,  \chi _{jn}(t,\hat x_j,y))$ is Lipschitz continuous for all $y \in \tl {\mathcal O _j}_{(t,\hat x_j)} $. Then, for $s=1,\dots , n$ such that $ b^{(s)}_{j1} \ne 0$, there exists a Borel function $\bar w=(\bar w_2,\dots , \bar w_m) :U \to \R^{m-1}$ such that for all $y$
\begin{equation}\label{equationODE}
\begin{aligned}
 \frac{\partial \phi}{\partial t}(t,\hat x_j, \chi _{j1}(t,\hat x_j,y),\dots ,  \chi _{jn}(t,\hat x_j,y)) & = \frac{1}{ b^{(s)}_{j1} } \frac{\partial ^2 \chi _{js} }{\partial t^2}(t,\hat x_j,y) \\
 & = \bar w_j (t,\hat x_j, \chi _{j1}(t,\hat x_j,y),\dots ,  \chi _{jn}(t,\hat x_j,y)),
 \end{aligned}
\end{equation}
for $\mathcal{L}^1$-a.e. $t\in [0,1]$.
\end{prop}

\begin{proof}
We use the similar technique exploited in \cite[Lemma 6.2]{biblio17}  in the context of Heisenberg groups.

%Without loss of generality, we consider $\G = (\R^{m+n}, \cdot  , \delta_\lambda )$ a Carnot group of step $2$ where the group operation is given by \eqref{1} for $m\times m$ real matrices $\mathcal{B}^{(1)}, \dots , \mathcal{B}^{(n)}$ which satisfy the property \eqref{avereunaphisola}. Then $ \chi _{js}$ is smooth for every $s=1,\dots , n$ except for at most one $s$ and, consequently, we just consider the case $b_ {j1 }^{(s)}\ne 0$. 

Fix $j=2,\dots , m$, $(\hat x_j, \hat y_{ s}) \in \R^{m-2}  \times \R^{n-1}$ such that the components belong to $U$. With a bit abuse to notation we use $U$ and $\mathcal O _j$ for the same thing.

 In order to check $\Upsilon _j :\tl {\mathcal O _j} \to U$ lifts $ \frac{1}{ b^{(s)}_{j1} } \frac{\partial ^2 \chi _{js} }{\partial t^2}$ to a map $\bar w_j$ a.e. defined on $\mathcal O _j$, which would provide our thesis, we split the proof in several steps.

$\mathbf{Step \, 1.}$  Using Tonelli theorem, we prove that the subset $\mathcal A \subset \tl {\mathcal O _j}$ of those $(t,\hat x_j, y, \hat y_{ s})$ where $ \chi _{j\bar s}$ is twice $t$-differentiable is a full measure, i.e. 
\begin{equation}\label{equinsieme}
\mathcal A = \bigcap_{\varepsilon \to 0 } \bigcup_{r\in \QQ}  \bigcup_{\delta >0} \left\{ (t,y)\,:\, \frac{\phi (\overline{\mathcal U ((t,y), \delta)}) - \phi (t,y)}{\delta }  \subset\overline{ \mathcal U(r,\varepsilon )} \right\} .
\end{equation}

 Indeed each $y-$section has full measure by the Lipschitz continuity of $$b^{(s)}_{j1}  \phi (t,\hat x_j, \chi _{j1}(t,\hat x_j,y),\dots ,  \chi _{jn}(t,\hat x_j,y))+ \frac{1}{2}\sum_{l=2}^{m}b^{(s)}_{jl}x_l.$$ Moreover, using the continuity of $\phi$, we have that $\mathcal A$ is $F_{\sigma \delta}$ set, i.e. it is a countable union of closed sets equal to \eqref{equinsieme}, as desired.

$\mathbf{Step \, 2.}$  It is relevant to notice that $ \frac{1}{ b^{(s)}_{j1} } \frac{\partial ^2 \chi _{j s} }{\partial t^2} (t,\hat x_j, y, \hat y_{ s})$ is a Borel function on its domain $\tl {\mathcal O _j}$.

Indeed, this follows from the simply fact that $$\frac{\partial  \chi _{j s} }{\partial t} (t,\hat x_j, y) = b^{(s)}_{j1}  \phi (t,\hat x_j, \chi _{j1}(t,\hat x_j,y),\dots ,  \chi _{jn}(t,\hat x_j,y))+ \frac{1}{2}\sum_{l=2}^{m}b^{(s)}_{jl}x_l$$ is Lipschitz continuous in $t$. %, then $ \frac{1}{ b^{(s)}_{j1}  } \frac{\partial ^2 \chi _{js} }{\partial t^2} (t,y)$ is a Borel function on its domain, as desired.

$\mathbf{Step \, 3.}$ Now we analyze $\Upsilon_j (\mathcal A)$ and partial inverse of it. 

We partition $\tl {\mathcal O _j}$ into the level sets of $\Upsilon_j $, which are $G_\delta$ sets (i.e. level sets of $\Upsilon_j $ are countable intersection of open sets). %pag. 58 libro di srivaslava 
Indeed, recalling that a closed subset of a metrizable space is a $G_\delta $ set (see  \cite{Librocheservefinale}, Exercise 2.1.16), it is easy to see that, if $(t_0,\hat x_j, y_0, \hat y_{ s}) \in \Upsilon_j (\tl {\mathcal O _j})$, then 
\begin{equation*}
\Upsilon_j^{-1} (t_0,\hat x_j, y_0, \hat y_{ s})=\{(t_0,\hat x_j)\} \times  I_0 \times \{\hat y_{ s} \},
\end{equation*}
where $I_0 \subset \R$ is a closed bounded interval. %We recall that the complement of a  $G_\delta$ set is a countable union of closed sets. 

Moreover, in order to prove the Borel measurability of the partition $\{ \Upsilon_j^{-1} (t_0,\hat x_j, y_0, \hat y_{ s}) \,:\, (t_0,\hat x_j, y_0, \hat y_{ s})\in \Upsilon_j (\tl {\mathcal O _j}) \}$, it is sufficient to show that $\Upsilon_j^{-1} (\Upsilon_j (\mathcal U))$ is a Borel set for each open set $\mathcal U \subset \R^{m+n}$ (see Section 5.1 in \cite{Librocheservefinale}). 

For simplicity, thanks to Proposition \ref{AppendiceA.1}, we can consider the case when $\chi _{j s}$ is already a full parameterization. Here, by Proposition \ref{propcontinuita}, $\chi _{j s}$ is continuous.

Because every open set $\mathcal U$ is $\sigma$-compact we know that, by continuity, $\Upsilon_j (\mathcal U)$ is $\sigma$-compact, and so $\Upsilon_j^{-1} (\Upsilon_j (\mathcal U))$  is also $\sigma$-compact.
Therefore 
%by Srivastava's theorem there is a Borel cross section S for the partition: defining the function $\Lambda ^{-1}$ as the restriction of ? to S ? B, then ??1 is Borel, injective and onto ? (B).
by \cite[Theorem 5.9.2]{Librocheservefinale}, we have that there is a Borel restriction which is one-to-one with image $ \Upsilon_j (\mathcal A)$, i.e., there is a Borel set $\mathcal S$ and a Borel injective map $\Lambda _j$
\begin{equation*}
\Lambda _j^{-1} : \mathcal S \subset \tl {\mathcal O _j} \to  \Upsilon_j (\mathcal A) \subset \mathcal O _j,
\end{equation*}
such that
\begin{enumerate}
\item it holds
\begin{equation*}
 \Lambda ^{-1}_j  (t,\hat x_j, y, \hat y_{ s})=  \Upsilon_j(t,\hat x_j, y, \hat y_{ s}), \,\, \mbox{ for all } (t,\hat x_j, y, \hat y_{ s}) \in  \mathcal S.
  \end{equation*}
\item $\mbox{Im}(\Lambda ^{-1}_j ) = \Upsilon_j (\mathcal A).$
\end{enumerate}

%i.e. $ \Lambda ^{-1}_j $ is Borel, injective and onto $\Upsilon_j (\mathcal A)$.
%This implies that $\Upsilon_j (\mathcal A)$ is a Borel image by a one-to-one map. 

Hence, Theorem 4.12.4 in \cite{Librocheservefinale}, due to Lusin, says that $\Upsilon_j (\mathcal A)$ is Borel  and that this restriction $\Lambda ^{-1}_j $ has a Borel inverse
\begin{equation*}
\Lambda _j :  \Upsilon_j (\tl {\mathcal O _j} ) \subset \mathcal O _j \to   \tl {\mathcal O _j}. 
\end{equation*}

$\mathbf{Step \, 4.}$ Finally, we can define $\bar w_j$ as
$$
\bar w_j  (t,\hat x_j, y, \hat y_{ s}) =\left\{
\begin{array}{l}
 \frac{1}{ b^{(s)}_{j1} } \frac{\partial ^2 \chi _{j s} }{\partial t^2} ( \Lambda _j(t,\hat x_j, y, \hat y_{ s})) , \qquad  \,\,\, \mbox{ if }  (t,\hat x_j, y, \hat y_{ s}) \in  \Upsilon_j (\mathcal A),\\
\\
0, \qquad \qquad \qquad \qquad \qquad  \qquad \,\,\mbox{ if }  (t,\hat x_j, y, \hat y_{ s})\in \mathcal O _j  -  \Upsilon_j (\mathcal A).\\
\end{array}
\right.\\
$$

Arguing as in Step 5 of Lemma 6.2 in  \cite{biblio17},  %and using the fact $(t,\hat x_j, y', \hat y_{ s})= \Lambda _j (\Upsilon_j(t,\hat x_j, y, \hat y_{ s})) $, 
we have that $ \frac{1}{ b^{(s)}_{j1} } \frac{\partial ^2 \chi _{js} }{\partial t^2} \circ \Upsilon_j ^{-1}$ is multivalued at most countable set and consequently $\bar w_j$ is well defined. The proof of this proposition is complete.
\end{proof}

Now we are able to show the proof of Theorem \ref{teoremaFinale1}.
\begin{proof}[Proof of Theorem $\ref{teoremaFinale1}$]
$\mathbf{(1)  \iff (2)}$ follows from Theorem $\ref{lemma5.4bcsc}.$

\medskip

$\mathbf{(2)  \implies  (3)}$

For each $j=2,\dots ,m$ and $\hat x _j \in (0,1)^{m-2}$ fixed, by Lemma  \ref{lemma4.2bcsc}  there exists a full Lagrangian  type parameterization associated to $\phi :U\to \R$.
Moreover from Lemma \ref{lemma5.1bcsc}, it follows that $\phi$ is Lipschitz continuous along characteristics and consequently by Proposition \ref{Lemma6.2bcsc} properties $(LS1)$ and $(LS2)$ of Definition \ref{defiLagrangiana} immediately hold.

Finally, in order to prove  $(LS3)$ of Definition \ref{defiLagrangiana}, we show the identification of $w$ and $\bar w$. By Corollary \ref{coro5.7bcsc}
\begin{equation}\label{6.1}
d\phi_{i(a)} (e_j)= \left(  D_j^{\phi} \phi (a) ,0,\dots ,0\right)=(w_j (a),0,\dots , 0), \quad \mbox{a.e.} \,\, a\in U,
\end{equation}
 where $e_j=(0,\dots , 0,1,0,\dots , 0)$ and $j$-th element is $1$. 

On the other hand, if
\begin{equation}\label{equationODECONBARW}
\begin{aligned}
 \frac{\partial \phi}{\partial t}(t,\hat x_j, \chi _{j1}(t,\hat x_j,y),\dots ,  \chi _{jn}(t,\hat x_j,y)) &  = \bar w _j(t,\hat x_j, \chi _{j1}(t,\hat x_j,y),\dots ,  \chi _{jn}(t,\hat x_j,y)),
 \end{aligned}
\end{equation}
for a.e. $t\in [0,1]$, for all $y \in \tl {\mathcal O _j}_{(t,\hat x_j)}$ and $\chi _{js}$ are onto, using again Lemma \ref{lemma5.1bcsc}, we have
\begin{equation}\label{6.2}
d\phi_{i(a)} (e_j)= \left(  D_j^{\phi} \phi (a) ,0,\dots ,0\right)=( \bar w_j (a),0,\dots , 0), \quad \mbox{a.e.} \,\, a\in U ,
\end{equation}
As a consequence, from \eqref{6.1} and \eqref{6.2}, $(LS3)$ of Definition \ref{defiLagrangiana} is true and the proof of $\mathbf{(2)  \implies  (3)}$ is complete.

\medskip

$\mathbf{(3)  \implies  (2)}.$ 

We refine the technique used in \cite[Theorem 6.10]{biblio17} in the context of Heisenberg groups. The main differences w.r.t. Heisenberg case are  Step 3 and 4. Here we use  on a mollification procedure in the Lagrangian variables; we begin regularizing the map $\chi _{j\bar s}$ and then using it we prove the smoothing of $ \Upsilon_j  $ and so of $\phi$  in the $t$-variable. Finally, using again the Friedrichs' mollifier, we regularize the datum $w.$

 By assumption, $ \chi _{js}$ is smooth for every $s=1,\dots , n$ except for at most one $\bar s.$ Hence, we just consider the case $b_ {j1 }^{(\bar s)}\ne 0$ and, for simplicity, we write $s$ instead of $\bar s.$

Fix $j=2,\dots , m$, $\hat x_j\in \R^{m-2}$ and $\hat y_{ s} \in \R^{n-1}$. Let $\Upsilon_j : \tl {\mathcal O _j} \to \mathcal O$ be a full Lagrangian  type parametrization of  $\phi :  \mathcal O \to \R$. This is possible thanks to Section \ref{Existence Lagrangian type parameterization} and Proposition \ref{AppendiceA.1}.

As this is a local argument, one can assume in order to avoid technicalities that $ \chi _{js}$ is constant out of a compact: just modify $ \chi _{js}$, and consequently $\phi$ and $\bar w$, out of an open ball where one wants to prove the statement. By a partition of unity the statement will hold then on the desired domain. By a change of variables one can as well assume that $ \chi _{js}$ is valued in $(0,1)$ and that the support of $\phi$ is compactly contained in $(0,1)^2$. In particular w.l.o.g. we can assume that $\Upsilon_j : \tl {\mathcal O _j}:=(0,1)\times (\tl t_1, \tl t_2) \to \mathcal O :=(0,1)\times (t_1, t_2)$  with $-\infty < \tl t_1<\tl t_2 < +\infty $, $0 < t_1 < t_2 < 1$ and that $ \chi _{js}$ can be meant as a continuous function $(0,1)\times \R \to (0,1)$  by defining it $ \chi _{js}(t,\hat x_j,y)=t_1$  for each $t\in (0,1)$ and $y_s\leq \tl t_1$, and $ \chi _{js}(t,\hat x_j,y)=t_2$  for each $t\in (0,1)$ and $y_s\geq \tl t_2$.

%For simplicity, with a little abuse of notation, we omit $\hat x_j$ and $\hat y_s.$

We split the proof in several steps:
\begin{enumerate}
\item Smoothing of $\chi _{js (\hat x_j, \hat y_s)};$
\item Existence  of the inverse map of  $ \Upsilon_j ;$
\item Smoothing of $\phi$ in the $t$-variable;
\item Approximation of the datum $w$;
\item Limiting argument.
\end{enumerate}

$\mathbf{Step \, 1.}$ As we said, we use Friedrichs' mollifier to regularize $ \chi _{js}$ in $y_s$-variable.

Consider a suitable convolution kernel $\rho _\varepsilon (y_s)$ compactly supported in $\{|y_s| < \varepsilon\}$ and define the $y_s$-regularized function 
\[
\chi _{js (\hat x_j, \hat y_s)}^\varepsilon =  \chi _{js}^\varepsilon : (0,1)\times \R \to \R,
\]
 given by
 \begin{equation*}
\chi _{js}^\varepsilon (t,y_s)=(1+\varepsilon y_s) \left( \chi _{js}(t,\cdot ) \ast  \rho _\varepsilon \right) (y_s)=(1+\varepsilon y_s) \int _\R \chi _{js}(t,\xi )  \rho _\varepsilon (y_s-\xi) \, d\xi .\\
\end{equation*}
Arguing as in Step 1 of \cite[Theorem 6.10]{biblio17}, it is possible to show that
\begin{itemize}
\item The function $\chi _{js}^\varepsilon$ is locally Lipschitz continuous;
\item for every $t \in (0,1), \chi _{js}^\varepsilon  (t, \cdot ):\R \to \R$ is strictly increasing and smooth 
\item for $\varepsilon $ small enough and each $(t,y_s) \in (0,1)\times \R ,$ it holds
\begin{equation*}
  \frac{\partial \chi _{js}^\varepsilon }{\partial y_s}  (t,y_s) = \varepsilon ( \chi _{js}(t,\cdot ) \ast \rho _\varepsilon ) (y_s) +(1+\varepsilon y_s) \left(  \frac{\partial \chi _{js}^\varepsilon }{\partial y_s}  (t,\cdot ) \ast \rho_\varepsilon  \right) (y_s)>0,
\end{equation*}
where $\frac{\partial \chi _{js}^\varepsilon }{\partial y_s}  (t,\cdot )  \ast \rho_\varepsilon $ denotes the convolution between the finite non-negative measure $\frac{\partial \chi _{js}^\varepsilon }{\partial y_s}  (t,\cdot )$ and the kernel $ \rho_\varepsilon $ 
\item $\chi _{js}^\varepsilon  $ converges locally uniformly to $\chi _{js}$ for $\varepsilon  \to 0.$ 
\end{itemize}

$\mathbf{Step \, 2.}$ Let $\Upsilon_{j,\hat x _j,\hat y_s, \varepsilon}=  \Upsilon_{j, \varepsilon} :(0,1) \times \R \to (0,1) \times \R^n $ defined as
\begin{equation*}
 \Upsilon_{j, \varepsilon} (t,y_s):= (t,\hat x_j, \chi _{j1} (t,y) ,\dots, \chi _{js}^\varepsilon (t,y_s), \dots,  \chi _{jn} (t,y)).
\end{equation*}
It is trivial the fact that $ \Upsilon_{j, \varepsilon}$ is a smooth function, if we do not consider the factor  $ \chi _{js}^\varepsilon .$ We put $ \chi _{j\ell}^{-1} (t,y)$ the inverse map of $\chi _{j\ell} (t,y) $ equal to \eqref{Gammaovvio} for $\ell =1,\dots, n$ with $\ell \ne s.$ Hence, the only problem is  the factor $ \chi _{js}^\varepsilon .$ By Proposition \ref{propcontinuita}, the map \begin{equation*}
F_{j, \varepsilon} (t,y_s):= (t, \chi _{js}^\varepsilon (t,y_s)),
\end{equation*}
 is continuous map. Moreover, by definition and Step 1, it is 1-to-1 and onto and, consequently, it admits an inverse map $F_{j, \varepsilon} ^{-1}:(0,1) \times \R \to (0,1)\times \R $ which is continuous, injective and onto between two open sets of $\R^2.$ We call $ \Upsilon_{j, \varepsilon} ^{-1}$ the map given by identity in vertical components and, for the last vertical components, the inverse component of each $ \chi _{j\ell}$  for $\ell =1,\dots, n$ with $\ell \ne s$ and $F_{j, \varepsilon} ^{-1}$ for $\ell =s$.

$\mathbf{Step \, 3.}$ We can define the approximation $ \phi^\varepsilon_{(\hat x_j, \hat y_s)} = \phi^\varepsilon :(0,1)\times \R \to \R$ as the map satisfying the following equality
\begin{equation*}
b^{(s)}_{j1}  \phi ^\varepsilon ( t,y_s) = \psi ^\varepsilon ( \Upsilon_{j, \varepsilon} ^{-1} (t,y_s)) -  \frac{1}{2}\sum_{l=2}^{m}b^{(s)}_{jl}x_l,  
%\phi^\varepsilon (t, \chi _{js}^\varepsilon (t,y_s))
\end{equation*}
where $\psi^\varepsilon :(0,1)\times \R \to \R $ is given by $$\psi^\varepsilon (t,y_s):= (1+\varepsilon y_s) \left( \left( b^{(s)}_{j1}  \phi (t,\hat x_j,  \chi _{j1}(t,\hat x_j,y), \dots, \cdot, \dots, \chi _{jn}(t,\hat x_j,y) )  \right) \ast  \rho _\varepsilon \right) (y_s),$$ where $\cdot$ is in s-position. The key observation is that, recall that $$\frac{\partial  \chi _{j s} }{\partial t} (t, \hat x_j,y) = b^{(s)}_{j1}  \phi (t, \chi _{j1}(t,\hat x_j,y), \dots, \chi _{jn}(t,\hat x_j,y))+ \frac{1}{2}\sum_{l=2}^{m}b^{(s)}_{jl}x_l,$$  we have that
\begin{equation}\label{equationservira}
\begin{aligned}
b^{(s)}_{j1}  \phi ^\varepsilon ( t, \chi _{js}^\varepsilon (t,y_s))  & =  \psi^\varepsilon (t,y_s) -  \frac{1}{2}\sum_{l=2}^{m}b^{(s)}_{jl}x_l\\
& = (1+\varepsilon y_s) \left(  \frac{\partial \chi _{js}}{\partial t}  (t, \hat x_j , \hat y_s, \cdot ) \ast  \rho _\varepsilon \right) (y_s) -  \frac{1}{2}\sum_{l=2}^{m}b^{(s)}_{jl}x_l\\
& =   \frac{\partial \chi _{js}^\varepsilon }{\partial t}  (t,\hat x _j ,y) -  \frac{1}{2}\sum_{l=2}^{m}b^{(s)}_{jl}x_l.%- \frac{1}{2}\sum_{l=2}^{m}b^{(s)}_{jl}x_l.
\end{aligned}
\end{equation}

Arguing as in the last two claim in Step 1 of \cite[Theorem 6.10]{biblio17}, $\phi ^\varepsilon$ is a locally Lipschitz smooth function which converges in $L^1$ to $\phi$ as $\varepsilon \to 0 .$ Here the only difference is that $\phi ^\varepsilon$ and $\phi $ are compactly supported in $(0,1)\times (0, B)$ for some $B>0$ instead of $(0,1)^2.$

$\mathbf{Step \, 4.}$ We define an approximate datum $w^\varepsilon _j: \mathcal O \to \R$ using the  approximation $\phi ^\varepsilon$ in the previous step:

\begin{equation}\label{equationODE.22}
\begin{aligned}
\frac{\partial \phi^\varepsilon }{\partial t}(t,x,y)+  b^{(s)}_{j1}  \frac{\partial }{\partial y_s} \frac{( \phi^\varepsilon )^2}{2}(t,x,y)+  \frac{1}{2} \sum_{i=2}^{m}  \frac{\partial  \phi^\varepsilon }{\partial y_i} (t,x,y) \sum_{l=2}^{m}b^{(i)}_{jl}x_l = w^\varepsilon_j (t,x,y), \\
 \end{aligned}
\end{equation}
%\begin{equation}\label{equationODE.22}
%\begin{aligned}
%\frac{\partial \phi^\varepsilon }{\partial t}(t,\hat x_j, \hat y_s, y_s)+  b^{(s)}_{j1}  \frac{\partial }{\partial y_s} \frac{( \phi^\varepsilon )^2}{2}(t,\hat x_j, \hat y_s,y_s)+  \frac{1}{2} \sum_{i=2}^{m}  \frac{\partial  \phi^\varepsilon }{\partial y_i} (t,\hat x_j, \hat y_s,y_s) \sum_{l=2}^{m}b^{(i)}_{jl}x_l = w^\varepsilon_j (t, \hat x_j, \hat y_s,y_s), \\
% \end{aligned}
%\end{equation}
almost everywhere in $\mathcal O,$ i.e., $ D^{\phi^{\varepsilon }}_j\phi ^{\varepsilon }=  w^\varepsilon_j.$  We can assume that $w^\varepsilon $ is a Borel map, because $\phi ^\varepsilon $ is a Lipschitz map. Moreover, using the smoothness of $\phi^\varepsilon $ and \eqref{equationservira}, the equality  \eqref{equationODE.22} is equivalent to
\begin{equation*}\label{equationODE.8}
\begin{aligned}
 w^\varepsilon_j (t,\hat x_j, \chi _{j1}(t,\hat x_j,y),\dots ,  \chi _{jn}(t,\hat x_j,y)) & = \frac{1}{ b^{(s)}_{j1} } \frac{\partial ^2 \chi _{js}^\varepsilon }{\partial t^2}(t,\hat x_j,y) \\
& = \frac{\partial \phi ^\varepsilon }{\partial t}(t,\hat x_j, \chi _{j1}(t,\hat x_j,y) ,\dots, \chi _{js}^\varepsilon (t,y_s), \dots,   \chi _{jn}(t,\hat x_j,y)). \\
 \end{aligned}
\end{equation*}
By Proposition \ref{Lemma6.2bcsc}, we have that \eqref{equationODE} holds with the datum $\bar w$ and so
\begin{equation*}\label{equationODE.408}
\begin{aligned}
 w^\varepsilon_j & (t,\hat x_j, \chi _{j1}(t,\hat x_j,y) ,\dots, \chi _{js}^\varepsilon (t,y_s), \dots,   \chi _{jn}(t,\hat x_j,y)) = \frac{1}{ b^{(s)}_{j1} } \frac{\partial ^2 \chi _{js}^\varepsilon }{\partial t^2}(t,\hat x_j,y) \\
& = (1+\varepsilon y_s) \left(  \frac{\partial ^2 \chi _{js}}{\partial t^2}  (t, \cdot ) \ast  \rho _\varepsilon \right) (y_s)\\
& =  (1+\varepsilon y_s) \left(  \bar w_j (t,\hat x_j, \chi _{j1}(t,\hat x_j,y)  ,\dots, \chi _{js}^\varepsilon (t,y_s), \dots,  \chi _{jn}(t,\hat x_j,y)) \ast  \rho _\varepsilon \right) (y_s).\\
 \end{aligned}
\end{equation*}
In particular, the datum $ w^\varepsilon_j$ are uniformly bounded by $(1 + \varepsilon)$ times the uniform bound for $\bar w$. Moreover, for each $t$ fixed $ w^\varepsilon_j (t,\hat x_j, \chi _{j1}(t,\hat x_j,y)  ,\dots, \chi _{js}^\varepsilon (t,y_s), \dots,   \chi _{jn}(t,\hat x_j,y)) $ converges in all $L^p(dy_s)$ to $\bar w_j (t,\hat x_j, \chi _{j1}(t,\hat x_j,y) ,\dots, \chi _{js}(t,y_s), \dots,   \chi _{jn}(t,\hat x_j,y)) $, and thus in $L^p(dtdy_s)$; Notice that if the datum $\bar w$ is continuous, then the convergence is clearly uniform.

$\mathbf{Step \, 5.}$  Arguing as in  Step 4 of \cite[Theorem 6.10]{biblio17}, we have that there is a bounded map $\tilde w$ such that $\phi $ is a distributional solution of the equation
\begin{equation*}
\begin{aligned}
\frac{\partial \phi }{\partial t}(t,x,y)+  b^{(s)}_{j1}  \frac{\partial }{\partial y_s} \frac{\phi ^2}{2}(t,x,y)+  \frac{1}{2} \sum_{i=2}^{m}  \frac{\partial  \phi }{\partial y_i} (t,x,y) \sum_{l=2}^{m}b^{(i)}_{jl}x_l = \tilde w_j (t,x,y),\\
 \end{aligned}
\end{equation*} on $U.$ 
Finally, we prove that $\tilde w = \bar w$ a.e.  on $U.$
Indeed, by Theorem \ref{Theorem 4.3.5fms}, the function $\tilde w$ is equal to the intrinsic gradient of $\phi$ almost everywhere. On the other hand, using the definition of intrinsic gradient of $\phi$ and Lemma \ref{lemma5.1bcsc}, the intrinsic gradient of $\phi$ coincides to the datum $\bar w$ almost everywhere.

This completes the proof of $\mathbf{(3)  \implies  (2)}$ and, consequently, the theorem is true.
\end{proof}

\bibliographystyle{plain}
\bibliography{BCSCBib}

\begin{thebibliography}{10}

\bibitem{biblioAMBkirc}
Luigi Ambrosio and Bernd Kirchheim.
\newblock Rectifiable sets in metric and {B}anach spaces.
\newblock {\em Math. Ann.}, 318(3):527--555, 2000.

\bibitem{biblioAKLD}
Luigi Ambrosio, Bruce Kleiner, and Enrico Le~Donne.
\newblock Rectifiability of sets of finite perimeter in {C}arnot groups:
  existence of a tangent hyperplane.
\newblock {\em J. Geom. Anal.}, 19(3):509--540, 2009.

\bibitem{biblio1}
Luigi Ambrosio, Francesco Serra~Cassano, and Davide Vittone.
\newblock Intrinsic regular hypersurfaces in {H}eisenberg groups.
\newblock {\em J. Geom. Anal.}, 16(2):187--232, 2006.

\bibitem{biblioADDD}
Gioacchino Antonelli, Daniela Di~Donato, and Sebastiano Don.
\newblock Distributional solutions of burgers' type equations for intrinsic
  graphs in carnot groups of step 2.
\newblock 2020.
\newblock Preprint.

\bibitem{biblioADDDLD}
Gioacchino Antonelli, Daniela Di~Donato, Sebastiano Don, and Enrico Le~Donne.
\newblock Characterizations of uniformly differentiable co-horizontal intrinsic
  graphs in {C}arnot groups.
\newblock 2020.
\newblock Preprint, available at \url{https://arxiv.org/abs/2005.11390}.

\bibitem{biblioALD}
Gioacchino Antonelli and Enrico Le~Donne.
\newblock Pauls rectifiable and purely {P}auls unrectifiable smooth
  hypersurfaces.
\newblock {\em Nonlinear Anal.}, 200:111983, 30, 2020.

\bibitem{biblioAM3}
Gioacchino Antonelli and Andrea Merlo.
\newblock On rectifiable measures in {C}arnot groups: structure theory.
\newblock {\em preprint}, 2020.

\bibitem{biblioAM2}
Gioacchino Antonelli and Andrea Merlo.
\newblock Intrinsically {L}ipschitz functions with normal targets in {C}arnot
  groups.
\newblock {\em Accepted to Ann. Acad. Sci. Fenn. Math.}, 2021.

\bibitem{biblio2}
Gabriella Arena and Raul Serapioni.
\newblock Intrinsic regular submanifolds in {H}eisenberg groups are
  differentiable graphs.
\newblock {\em Calc. Var. Partial Differential Equations}, 35(4):517--536,
  2009.

\bibitem{biblio17}
F.~Bigolin, L.~Caravenna, and F.~Serra~Cassano.
\newblock Intrinsic {L}ipschitz graphs in {H}eisenberg groups and continuous
  solutions of a balance equation.
\newblock {\em Ann. Inst. H. Poincar\'{e} Anal. Non Lin\'{e}aire},
  32(5):925--963, 2015.

\bibitem{biblio27}
Francesco Bigolin and Francesco Serra~Cassano.
\newblock Intrinsic regular graphs in {H}eisenberg groups vs. weak solutions of
  non-linear first-order {PDE}s.
\newblock {\em Adv. Calc. Var.}, 3(1):69--97, 2010.

\bibitem{biblio3}
A.~Bonfiglioli, E.~Lanconelli, and F.~Uguzzoni.
\newblock {\em Stratified {L}ie groups and potential theory for their
  sub-{L}aplacians}.
\newblock Springer Monographs in Mathematics. Springer, Berlin, 2007.

\bibitem{biblio26}
Giovanna Citti, Maria Manfredini, Andrea Pinamonti, and Francesco
  Serra~Cassano.
\newblock Smooth approximation for the intrinsic {L}ipschitz functions in the
  {H}eisenberg group.
\newblock {\em Calc. Var. Partial Differ. Equ.}, 49:1279--1308, 2014.

\bibitem{biblioCorni}
Francesca Corni.
\newblock Intrinsic regular surfaces of low codimension in {H}eisenberg groups.
\newblock 2020.
\newblock Accepted paper on Ann. Acad. Sci. Fenn. Math.

\bibitem{biblioDDD2}
Daniela Di~Donato.
\newblock Intrinsic {L}ipschitz graphs in {C}arnot groups of step 2.
\newblock {\em Ann. Acad. Sci. Fenn. Math}, 45(2):1013--1063, 2020.

\bibitem{biblioDDD}
Daniela Di~Donato.
\newblock Intrinsic differentiability and intrinsic regular surfaces in
  {C}arnot groups.
\newblock {\em Potential Anal.}, 54(1):1--39, 2021.

\bibitem{biblioDDF}
Daniela Di~Donato and Katrin Fassler.
\newblock Extensions and corona decompositions of low-dimensional intrinsic
  {L}ipschitz graphs in {H}eisenberg groups.
\newblock {\em Accepted to Annali di Matematica Pura ed Applicata}, 2021.

\bibitem{biblioDDFO}
Daniela Di~Donato, Katrin Fassler, and Tuomas Orponen.
\newblock Metric rectifiability of {H}-regular surfaces with {H}older
  continuous horizontal normal.
\newblock {\em Accepted to International Mathematics Research Notices}, 2020.

\bibitem{biblioDLDMV}
Sebastiano Don, Enrico Le~Donne, Terhi Moisala, and Davide Vittone.
\newblock A rectifiability result for finite-perimeter sets in {C}arnot groups.
\newblock 2019.
\newblock Accepted to Indiana Univ. Math. J., available at
  \url{https://arxiv.org/abs/1912.00493}.

\bibitem{biblio5}
Herbert Federer.
\newblock Geometric measure theory.
\newblock pages xiv+676, 1969.

\bibitem{biblio21}
Bruno Franchi, Marco Marchi, and Raul Serapioni.
\newblock Differentiability and approximate differentiability for intrinsic
  lipschitz functions in carnot groups and a rademarcher theorem.
\newblock {\em Anal. Geom. Metr. Spaces}, 2(3):258--281, 2014.

\bibitem{biblio6}
Bruno Franchi, Raul Serapioni, and Francesco Serra~Cassano.
\newblock Rectifiability and perimeter in the {H}eisenberg group.
\newblock {\em Math. Ann.}, 321(3):479--531, 2001.

\bibitem{biblio8}
Bruno Franchi, Raul Serapioni, and Francesco Serra~Cassano.
\newblock On the structure of finite perimeter sets in step 2 {C}arnot groups.
\newblock {\em The Journal of Geometric Analysis}, 13(3):421--466, 2003.

\bibitem{biblio7}
Bruno Franchi, Raul Serapioni, and Francesco Serra~Cassano.
\newblock Regular hypersurfaces, {I}ntrinsic {P}erimeter and {I}mplicit
  {F}unction {T}heorem in {C}arnot groups.
\newblock {\em Comm. Anal. Geom.}, 11(5):909--944, 2003.

\bibitem{bibliofsscINTRINSICLIP}
Bruno Franchi, Raul Serapioni, and Francesco Serra~Cassano.
\newblock Intrinsic {L}ipschitz graphs in {H}eisenberg groups.
\newblock {\em J. Nonlinear Convex Anal.}, 7(3):423--441, 2006.

\bibitem{biblio24}
Bruno Franchi, Raul Serapioni, and Francesco Serra~Cassano.
\newblock Regular submanifolds, graphs and area formula in {H}eisenberg groups.
\newblock {\em Adv. Math.}, 211(1):152--203, 2007.

\bibitem{biblio22}
Bruno Franchi and Raul~Paolo Serapioni.
\newblock Intrinsic {L}ipschitz graphs within {C}arnot groups.
\newblock {\em J. Geom. Anal.}, 26(3):1946--1994, 2016.

\bibitem{biblioGARN}
N.~Garofalo and D.M. Nhieu.
\newblock Isoperimetric and {S}obolev inequalities for {C}arnot-carath\'eodory
  spaces and the existence of minimal surfaces.
\newblock {\em Comm. Pure Appl. Math.}, 49:1081--1144, 1996.

\bibitem{biblioOMM}
K.~O. Idu, V.~Magnani, and F.P. Maiale.
\newblock Characterizations of k-rectifiability in homogeneous groups.
\newblock {\em J. Math. Anal. Appl. 500}, 49(2):22E25 (28A75), 2021.

\bibitem{biblioJNGV}
Antoine Julia, Sebastiano Nicolussi~Golo, and Davide Vittone.
\newblock Nowhere differentiable intrinsic lipschitz graphs.
\newblock 2021.
\newblock Preprint.

\bibitem{biblio19}
K.~Kirchheim and F.~F.Serra~Cassano.
\newblock Rectifiability and parametrization of intrinsic regular surfaces in
  the {H}eisenberg group.
\newblock {\em Ann. Scuola Norm. Sup. Pisa Cl. Sci.}, 5:871--896, 2004.

\bibitem{biblioKOZHEVNIKOV}
Artem Kozhevnikov.
\newblock Propri\'et\'es m\'etriques des ensembles de niveau des applications
  diff\'erentiables sur les groupes de carnot.
\newblock {\em PhD Thesis, Universit\'e Paris Sud - Paris XI}, 2015.

\bibitem{biblioLeDonne}
Enrico Le~Donne.
\newblock A primer on {C}arnot groups: homogenous groups,
  {C}arnot-{C}arath\'{e}odory spaces, and regularity of their isometries.
\newblock {\em Anal. Geom. Metr. Spaces}, 5:116--137, 2017.

\bibitem{biblioMAGN1}
Valentino Magnani.
\newblock Unrectifiability and rigidity in stratified groups.
\newblock {\em Arch. Math.}, 83(6):568--576, 2004.

\bibitem{biblioMAGNANI}
Valentino Magnani.
\newblock Towards differential calculus in stratified groups.
\newblock {\em J. Aust. Math. Soc.}, 95(1):76--128, 2013.

\bibitem{biblioMitchell}
J.~Mitchell.
\newblock On carnot-carath\'eodory metrics.
\newblock {\em J. Differ. Geom.}, 21:35--45, 1985.

\bibitem{biblio11}
Pierre Pansu.
\newblock M\'{e}triques de {C}arnot-{C}arath\'{e}odory et quasiisom\'{e}tries
  des espaces sym\'{e}triques de rang un.
\newblock {\em Ann. of Math. (2)}, 129(1):1--60, 1989.

\bibitem{biblioPAUl}
Scott~D. Pauls.
\newblock A notion of rectifiability modeled on {C}arnot groups.
\newblock {\em Indiana Univ. Math. J.}, 53(1):49--81, 2004.

\bibitem{biblioRR}
H.M. Reimann and F.~Ricci.
\newblock The complexified {H}eisenberg group.
\newblock {\em Proceedings on Analysis and Geometry (Russian) (Novosibirsk
  Akademgorodok, 1999), 465-480, Izdat. Ross. Akad. Nauk Sib. Otd. Inst. Mat.,
  Novosibirsk}, 2000.

\bibitem{biblioLIBROSC}
Francesco Serra~Cassano.
\newblock Some topics of geometric measure theory in {C}arnot groups.
\newblock In {\em Geometry, analysis and dynamics on sub-{R}iemannian
  manifolds. {V}ol. 1}, EMS Ser. Lect. Math., pages 1--121. Eur. Math. Soc.,
  Z\"urich, 2016.

\bibitem{Librocheservefinale}
S.M. Srivastava.
\newblock In {\em A {C}ourse on {B}orel {S}ets}. Grad. Texts Math., vol. 180,
  Springer, 1998.

\bibitem{biblioVIT2}
Davide Vittone.
\newblock Lipschitz surfaces, perimeter and trace theorems for {BV} functions
  in {C}arnot-{C}arath\'{e}odory spaces.
\newblock {\em Ann. Sc. Norm. Super. Pisa Cl. Sci. (5)}, 11(4):939--998, 2012.

\bibitem{biblioVIT1}
Davide Vittone.
\newblock Lipschitz graphs and currents in {H}eisenberg groups.
\newblock 2020.
\newblock Preprint, available at \url{https://arxiv.org/abs/2007.14286}.

\end{thebibliography}

\end{document}